\documentclass[british, a4paper,11pt]{article}

\usepackage[utf8]{inputenc}
\usepackage[british]{babel}

\usepackage{amsthm}

\usepackage[mono=false]{libertine}
\usepackage[T1]{fontenc}
\usepackage[cmintegrals,libertine]{newtxmath}
\usepackage[cal=euler, calscaled=.95, scr=boondoxo]{mathalfa}

\useosf

\usepackage{microtype}

\usepackage{numprint}

\usepackage[explicit]{titlesec}

\titleformat{\section}[block]{\normalfont\large\bfseries\boldmath\centering}{\raggedright\makebox[1em][l]{\thesection.}}{.25em}{#1}
\titleformat{\subsection}[block]{\normalfont\bfseries\boldmath\centering}{\raggedright\makebox[1em][l]{\thesubsection.}}{1em}{#1}
\titleformat{\subsubsection}[runin]{\normalfont\bfseries\boldmath}{\raggedright\makebox[1em][l]{\thesubsubsection.}}{1.5em}{#1.~\hbox{---}}

\renewenvironment{abstract}{%
\begin{center}
\begin{minipage}{.9\textwidth}\linespread{1.05}\selectfont\small
\makebox[5em][l]{\bfseries\abstractname.~\hbox{---}}
\normalfont}
{\par\vspace{1em}
\end{minipage}
\end{center}
}

\usepackage[a4paper,margin=2.5cm]{geometry}

\usepackage[font=sf]{caption}
\usepackage{xcolor}

\usepackage{hyperref}
\hypersetup{
	colorlinks=true,
		citecolor=blue!60!black,
		linkcolor=red!60!black,
		urlcolor=green!40!black,
		filecolor=yellow!50!black,
	breaklinks=true,
	pdfpagemode=UseNone,
	bookmarksopen=false,
}

\usepackage{enumitem}

\numberwithin{equation}{section}

\newtheorem{thm}{\bfseries \upshape Theorem}[section]
\newtheorem{lem}[thm]{Lemma}
\newtheorem{prop}[thm]{Proposition}
\newtheorem{cor}[thm]{Corollary}

\theoremstyle{definition}
\newtheorem{rem}[thm]{Remark}

\newtheorem{thmA}{Theorem}

\newcommand{\ensemble}[1]{\mathbb{#1}}
\renewcommand{\P}{\ensemble{P}}
\renewcommand{\Pr}[1]{\P\left(#1\right)}

\newcommand{\E}{\ensemble{E}}
\newcommand{\Es}[1]{\ensemble{E}\left[#1\right]}
\newcommand{\Esc}[2]{\E\left[#1 \;\middle|\; #2\right]}
\newcommand{\Var}{\operatorname{Var}}
\newcommand{\Cov}{\operatorname{Cov}}

\newcommand{\R}{\ensemble{R}}
\newcommand{\Z}{\ensemble{Z}}
\newcommand{\N}{\ensemble{N}}
\newcommand{\D}{\ensemble{D}}

\newcommand{\dgr}{d_{\mathrm{gr}}}

\newcommand{\Msf}{\mathsf{M}}
\newcommand{\Lsf}{\mathsf{LT}}
\newcommand{\Zsf}{\mathsf{Z}}
\newcommand{\Vsf}{\mathsf{V}}
\newcommand{\Jsf}{\mathsf{J}}
\newcommand{\Csf}{\mathsf{C}}
\newcommand{\Xsf}{\mathsf{X}}
\newcommand{\Rsf}{\mathsf{R}}

\newcommand{\Loop}{\mathrm{Loop}}
\newcommand{\Leb}{\mathrm{Leb}}

\newcommand{\e}{\operatorname{e}} 
\newcommand{\q}{\mathbf{q}}
\renewcommand{\d}{\mathop{}\!\mathrm{d}}
\renewcommand{\i}{\operatorname{i}}

\newcommand{\degf}{\mathrm{f}}
\newcommand{\Degf}{\mathbf{f}}

\newcommand{\bdtheta}{\boldsymbol{\theta}}

\renewcommand{\ge}{\geqslant}
\renewcommand{\le}{\leqslant}

\newcommand{\ind}[1]{\mathbf{1}_{\{#1\}}}

\newcommand{\cv}[1][n]{\enskip\mathop{\longrightarrow}^{}_{#1 \to \infty}\enskip}
\newcommand{\cvloi}[1][n]{\enskip\mathop{\longrightarrow}^{(d)}_{#1 \to \infty}\enskip}
\newcommand{\cvps}[1][n]{\enskip\mathop{\longrightarrow}^{a.s.}_{#1 \to \infty}\enskip}
\newcommand{\cvproba}[1][n]{\enskip\mathop{\longrightarrow}^{\P}_{#1 \to \infty}\enskip}
\newcommand{\eqloi}[1][n]{\enskip\mathop{=}^{(d)}_{}\enskip}
\newcommand{\eqps}[1][n]{\enskip\mathop{=}^{a.s.}_{}\enskip}

\newcommand{\tildep}[1]{\widetilde{#1}\vphantom{#1}}

\newcommand{\overp}[1]{\overline{#1}\vphantom{#1}}

\usepackage{mathtools}
\DeclarePairedDelimiter\ceil{\lceil}{\rceil}
\DeclarePairedDelimiter\floor{\lfloor}{\rfloor}

\author{
	Cyril \textsc{Marzouk}\thanks{CMAP, \'{E}cole polytechnique.\hfill  \href{mailto:cyril.marzouk@polytechnique.edu}{\texttt{cyril.marzouk@polytechnique.edu}}}
}

\title{Scaling limits of random looptrees and bipartite plane maps with prescribed large faces}

\begin{document}

\maketitle

\begin{abstract}
We first rephrase and unify known bijections between bipartite plane maps and labelled trees with the formalism of looptrees, which we argue to be both more relevant and technically simpler since the geometry of a looptree is explicitly encoded by the depth-first walk (or \L ukasiewicz path) of the tree, as opposed to the height or contour process for the tree. We then construct continuum analogues associated with any c\`adl\`ag path with no negative jump and derive several invariance principles. 
We especially focus on uniformly random looptrees and maps with prescribed face degrees and study their scaling limits in the presence of macroscopic faces, which complements a previous work in the case of no large faces. The limits (along subsequences for maps) form new families of random metric measured spaces related to processes with exchangeable increments with no negative jumps and our results generalise previous works which concerned the Brownian and stable L\'evy bridges. 
\end{abstract}

\begin{figure}[!ht] \centering
\begin{minipage}{.6\linewidth}
\includegraphics[width=\linewidth]{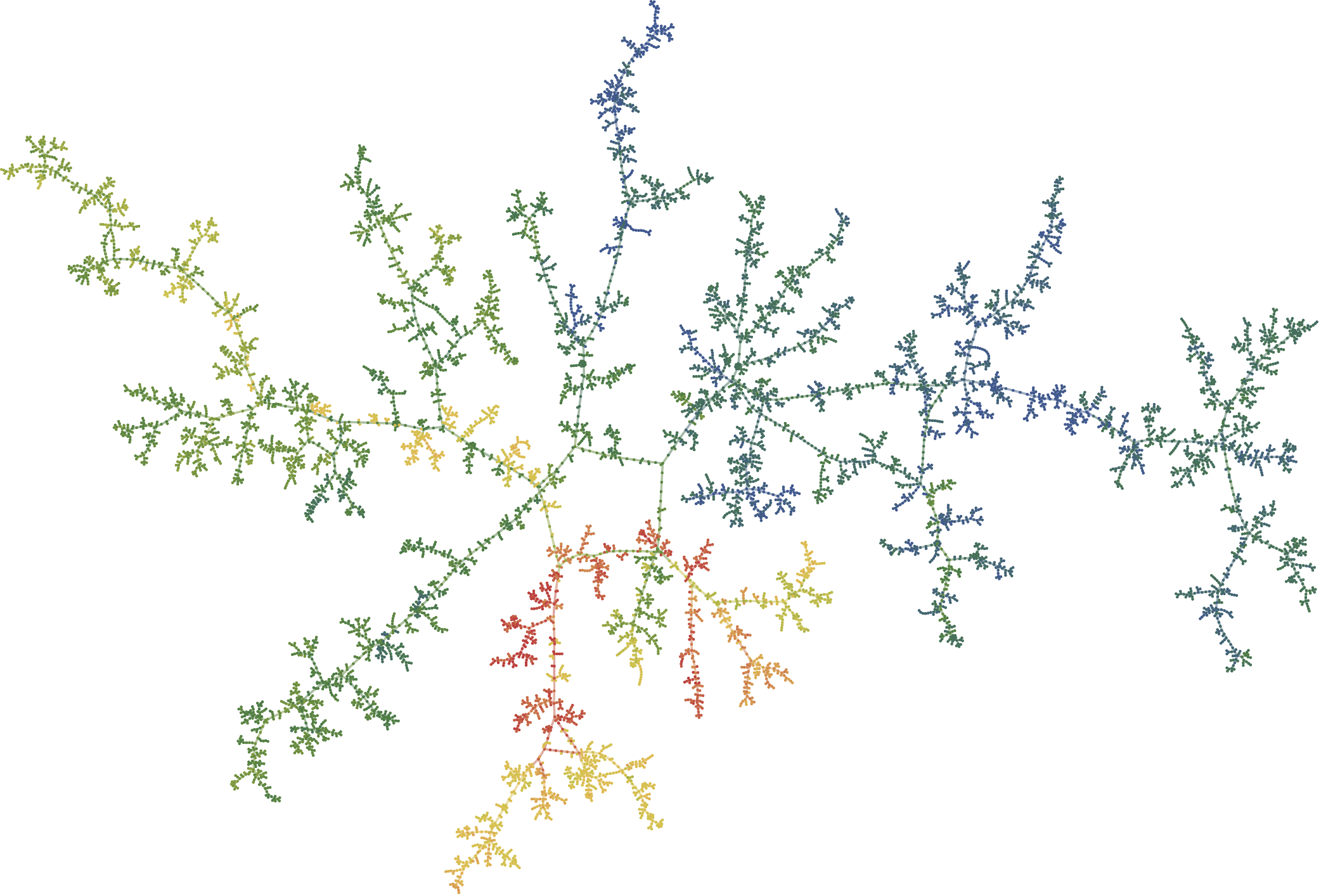}
\end{minipage}
\begin{minipage}[t]{.39\linewidth}
\includegraphics[width=\linewidth]{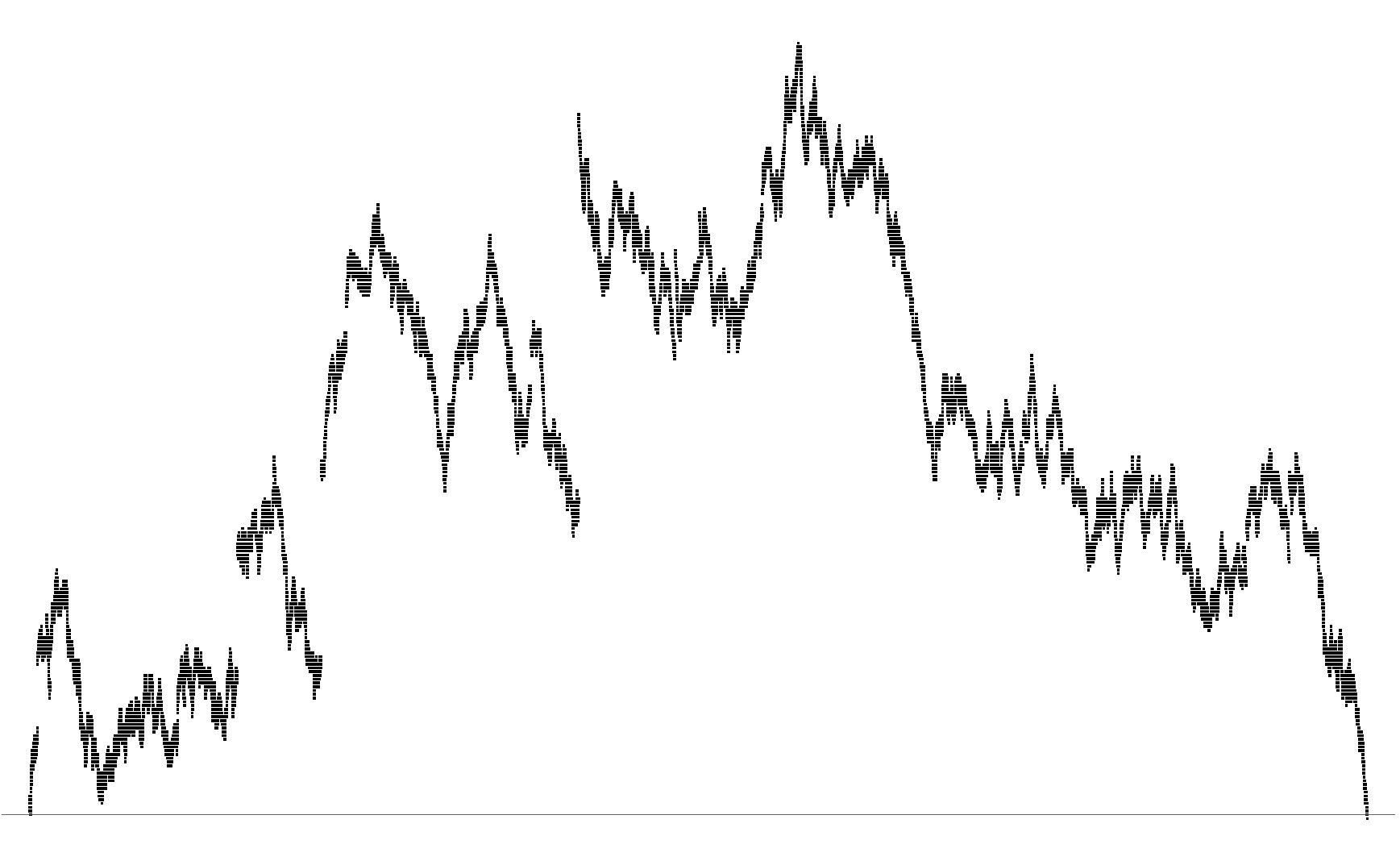}
\includegraphics[width=\linewidth]{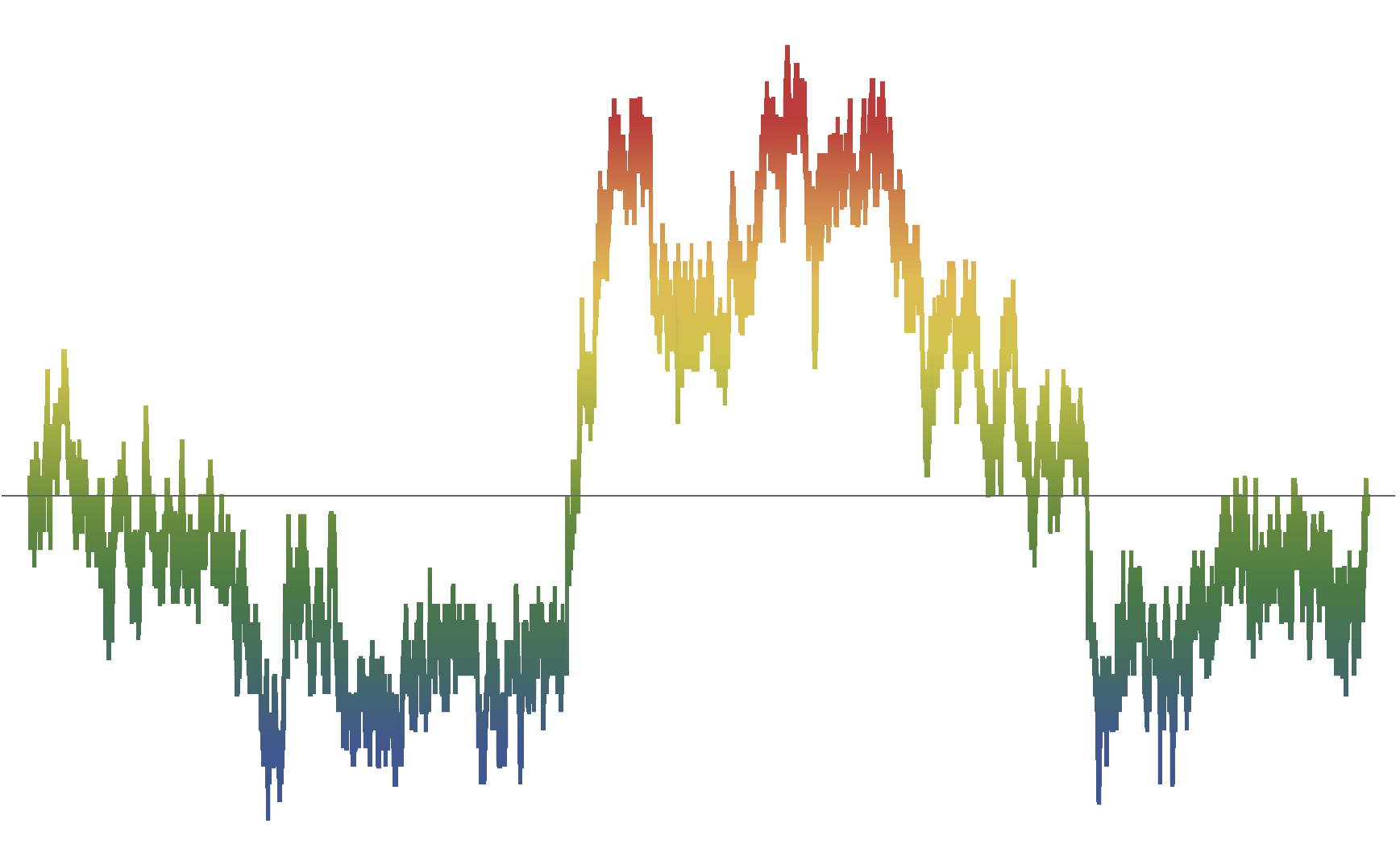}
\end{minipage}
\caption{\emph{Left:} A large uniformly labelled random looptree, drawn non-isometrically in the plane, 
which approximates the limit in the main theorems associated with a sequence satisfying both $\theta_0>0$ and  $\sum_i \theta_i = \infty$, with labels indicated by colours. The aspect is mostly tree-like, with a few long cycles.
\emph{Right:} An excursion with no negative jump coding its geometry on the top and a bridge with continous paths coding the labels on the bottom.}
\label{fig:simu}
\end{figure}

\clearpage

\setcounter{tocdepth}{2}
\tableofcontents
\clearpage

\section{Introduction}

A rooted plane map is the embedding of a finite, connected multigraph with a distinguished oriented edge (hereafter the \emph{root edge}) in the two-dimensional sphere and viewed up to orientation-preserving homeomorphisms.
We shall drop both adjectives and simply write ``maps'' in the rest of this paper.
Maps can be viewed as discrete surfaces and thus random maps offer simple models of random geometry and one can hope that by sampling large maps and letting the edge length tend to zero appropriately, one can obtain in the limit a nontrivial continuum random surface.
Due to embedding of the graph, it makes sense to define the \emph{faces} of the map as the connected components of the complement of the graph on the sphere; the \emph{degree} of a face is the number of edges incident to it, counted with multiplicity, i.e.~an edge incident on both sides to the same face contributes twice to its degree. 
The simplest model to study is that of quadrangulations with $n$ faces, all of which with degree $4$, sampled uniformly at random. For this model Chassaing \& Schaeffer~\cite{CS04} identified the growth rate as $n^{1/4}$ and after a series of works, Le~Gall~\cite{LG13} and Miermont~\cite{Mie13} independently proved that the metric space obtained by endowing the vertex set of such a quadrangulation with $(9/(8n))^{1/4}$ times the graph distance converges in distribution towards a random space $(S, D^\ast)$ called the \emph{Brownian sphere} (or \emph{Brownian map}). The latter is almost surely homemorphic to the sphere~\cite{LGP08, Mie08} and with Hausdorff dimension $4$~\cite{LG07}.

A natural question following this result is that of its universality. Already in~\cite{LG13}, Le~Gall proves that for $p=3$ or any $p \ge 6$ even, random maps with $n$ faces all of which with degree $p$ also converge in distribution towards $(S, D^\ast)$ at the scale $n^{1/4}$ up to a model-depending constant. The case of odd $p$'s has been treated more recently by Addario-Berry \& Albenque~\cite{ABA21}. 
In~\cite{Mar18b, Mar19}, a more general model was introduced, in which one chooses for every $n$ a deterministic array of $n$ positive (even) integers, which can differ from each others as well as vary with $n$, and one samples a plane map uniformly at random with $n$ faces whose degrees are given by these numbers.
In~\cite{Mar19}, fully extending~\cite{Mar18b}, convergence to the Brownian sphere for this model is shown under an optimal assumption of ``no macroscopic degree''.
The aim of this paper is to study this model in the regime when large faces are allowed. 
Let us next discuss our method before stating our main results.
Let us already mention that independently, Blanc-Renaudie~\cite{BR22} is currently working on the same model and same results but with completely different methods.

\subsection{A new formulation of the bijections with labelled trees}

A key tool to study the behaviour of large random maps
is a tailored constructive bijection with \emph{labelled trees}, i.e.~plane trees in which each vertex carries an integer, positive or not, and not necessarily distinct. It started with the celebrated Cori--Vauquelin--Schaeffer bijection that applies to quadrangulations and was then generalised to any plane maps by Bouttier, Di~Francesco, \& Guitter~\cite{BDG04} which was successfully used in numerous works such as~\cite{MM07, Mie06, LG07, Wei07, MW08, LG13, Abr16, BM17, ABA21}. This bijection takes a much simpler form in the case of \emph{bipartite} plane maps (when all faces have even degree), to which we will stick. Nevertheless the labelled trees, called \emph{mobiles} are still more complicated than in the case of quadrangulations. Later, Janson \& Stef{\'a}nsson~\cite{JS15} constructed a bijection between these mobiles and trees, thus providing a bijection between bipartite maps and simple labelled trees. This bijection has been used in the context of maps in e.g.~\cite{CK15, Ric18} without the labels, and in~\cite{Mar18b, Mar18a, Mar19} with the labels.

\begin{figure}[!ht] \centering
\includegraphics[page=7, height=9\baselineskip]{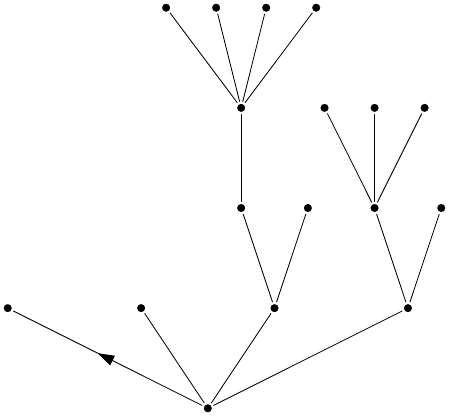}
\quad
\includegraphics[page=8, height=9\baselineskip]{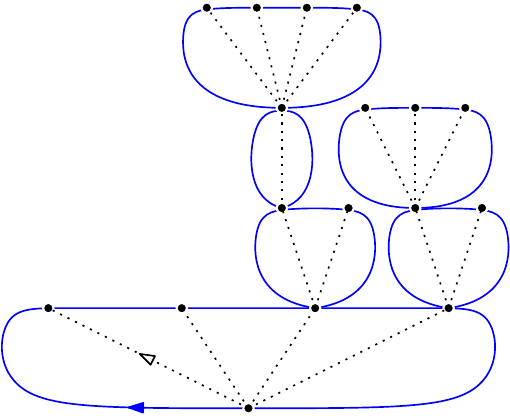}
\quad
\includegraphics[page=11, height=9\baselineskip]{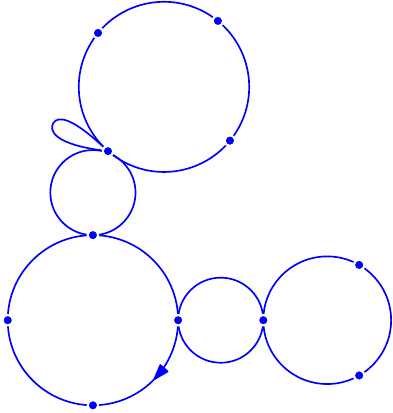}
\caption{From left to right: A plane tree, its looptree version as defined in~\cite{CK14}, and the version we consider here obtained by merging each internal vertex of the tree with its right-most offspring.}
\label{fig:looptree_intro}
\end{figure}

As opposed to all other listed papers, in~\cite{Mar19} the limit theorems for maps are obtained without a strong control on the associated trees themselves (their maximal height are unknown for example).
This suggests that these trees are in some sense not the right object of interest. Instead, we propose here to code the maps by a variation of the so-called \emph{looptree} version of a tree introduced in~\cite{CK14}, as depicted in Figure~\ref{fig:looptree_intro}. This encoding, described in Lemma~\ref{lem:bijection} below, is a simple modification of that from~\cite{BDG04} and the reader acquainted with the latter can directly look at Figure~\ref{fig:bijections} below. An example of this bijection is provided in Figure~\ref{fig:carte_intro}.
Let us already mention that it boils down to the CVS bijection in the particular case of quadrangulations, see Figure~\ref{fig:Schaeffer}.
Ultimately, we code these objects by a pair of discrete paths, which is the same as in~\cite{Mar19} but we believe that this viewpoint sheds some lights on the intermediate results of this reference and more generally on the construction of maps, their continuum analogues, and convergences. Furthermore, it allows to extend the results in~\cite{Mar19} without caring about the behaviour of the trees, which is still to be understood, although progress has been made very recently.
The point is that the geometry of the looptree is explicitly coded by the so-called \emph{{\L}ukasiewicz path}, whereas that of the tree is encoded in the \emph{height} or \emph{contour} process, and the latter is much more complicated to study than the former. 

Let us next introduce formally our model of random looptrees and maps.

\subsection{Random models with prescribed degrees}
\label{ssec:modele_intro}

Let us consider a triangular array of nonnegative integers: $\degf_{n,1} \ge \dots \ge \degf_{n,n+1} \ge 0$ such that $\sum_{i=1}^{n+1} \degf_{n,i} = n$ for every $n\ge 1$. 
For every subset $A \subset \Z$, let $\degf_n(A) = \#\{i \le n+1 : \degf_{n,i} \in A\}$ and simply write $\degf_n(k) = \degf_n(\{k\})$ in the case of a singleton.
In order to lighten the notation, we shall only indicate the dependence in $n$ in our objects, although they really depend on the whole sequence of $\degf_{n,i}$'s.
Then sample a looptree $LT^n$ uniformly at random amongst all those whose cycle lengths are the nonzero terms amongst $\degf_{n,1}, \dots, \degf_{n,n+1}$.
We refer the reader to Section~\ref{ssec:def_objets} for a formal definition of a looptree, and to Figure~\ref{fig:looptree_intro} for an example.
We will see that, once suitably randomly labelled, it encodes a bipartite map $M^n$ chosen uniformly at random amongst all those whose face degrees are the nonzero terms amongst $2\degf_{n,1}, \dots, 2\degf_{n,n+1}$.
One can check that both $LT^n$ and $M^n$ necessarily have $n$ edges and their number of vertices is respectively $\degf_n(0)$ and $\degf_n(0)+1$.

This model is inspired by similar trees sampled uniformly amongst all plane trees 
whose offspring numbers are the $\degf_{n,i}$'s,
studied in~\cite{AB12, BM14}, also~\cite{Lei19} in case of forests, and more recently in~\cite{AHUB20b, ABBHK21, ABDMM21, BR21, BHT21}. 
Their scaling limits are expected to be the so-called Inhomogeneous Continuum Random Trees studied especially in~\cite{AMP04} in a framework close to ours, and recently in~\cite{BR20} from another point of view. These trees are related to random processes with exchangeable increments which constitute the analogue of the {\L}ukasiewicz path in the discrete models. The convergence of the random {\L}ukasiewicz paths to such processes is well-understood, but this is not sufficient to control strongly the geometry of the trees and~\cite{BM14, Lei19} are restricted to the Brownian Continuum Random Tree, while the other papers are mostly restricted to a weak notion of convergence. What is especially missing is a tightness argument, however we shall prove in Theorem~\ref{thm:convergence_arbres_reduits} the convergence of such trees in the weak sense of so-called subtrees spanned by finitely many independent and uniform random vertices.
The study of the geometry of the associated labelled looptrees is however simpler.

The fundamental quantity which appears in our statements is
\[\sigma_n^2 = \sum_{i=1}^{n+1} \degf_{n,i} (\degf_{n,i} - 1) = \sum_{k \ge 0} k (k-1) \degf_n(k).\]
Observe that $k-1 \le k(k-1) \le 2(k-1)^2$ for every $k\ge 0$, which shows that
$\sigma_n^2$ lies between $\degf_n(0)-1$ and $2 (\degf_n(0)-1)^2$. We shall therefore assume henceforth that $\sigma_n^2 \to \infty$ as otherwise our graphs have a bounded number of vertices.
We shall also assume that there exists a sequence $\bdtheta$ of real numbers $\theta_1 \ge \theta_2 \ge \dots \ge 0$ such that
\begin{equation}\label{eq:hyp_sauts}
\sigma_n^{-1} \degf_{n,i} \cv \theta_i \quad\text{for every}\enskip i \ge 1.
\end{equation}
By Fatou's lemma, we have $\sum_{i \ge 1} \theta_i^2 \le 1$; let us define $\theta_0 \in [0,1]$ by
\[\theta_0^2 = 1 - \sum_{i \ge 1} \theta_i^2.\]
Let us refer to~\cite[Section~6.2]{AHUB20b} for explicit examples of triangular arrays satisfying~\eqref{eq:hyp_sauts} in the case $\theta_0=0$ and $\sum_i \theta_i = \infty$ as well as in the case $\theta_0, \theta_1, \theta_2, \ldots > 0$.

The limits in the theorems below will be constructed later in this paper. More precisely, Theorem~\ref{thm:convergence_looptree_intro} will be specified in Proposition~\ref{prop:tension_Holder_looptrees} and 
Theorem~\ref{thm:convergence_looptree_degres_prescrits_general}.
The definition of the Gromov--Hausdorff--Prokhorov topology is recalled in Section~\ref{ssec:labels_to_cartes}.

\begin{thmA}\label{thm:convergence_looptree_intro}
Suppose that $\sigma_n^2 \to \infty$ as $n\to\infty$.
\begin{enumerate}
\item From every increasing sequence of integers, one can extract a subsequence along which $\sigma_n^{-1} LT^n$ converges in distribution in the Gromov--Hausdorff--Prokhorov topology towards a limit with nonzero diameter.
\item Suppose that~\eqref{eq:hyp_sauts} holds and that there exists $a \ge 0$ such that 
\[\sigma_n^{-2} \degf_n(2\Z_+) \cv a \theta_0^2.\]
Then the convergence in distribution
\[\sigma_n^{-1} LT^n \cvloi \Loop(\bdtheta, a)\]
holds in the Gromov--Hausdorff--Prokhorov topology, where $\Loop(\bdtheta, a)$ is a random compact metric measured space whose law only depends on $a$ and $\bdtheta$.
\end{enumerate}
\end{thmA}

Let us next decorate our random looptree $LT^n$ by the analogue of a branching random walk (or random snake) on trees. Precisely, let $(\xi_k)_{k \ge 0}$ be i.i.d. copies of a random variable $\xi$ supported by $\Z$, which is centred and has variance $\Var(\xi) > 0$. Assume that for every $\ell \ge 1$, the event $\{\xi_1 + \dots + \xi_\ell = 0\}$ has a nonzero probability.
Then conditionally given $LT^n$, equip each vertex with a random label, such that the root has label $0$, and the label increments when following any cycle in clockwise order has the law of $(\xi_1, \dots, \xi_\ell)$ under $\P(\,\cdot\mid \xi_1 + \dots + \xi_\ell = 0)$, where $\ell$ is the length of the cycle, and finally the collection of these bridges for all cycles are independent.
We encode the labels into a process $Z^n = (Z^n_i)_{0 \le i \le n}$ which gives successively the label of the $i$'th vertex when following the contour of the looptree and we extend it to $[0,n]$ by linear interpolation. See the precise definition in Section~\ref{ssec:codage_chemins}. Observe that $Z^n_0 = Z^n_n = 0$.
Theorem~\ref{thm:convergence_labels_intro} will be specified in Corollary~\ref{cor:tension_Holder_labels} and 
Theorem~\ref{thm:convergence_labels_degres_prescrits_general}.

\begin{thmA}\label{thm:convergence_labels_intro}
Suppose that $\sigma_n^2 \to \infty$ as $n\to\infty$ and that $\E[|\xi|^{4+\varepsilon}] < \infty$ for some $\varepsilon > 0$.
\begin{enumerate}
\item From every increasing sequence of integers, one can extract a subsequence along which the processes $(\sigma_n^{-1/2} Z^n_{nt})_{t \in [0,1]}$ converge in distribution in the uniform topology towards a limit which is not constant null.
\item Suppose that~\eqref{eq:hyp_sauts} holds and that there exists $\Sigma \ge 0$ such that 
\[\sigma_n^{-2} \sum_{k \ge 1} \degf_n(k) k (k+1) \Var(\xi_1 \mid \xi_1 + \dots + \xi_k = 0) \cv \Var(\xi) (1 + (\Sigma^2-1) \theta_0^2).\]
Then the convergence in distribution
\[\left((\Var(\xi) \sigma_n)^{-1/2} Z^n_{n t}\right)_{t \in [0,1]} \cvloi Z^{\bdtheta, \Sigma^2/3},\]
holds for the uniform topology and the law of the limit only depends on $\bdtheta$ and $\Sigma$.
\end{enumerate}
\end{thmA}

Random maps are associated via the bijection of Lemma~\ref{lem:bijection} below (which is a mere reformulation of~\cite{BDG04}) with such labelled looptrees in the particular case when $\xi$ has the distribution given by $\P(\xi=i) = 2^{-i-2}$ for every $i \ge -1$; note that it admits moments of all order and that $\Var(\xi) = 2$.
Moreover, according to Marckert \& Miermont~\cite[page~1664]{MM07}, we have for all $k \ge 1$ (beware that they consider random bridges with length $k+1$),
\[\Var(\xi_1 \mid \xi_1 + \dots + \xi_k = 0) = \frac{2(k-1)}{k+1}.\] 
In this case, the series in Theorem~\ref{thm:convergence_labels_intro} above equals $2\sigma_n^2$ and so the constant $\Sigma$ simply equals $1$.
Let $M^n$ denote a bipartite plane map sampled uniformly at random with 
$\degf_n(k)$ faces with degree $2k$ for every $k\ge 1$.
Recall that it always has $\degf_n(0)+1$ vertices; let $u_n$ denote a vertex sampled independently and uniformly at random. We let $d_{M^n}$ denote the graph distance in $M^n$ and $d_{M^n}(u_n, \vec{e}_n)$ denote the smallest distance between the vertex $u_n$ and the endpoints of the root edge.
The next result which, again, will be specified in Theorem~\ref{thm:convergence_cartes_degres_prescrits} below, follows then straightforwardly from the preceding one and the properties of the bijection. The first statement can be found in~\cite{Mar19}.

\begin{figure}[!ht] \centering
\includegraphics[page=1, height=13\baselineskip]{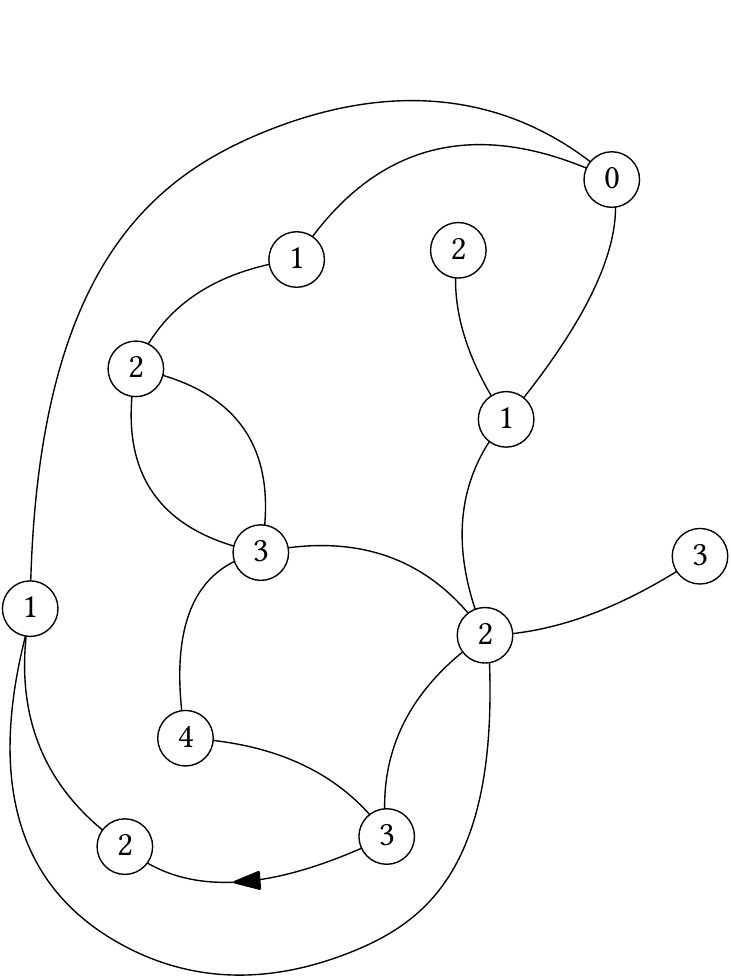}
\qquad
\includegraphics[page=17, height=11\baselineskip]{bijections}
\caption{\emph{Left:} A negative pointed bipartite map, with its vertices labelled by their distance to the distinguished vertex.
\emph{Right:} The associated looptree equipped with a good labelling via the bijection from Lemma~\ref{lem:bijection}.}
\label{fig:carte_intro}
\end{figure}

\begin{thmA}\label{thm:convergence_cartes_intro}
Suppose that $\sigma_n^2 \to \infty$ as $n\to\infty$.
\begin{enumerate}
\item From every increasing sequence of integers, one can extract a subsequence along which $\sigma_n^{-1/2} M^n$ converges in distribution in the Gromov--Hausdorff--Prokhorov topology to the quotient space $[0,1]/\{D_\infty=0\}$ where $D_\infty$ is a random continuous pseudo-distance on $[0,1]$, which is not constant null.

\item Suppose that~\eqref{eq:hyp_sauts} holds and set $\mathbf{Z} = Z^{\bdtheta, 1/3}$, then
\[\frac{1}{\sqrt{2 \sigma_n}} d_{M^n}(u_n, \vec{e}_n) \cvloi - \min \mathbf{Z}
\qquad\text{and}\qquad
\frac{1}{\sqrt{2 \sigma_n}} \max_{v \in V(M^n)} d_{M^n}(u_n, v) \cvloi \max \mathbf{Z} - \min \mathbf{Z},\]
and for every continuous and bounded function $F$,
\[\frac{1}{\degf_n(0)} \sum_{v \in V(M^n)} F\left(\frac{1}{\sqrt{2 \sigma_n}} d_{M^n}(u_n, v)\right) \cvloi \int_0^1 F(\mathbf{Z}_t - \min \mathbf{Z}) \d t.\]

\item Still under~\eqref{eq:hyp_sauts} all the subsequential limits 
$D_\infty$ are 
such that if $U$ is sampled uniformly at random on $[0,1]$ and independently of the rest, then
\[D_\infty(U,\cdot) \eqloi \mathbf{Z} - \min \mathbf{Z}.\]
\end{enumerate}
\end{thmA}

In the particular case $\theta_1 = 0$, when the largest degree is $\degf_{n,1} = o(\sigma_n)$, the space $\Loop(\bdtheta, a)$ reduces to the Brownian tree, coded by $a$ times the standard Brownian excursion, and the process $Z^{\bdtheta, \Sigma^2/3}$ is known in the literature as the head of the Brownian snake driven by $\Sigma^2/3$ times the standard Brownian excursion. In this case, Theorem~\ref{thm:convergence_cartes_intro} is completed in~\cite{Mar19} by showing that the subsequential limits all agree with $D_\infty = D^\ast$ and the rescaled maps thus converge without extraction towards the Brownian sphere $(S, D^\ast)$. This simply follows from the identity in law at the end of Theorem~\ref{thm:convergence_cartes_intro} and the work of Le~Gall~\cite{LG13} or Miermont~\cite{Mie13} on quadrangulations (which is therefore used as an input), via a re-rooting trick.
However for other models, this last identity is not sufficient to characterise the subsequential limits: it only gives the distances to a uniform random point, but one would need to identify the joint law of all the pairwise distances in any finite sample of i.i.d uniform random points.
This theorem can thus be seen as an extension to this more general model of maps of the work~\cite{LGM11} on so-called stable Boltzmann maps, for which proving uniqueness of the subsequential limits is under active investigation~\cite{CMR}. See Section~\ref{sec:Levy} for a discussion on the relation between the two models.

\subsection{Plan of the paper}

In Section~\ref{sec:bijection_objets_discrets}, we first recall the definitions of the discrete objects and we construct the bijection between bipartite maps and labelled looptrees. 
We then describe these objects by a pair of discrete paths. Then in Section~\ref{sec:deterministe} we construct analogously continuum looptrees from deterministic paths with no negative jump and then add random Gaussian labels on them. The constructions somewhat interpolate between those from~\cite{CK14, LGM11} which apply under some ``pure jump'' assumption, and the construction of the Brownian tree and snake from a continuous path. In Section~\ref{sec:convergence_saut_pur} we provide first invariance principles on randomly labelled looptrees in the pure jump case. In Section~\ref{sec:degres_prescrits} we introduce more precisely the models of random (loop)trees and maps with a prescribed degree sequence, we prove tightness results and we develop a consequence of a spinal decomposition from~\cite{Mar19} which will be the key to the identification of the subsequential limits. In Section~\ref{sec:echangeable} we discuss processes with exchangeable increments, which are the starting point of the construction of the limits. Section~\ref{sec:convergence} is then devoted to the proof of the invariance principles. Finally in Section~\ref{sec:Levy} we briefly present some applications with L\'evy processes, which are more understood than general exchangeable increment processes and will be studied more in the forthcoming work~\cite{KM22}.

\subsection*{Acknowledgment}

I am very grateful to Igor Kortchemski for multiple discussions on processes with exchangeable increments, especially the question presented in Remark~\ref{rem:PJ_echangeable_mieux}, also to Paul Th\'evenin and to Arthur Blanc-Renaudie for interesting discussions on their respective work~\cite{BHT21,BR22}.

\section{A new look at the bijections with labelled (loop)trees}
\label{sec:bijection_objets_discrets}

\subsection{The main characters}
\label{ssec:def_objets}

\paragraph{Maps and looptrees.}
Recall the notion of (rooted plane) maps from the introduction, which have a distinguished oriented root edge; we shall often also distinguish a vertex $v_\star$ in the map, in which case we say that the map is \emph{pointed}.
Then a (rooted plane) \emph{tree} is a map with only one face.
We shall think of such an object as a genealogical tree: the origin of the root edge (hereafter the \emph{root vertex}) is the ancestor of the family, the tip of the root edge is its left-most child, and then the neighbour of a vertex closer to the root is its parent whereas the other ones are its offspring, which are ordered from left to right. More generally, the vertices lying between the root and a given vertex are the ancestors of the latter. Finally an individual with no child is called a leaf, the other ones are called internal vertices.

Another way to view a tree is via its \emph{looptree} version and we use here a variation of the modified looptrees defined in Section~4 of~\cite{CK14}. First, for every internal vertex, instead of linking it to each offspring, only keep the edges to the first and last one and then link two consecutive offspring to each other; if it has only one offspring, then create a double edge. The resulting graph inherits a root edge from that of the tree.
This graph is the looptree defined in~\cite{CK14} and we shall denote such objects by $\overp{LT}$; here we further contract in each cycle the edge that links the parent to its right-most child and this looptree is denoted by $LT$. 
Let us refer to Figure~\ref{fig:looptree_intro} for an example.
Note that each cycle of this looptree corresponds to an internal vertex of the tree and the length of the cycle equals the offspring number of the vertex; in particular an individual with only one child induces a loop in the looptree. 
These looptrees already appeared in several works in order to study large boundaries of maps~\cite{CK15, Ric18, KR20}.

For a direct definition, the present looptrees are plane maps which satisfy the property that there is a distinguished ``outer face'', to the left of the root edge, and each edge has exactly one side incident to this face. This implies that all other ``inner faces'' are simple cycles and are edge-disjoint; also no edge is pending inside the outer face. 
The looptrees from~\cite{CK14} ironically forbid loops as well as vertices with degree more than $4$.

\paragraph{Contour sequence and labelled looptrees.}

Given a looptree with say $n$ edges one can orient these edges such that the outer face always lie to their left. In this way the edges are naturally ordered $e_0, \dots, e_{n-1}$, where $e_0$ is the root edge, by following the outer face in the direction prescribed by the oriented edges. We extend the sequence $(e_i)_{i \ge 0}$ by periodicity. Then for every $i \ge 1$, the angular sector in the outer face between $e_i$ and $e_{i-1}$ is called a \emph{corner} and denoted by $c_i$; we also set $c_0 = c_n$ to be the root corner. It may be useful later to identify $e_i$ and $c_i$ in a one-to-one correspondence. 
One can define a similar contour order on the $2n$ corners of a tree with $n$ edges by applying this construction to the looptree obtained by simply doubling each edge of the tree. 

Note that each corner is incident to a vertex of the (loop)tree, and may thus be identified with it, but of course the list of vertices thus obtained contains redundancies. We can get rid of them by retaining each vertex either only when it first appears in this list, which corresponds to the so-called \emph{depth-first search order}, or only when its last appears in this list, i.e.~when we finish visiting all its corners.
The following result should be clear from a picture and the proof is left as an exercise to the reader.

\begin{lem}\label{lem:ordre_contour_looptree_DFS}
Let $T$ be a tree with $n$ edges and let $LT$ denote its associated looptree. Then $LT$ has $n$ edges as well and furthermore the following two properties hold:
\begin{enumerate}
\item The corners $(c_i)_{1 \le i \le n}$ of $LT$ are in a one-to-one correspondence with the non root vertices of $T$, and the induced order is that of their first visit in the contour sequence of $T$;

\item The vertices of $LT$ are in a one-to-one correspondence with the leaves of $T$ and the order of appearance of the latter in the contour sequence of $T$ corresponds to the order of the vertices of $LT$ by their last visit in the contour sequence of $LT$.
\end{enumerate}
\end{lem}

Note the shift in the first property: the root corner $c_n = c_0$ is placed at the end of the sequence.
The first property comes from the fact that when following the contour of the tree, after the first visit of a vertex, each other visit corresponds to backtracking in the DFS order, while the second property comes from the fact that we close each loop when visiting the right-most leaf of the subtree of the descendants of the corresponding internal vertex.
This lemma will be used in Section~\ref{ssec:codage_chemins} below.

We shall equip looptrees with labels that assign to each vertex a real number, not necessarily distinct, see the right of Figure~\ref{fig:bijection_JS} for an example. 
To this end, we assign instead labels to the oriented edges (keeping the external face to the left) of the looptree, which satisfy the consistency relation that the sum of all labels along any cycle must equal $0$. Then given any two vertices and any path between them in the looptree, the sum of the labels of the edges on this path traversed in their canonical direction minus the sum of the labels of the edges traversed in opposite direction does not depend on the path. This labelling of the edges therefore defines a unique labelling of the vertices up to a global shift; we fix this shift by requiring that the label of the root is $0$.
We shall encode the labelling by considering the bridges when turning around each cycle. Precisely, for a given cycle with length, say $\ell \ge 1$, let $(e_{i_1}, \dots, e_{i_\ell})$ denote its oriented edges, with $0 \le i_1 < \dots < i_\ell \le n-1$, let then $(\xi_{i_1}, \dots, \xi_{i_\ell})$ denote the label increments along these oriented edges.
Note that the consistency relation reads $\xi_{i_1} + \dots + \xi_{i_\ell} = 0$.
The collection of all these bridges associated with each cycle then entirely characterises the labels on the looptree.

We say that a labelling of a looptree is \emph{good labelling} if the label increment $\xi$ along each edge oriented in the natural direction lies in $\Z_{\ge-1} = \{-1, 0, 1, 2, \dots\}$. In other words, each vector $(\xi_{i_1}, \dots, \xi_{i_\ell})$ associated with each cycle lies in the set
\begin{equation}\label{eq:pont_sans_saut_negatif}
\mathcal{B}_\ell^{\ge-1} \coloneqq \left\{(x_1, \dots, x_\ell) \in \Z_{\ge -1}^\ell : x_1 + x_2 + \dots + x_\ell = 0\right\}.
\end{equation}

\subsection{The bijection}
\label{ssec:la_bijection}

Let us now review the bijection of~\cite{BDG04}, in the particular case of bipartite maps, from the point of view of looptrees. 
Let $M$ denote a bipartite pointed map, i.e.~which a distinguished vertex $v_\star$ in addition to the root edge. Then label all the vertices of $M$ by their graph distance to $v_\star$ and observe that since $M$ is bipartite, then the labels along each edge differ by exactly one. In particular the root edge may either be positively or negatively oriented.
We only consider negative maps, when the tip is closer to $v_\star$ than the root vertex; if the map is positive, we therefore switch the orientation of the root edge.
The next result is really just a reformulation of~\cite[Section~2]{BDG04}.

\begin{lem}[\cite{BDG04}]
\label{lem:bijection}
There is a one-to-one correspondence between looptrees equipped with a good labelling and pointed negative bipartite plane maps, which furthermore enjoys the following two properties:
\begin{enumerate}
\item The map and the looptree have the same amount of edges;
\item The cycles of the looptree correspond to the faces of the map, and the length of a cycle is half the degree of the associated face;
\item The vertices of the looptree correspond to the non-distinguished vertices of the map, and their label, minus the smallest label plus one, equals the graph distance in the map of the associated vertex to the distinguished one.
\end{enumerate}
\end{lem}

Let us construct the correspondence in both directions. Let us already mention that a similar construction yields a bijection between rooted maps without distinguished vertex and looptrees equipped with a \emph{positive labelling}, in which instead all labels are positive and the root vertex has label $1$. However these objects are more complicated to study, which is why we consider pointed maps.

\begin{figure}[!ht] \centering
\includegraphics[page=2, height=16\baselineskip]{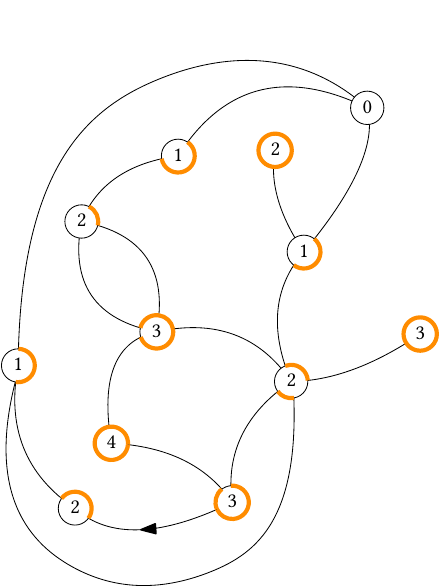}
\qquad
\includegraphics[page=5, height=16\baselineskip]{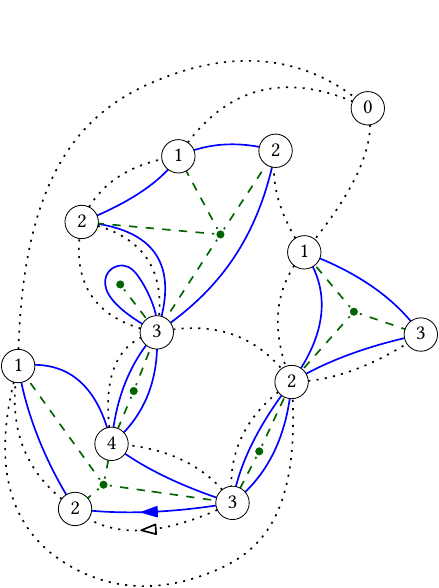}
\caption{
\emph{Left:} a negative map with its vertices labelled by their distance to a distinguished vertex; each corner is marked in orange if the next one in clockwise order inside the face has a smaller label.
\emph{Right:} the corresponding mobile in dashed green, obtained by linking in each face each marked corner to an extra vertex inside the face, 
and in plain blue the corresponding looptree, obtained instead by linking the marked corners in each face in a cycle; the labelling of its vertices is a good labelling shifted so the minimum label is $1$.
}
\label{fig:bijections}
\end{figure}

\paragraph{From maps to looptrees.}
Let us start with the map; inside each face, the cyclic sequence of labels in clockwise order only has increments either $+1$ or $-1$. Let us mark each corner if the next one in clockwise order in the face has a smaller label, so half of the corners in each face are marked. Note that the vertex $v_\star$ has no marked corner.
In~\cite{BDG04} one then adds a new vertex inside each face and link it to each marked corner of the face; the collection of these edges only and their endpoints is called a \emph{mobile}. Instead, let us join these marked corners inside each face in a cycle, whose length is therefore half the degree of the face. See Figure~\ref{fig:bijections} for an example where both constructions are depicted. Since the mobile is a tree~\cite[Section~2.1]{BDG04}, then what we construct here is a looptree; the mobile is simply a ``vertex-dual'' graph of the looptree, obtained by linking each vertex in each cycle to an extra vertex inside. Note that, since the map is negative, then the corner of the root vertex in the face to the right of the root edge is marked, then the edge we draw from this corner is chosen as the root of the looptree.

The distance labelling in the map induces a labelling of the looptree. In a sense, in each face, reading the labels in clockwise order, we have merged the chains of positive increments into single nonnegative jumps, therefore the label increment in the looptree along each edge oriented in the canonical direction lies in $\Z_{\ge-1}$. Furthermore, the labels are all positive and the minimum equals $1$; by shifting them all so the root has label $0$ we thus obtain a good labelling of the looptree.

\paragraph{From looptrees to maps.}
For the converse construction, take a looptree equipped with a good labelling and shift all labels so the minimum is $1$. Then let us assign to each corner in the outer face the label of the incident vertex and link each corner of the looptree to the next one (in the infinite periodic sequence) with a smaller label, which in fact can only be smaller by exactly $1$. Note that this construction fails for the corners labelled $1$, instead we join them all to an extra vertex labelled $0$ in the outer face.
By observing that the corner sequence of the looptree corresponds to the corner sequence defined in the mobile in~\cite[Section~2.2]{BDG04}, the arguments there show that the graph we just produced is a map, pointed at the extra vertex, and the labelling corresponds to the graph distance in this map to this vertex. The edge emanating from the root corner of the looptree is the root edge of the (necessarily negative) map.
Finally the two constructions are inverse of one another by~\cite[Section~2.3]{BDG04} and the claimed properties in Lemma~\ref{lem:bijection} follow from each construction respectively.

\begin{rem}\label{rem:bijection_Schaeffer}
When all the faces of the map are quadrangles, then all the cycles of the looptrees have length $2$, so a good labelling can actually only vary by either $-1$, $0$, or $+1$ along each oriented edge. This construction in fact reduces exactly to Schaeffer's bijection in which each edge of the tree has been doubled, see Figure~\ref{fig:Schaeffer} for an example.
\begin{figure}[!ht] \centering
\includegraphics[page=20, height=11\baselineskip]{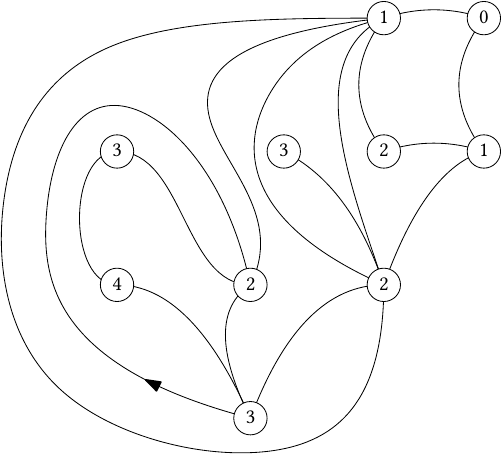}
\qquad\qquad
\includegraphics[page=21, height=11\baselineskip]{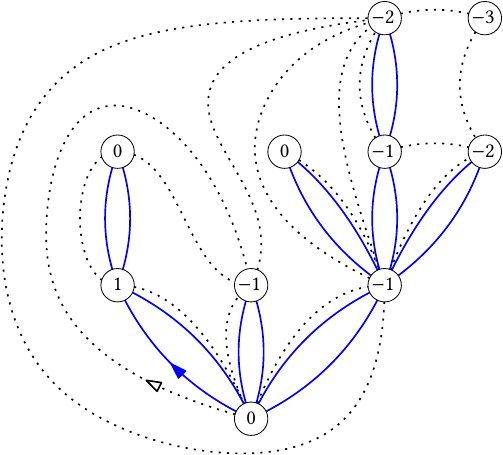}
\caption{In the particular case of quadrangulations, the looptree reduces to the CVS tree in which each edge has been doubled.}
\label{fig:Schaeffer}
\end{figure}
\end{rem}

\paragraph{Relation with one-type trees.}
Let us now mention the relation with the bijection from~\cite{JS15}, which is represented in our example in Figure~\ref{fig:bijection_JS}. See also~\cite[Section~2.4]{Mar18b} for a closely related discussion.
The mobile of~\cite{BDG04} is a tree in which two types of vertices alternate: the vertices of the map different from $v_\star$, hereafter called ``white'', and the extra vertices placed inside each face, hereafter called ``black''. The root edge of the mobile goes from the root vertex of the map to the black vertex inside the face to the right of the root edge of the map.
Then in~\cite{JS15} the authors modify the construction of the mobile as follow: 
if a white vertex is linked to several black vertices, then remove these edges and instead link these black vertices in a chain, starting from the closest one to the root of the mobile and turning around the white vertex in clockwise order; finally link the last black vertex to the white vertex.
This modification of the mobile creates another tree, in which the white vertices turn into leaves and the black vertices into internal vertices. Let us mention that this bijection between mobiles and trees was already constructed by Deutsch~\cite{Deu00}. A moment's thought and Lemma~\ref{lem:ordre_contour_looptree_DFS} show that the looptree version of this new tree exactly corresponds to the one we constructed previously directly from the map.

\begin{figure}[!ht] \centering
\includegraphics[page=4, height=12\baselineskip]{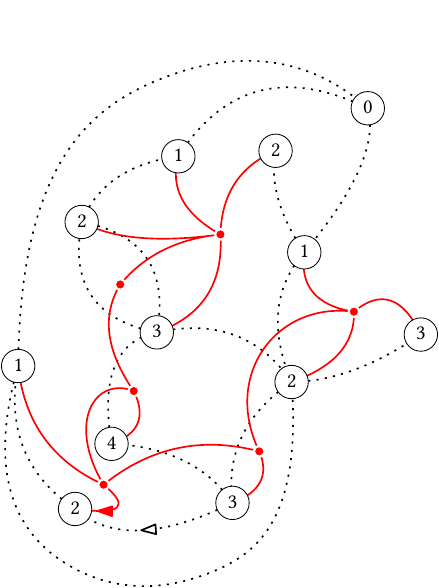}
\includegraphics[page=9, height=9\baselineskip]{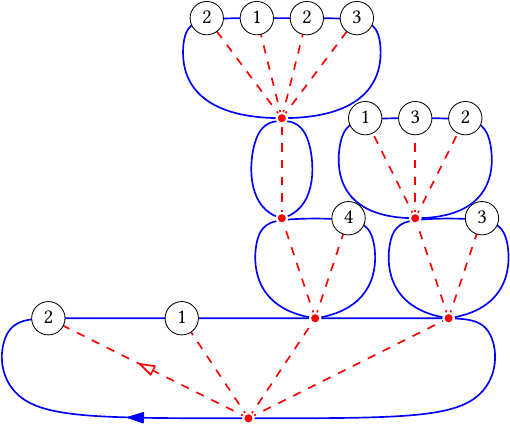}
\quad
\includegraphics[page=10, height=9\baselineskip]{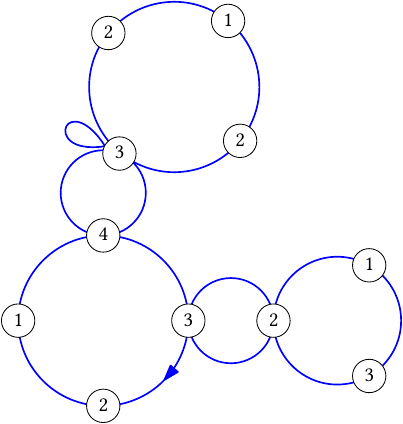}
\caption{\emph{Left:} The map in dotted lines and the labelled one-type tree in plain red.
\emph{Middle:} The tree in dashed red and in plain blue its looptree version as in~\cite{CK14}.
\emph{Right:} The looptree version we consider here, equipped with a good labelling shifted so the minimum label is $1$.
}
\label{fig:bijection_JS}
\end{figure}

\subsection{Coding paths}
\label{ssec:codage_chemins}

Lemma~\ref{lem:bijection} reduces the study of a pointed bipartite map to that of a simpler labelled looptree. We explain now how to code these objects by a pair of discrete paths, as illustrated in Figure~\ref{fig:Luka}. In the next section we then adapt this construction with paths that evolve in continuous time in order to build the potential scaling limits of large looptrees and maps.

\paragraph{The (modified) {\L}ukasiewicz path.}
Fix a tree $T$ with $n$ edges and let $LT$ be its associated looptree. Recall from Lemma~\ref{lem:ordre_contour_looptree_DFS} that the depth-first search order on the vertices of $T$ corresponds to the contour order of the corners of $LT$. Also the cycles of $LT$ correspond to the internal vertices of $T$, and the lengths of the former equal the offspring numbers of the latter. Therefore we may construct the so-called \emph{{\L}ukasiewicz path} of the tree directly from its looptree version as follows. 
We shall modify the usual definition in such a way that applies similarly to the discrete and continuum models.

Recall the notation $(e_0, \dots, e_{n-1})$ for the sequence of (oriented) edges of the looptree, then extended by periodicity; for every $i \ge 0$, let $\ell(e_i)$ denote the length of the cycle adjacent to $e_i$ \emph{if} $e_i$ is the first edge in $(e_j)_{j \ge 0}$ adjacent to this cycle, and let $\ell(e_i) = 0$ otherwise. Note that $\ell(e_i) = 0$ for every $i \ge n$.
Then define the {\L}ukasiewicz path $X = (X_t ; t \in \R)$ associated with $LT$
by $X_t = 0$ for every $t < 0$, and for every $t \ge 0$,
\[X_t = \sum_{i=0}^{\floor{t}} \ell(e_i) - (t-\floor{t}).\]
In words, $X_0 = \ell(e_0)$ equals the length of the cycle incident to the root edge; then at every integer time, say $i \in \{1, \dots, n\}$, the path $X$ makes a nonnegative jump $\Delta X_i = X_i - X_{i-} = \ell(e_i)$, and between integer times it decreases at unit speed. One easily checks that $X_{n} = 0$, and further $X_{n+t} = -t$ for every $t \ge 0$, whereas for every $0 \le t < n$, we have $X_t > 0$ (although it may occur that $X_{t-} = 0$ for some integer $t \in (0,n)$). 
See Figure~\ref{fig:Luka} for an example.

\begin{rem}
By Lemma~\ref{lem:ordre_contour_looptree_DFS}, the path $X$ is a simple variation of the classical depth-first walk of the plane tree $T$ associated with $LT$, which takes value $0$ at time $0$ and $X_{i-1}-1$ at time $i \in \{1, \dots, n+1\}$, and it is usually extended to non integer values by adding flat steps, see e.g.~\cite[Chapter~6]{Pit06} or~\cite{LG05}. Note that one path converges after suitable scaling if and only if the other one does, and with the same limit so this won't change anything, except that the construction is now exactly the same in the continuum setting.
\end{rem}

The path $X$ entirely describes the geometry of the tree and its looptree and one can explicitly recover the graph distance between two vertices of the looptree from $X$.
Indeed let us call \emph{ancestral cycles} of a vertex (and by abuse of notation, of a corner) the cycles which must be traversed by any path to the root. For $j \in \{1, \dots, n-1\}$, the ancestral cycles of $c_j$ are in one-to-one correspondence with the integers $i<j$ such that $X_t > X_{i-}$ for every $t \in (i, j)$. Note that necessarily for such an $i$ we have $\Delta X_i > 0$ and the precise value of $\Delta X_i$ is the length of the cycle; also $X_i - \inf_{t \in [i, j]} X_t$ and $\inf_{t \in [i, j]} X_t - X_{i-}$ equal respectively the length of the left and right part of the cycle when going from the root to $c_j$. This allows to recover the graph distance of $c_j$ to the root by summing the minimum of these two lengths over all ancestral cycles. To recover the graph distance between any to vertices, we add the preceding contribution over all cycles which are ancestral cycles of exactly one of them, whereas the cycle which is ancestral to both (if any) receives a special treatment. We refer to Equation~\eqref{eq:def_distance_looptree_reduit} below which generalises this idea to any c\`adl\`ag path with no negative jump.

\paragraph{The label process.}
Next, if the looptree is labelled, meaning that every vertex carries a number, then we define the \emph{label process} $Z = (Z_i)_{0 \le i \le n}$ such that $Z_i$ is the label of the origin of the edge $e_i$ for every $0 \le i \le n$. Since $e_n=e_0$, then $Z_0 = Z_n$ which shall always be $0$. We further extend $Z$ to the whole interval $[0,n]$ by linearly interpolating between integer times. See Figure~\ref{fig:Luka} for an example.
Then a labelled looptree is entirely characterised by the pair of paths $(X, Z)$.

\begin{figure}[!ht] \centering
\includegraphics[page=17, height=12\baselineskip]{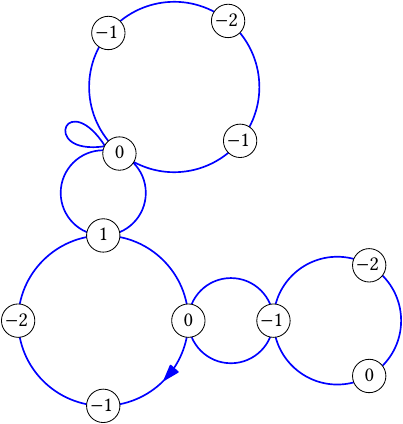}
\quad
\includegraphics[page=18, height=12\baselineskip]{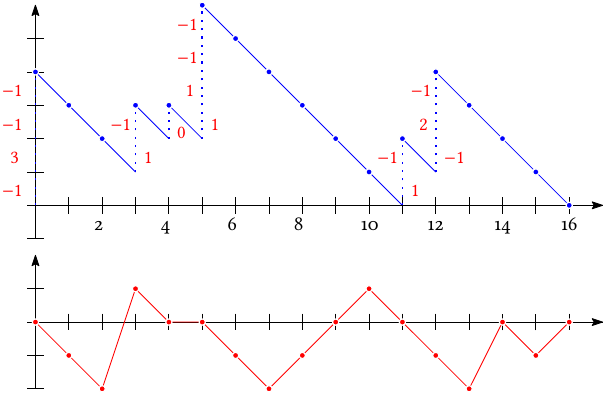}
\caption{\emph{Left:} A looptree equipped with a good labelling.
\emph{Top right:} its {\L}ukasiewicz path in blue and the values of the label increments along each edge in red, viewed as being indexed by the vertical jumps, from top to bottom.
\emph{Bottom right:} the label process.
}
\label{fig:Luka}
\end{figure}

It will be useful to describe the label process in terms of the {\L}ukasiewicz path. 
Recall that the labels on the vertices are obtained by summing labels on the edges of the looptree on the path from the root which always follows each cycle on its left. Recall also that we assign to each corner the label of its incident vertex.
Recall finally the bridges of label increments along each cycle introduced at the end of Section~\ref{ssec:def_objets}: each integer $i \in \{0, \dots, n-1\}$ such that $\Delta X_i > 0$ encodes a cycle of the looptree, with length $\ell_i = \Delta X_i$ and we denote by $(\xi^{i}_1, \dots, \xi^{i}_{\ell_i})$ the label increments along the edges of the cycle, in the order induced by following the contour of the looptree.
Then the label process is given by the formula: for every $0 \le j \le n$,
\begin{equation}\label{eq:def_processus_labels_discret}
Z_j \coloneqq \sum_{i=0}^{j-1} \sum_{k=0}^{\ell_i} \left(\xi^{i}_1+\dots+\xi^{i}_{k}\right) \ind{\inf_{t \in [i, j]} X_{\floor t} - X_{i-}=\ell_i-k}.
\end{equation}
We point out that the terms with $k = 0$ and $k = \ell_i$ actually give a null contribution.

Let us finally recall that the labelling is good when each bridge $(\xi^{i}_{1}, \dots, \xi^{i}_{\ell_i})$ belongs to $\mathcal{B}_{\ell_i}^{\ge-1}$ defined in~\eqref{eq:pont_sans_saut_negatif}. In this case, the increments of $Z$ all lie in $\Z_{\ge-1}$.
If $(\xi_k)_{k \ge 1}$ are i.i.d. copies of a random variable $\xi$ with the centred geometric distribution $\sum_{i \ge -1} 2^{-i-2} \delta_i$, then for every $\ell \ge 1$, one can check that the sequence $(\xi_1, \dots, \xi_\ell)$ under the conditional law $\P(\,\cdot\mid \xi_1+\dots+\xi_\ell = 0)$ has the uniform distribution on the set $\mathcal{B}_\ell^{\ge-1}$.
This provides a way to sample conditionally given $X$ a good labelling of the looptree uniformly at random.

\paragraph{The vertex-counting process.}
Consider a looptree with $n$ edges and its {\L}ukasiewicz path $X^n$.
Recall that following the contour of the looptree induces an order on the external corners, say $c^n_0, \dots, c^n_{n}$, where $c^n_n = c^n_0$. If we denote by $v^n_i$ the vertex incident to the corner $c^n_i$, then $(v^n_0, \dots, v^n_n)$ lists the vertices of looptree \emph{with redundancies} since some vertices are visited more than once.
For every $0 \le t \le n$, let us denote by $\Lambda^n(t)$ the number of vertices amongst $v^n_{0}, \dots, v^n_{\floor{t}}$ which have all their external corners in the list $(c^n_{0}, \dots, c^n_{\floor{t}})$ and note that it equals $\#\{0 \le i \le \floor{t} : \Delta X^n_i = 0\}$, the number of null jumps up to time $\floor{t}$; then let $V_n = \Lambda^n(n)$ denote the number of vertices of the looptree.
Conversely for $0 \le t \le V_n$, let $\lambda^n(t) \in \{0, \dots, n-1\}$ be the index such that the vertex $v^n_{\lambda(\ceil{t})}$ is the $\ceil{t}$'th vertex fully visited in the contour of the looptree. Note that the sequence $(v^n_{\lambda^n(0)}, \dots, v^n_{\lambda^n(V_n-1)})$ now lists the vertices \emph{without redundancies}. In our random model, the following convergence
\begin{equation}\label{eq:hyp_proportion_feuilles}
\left(V_n^{-1} \Lambda^n(n t), n^{-1} \lambda^n(V_n t)\right)_{t \in [0,1]} \cv \left(t, t\right)_{t \in [0,1]}
\end{equation}
will hold for the uniform topology.
Observe that since $\Lambda^n$ and $\lambda^n$ are inverse of one another and nondecreasing, then $n^{-1} \lambda^n(V_n \cdot)$ converges to the identity if and only if $V_n^{-1} \Lambda^n(n \cdot)$ does. Under this assumption, sampling uniformly at random in the looptree a corner (which can be done simply by sampling an instant in $[0,n]$) or a vertex is asymptotically equivalent.

\subsection{From labels to large plane maps}
\label{ssec:labels_to_cartes}

Let us end this section with invariance principles for bipartite pointed plane maps.
Precisely, we deduce partial results on the geometry of the maps from the convergence of the associated labels; this is very classical and we only recalled these results for the reader's convenience.
Lemma~\ref{lem:profil_general} below on the \emph{radius} and \emph{profile} was derived in a similar context in~\cite{MM07, LGM11} for random pointed bipartite maps and in~\cite{Mie06} for general maps, where in both cases, as here, distances are measured to the distinguished vertex. For non-pointed maps, where distances are measured to the origin of the root edge, it was derived for quadrangulations in~\cite{CS04, LG06}, bipartite maps in~\cite{Wei07}, and finally general maps in~\cite{MW08}.

Let $\Msf^n$ be a plane map with $n$ edges and $\Vsf_n+1$ vertices and let $v^n_\star$ be a distinguished vertex. Recall from Lemma~\ref{lem:bijection} that such a pair is in 2-to-1 correspondence (only the orientation of the root edge is unknown) with a labelled looptree, which is encoded by a pair $(\Xsf^n, \Zsf^n)$ where $\Xsf^n$ is the {\L}ukasiewicz path and $\Zsf^n$ the label process.
In this coding, the vertices of $\Msf^n$ different from $v^n_\star$ correspond to those of the looptree. 
Then with the preceding notation, the sequence $(v^n_{\lambda^n(0)}, \dots, v^n_{\lambda^n(\Vsf_n-1)})$ lists without redundancies the vertices of $\Msf_{n}$ different from $v^n_\star$.

The key property of the bijection from Lemma~\ref{lem:bijection} that we shall need is that the labels on the looptree correspond to the graph distance to $v^n_\star$ in the map, namely, for every $0 \le i \le n$,
\begin{equation}\label{eq:distance_carte_labels}
d_{\Msf^n}(v^n_{i}, v^n_{\star}) = \Zsf^{n}_{i} - \min \Zsf^{n}+1.
\end{equation}
The next result easily follows from the definitions and~\eqref{eq:distance_carte_labels}.
Recall from the introduction the notation $d_{\Msf^n}(v, \vec{e}_n)$ for the smallest between the vertex $v$ and the endpoints of the root edge.

\begin{lem}\label{lem:profil_general}
Let $\Zsf$ be a continuous function and suppose that there exists a sequence of positive real numbers $a_n \to \infty$ such that the convergence
\begin{equation}\label{eq:hyp_cv_labels_et_proportion_feuilles}
\left(a_n^{-1} \Zsf^n_{n t}, n^{-1} \lambda^n(\Vsf_n t)\right)_{t \in [0,1]} \cv \left(\Zsf_t, t\right)_{t \in [0,1]}
\end{equation}
holds for the uniform topology.
Then we have
\[\left(a_n^{-1}\dgr(v^n_\star, v^n_{\lambda^n(\Vsf_n t)})\right)_{t \in [0,1]} \cv \left(\Zsf_t - \min \Zsf\right)_{t \in [0,1]},\]
for the uniform convergence. In particular,
\[a_n^{-1} d_{\Msf^n}(v^n_\star, \vec{e}_n) \cvloi - \min \Zsf
\qquad\text{and}\qquad
a_n^{-1} \max_{v \in V(\Msf^n)} d_{\Msf^n}(v^n_\star, v) \cvloi \max \Zsf - \min \Zsf,\]
and for every continuous and bounded function $F$, we have
\[\Vsf_n^{-1} \sum_{v \in V(\Msf^n)} F\left(a_n^{-1} d_{\Msf^n}(v^n_\star, v)\right) \cvloi \int_0^1 F(\Zsf_t - \min \Zsf) \d t.\]
\end{lem}

\begin{proof}
The convergence of the distances to $v^n_\star$ is an immediate consequence of~\eqref{eq:distance_carte_labels} and the convergence~\eqref{eq:hyp_cv_labels_et_proportion_feuilles}. 
This directly implies the next two convergences. For the last one, let us write by~\eqref{eq:distance_carte_labels}
\begin{align*}
\sum_{v \in V(\Msf^n)} F\left(a_n^{-1} d_{\Msf^n}(v^n_\star, v)\right)
&= F(0) + \sum_{i = 1}^{\Vsf_n} F\left(a_n^{-1/2} \dgr(v^n_{\lambda^n(i)}, v^n_\star)\right)
\\
&= F(0) + \sum_{i = 1}^{\Vsf_n} F\left(a_n^{-1/2} \left(\Zsf^n_{\lambda^n(i)} - \min\Zsf^n\right)\right)
\\
&= F(0) + \Vsf_n \int_0^1 F\left(a_n^{-1/2} \left(\Zsf^n_{\lambda^n(\Vsf_n t)} - \min\Zsf^n\right)\right) \d t.
\end{align*}
Then the convergence follows from~\eqref{eq:hyp_cv_labels_et_proportion_feuilles} again.
\end{proof}

We finally consider the map as a metric measured space by endowing the underlying graph with its graph distance and uniform probability measure. Lemma~\ref{lem:tension_cartes_general} below on the convergence of this metric space \emph{after extraction of a subsequence} finds its root in the work of Le~Gall~\cite{LG07} using a different encoding and was since adapted in many works, e.g.~\cite{LGM11, LG13, Abr16, BM17, ABA21, Mar19}.
The major issue is then to prove that all the subsequential limits agree. In the case of random quadrangulations and the Brownian sphere, this was eventually solved simultaneously in~\cite{LG13, Mie13} and their arguments allow to extend this uniqueness to other discrete models that converge to the same limit and is a key input in the papers cited above. In the case of stable maps, this is under active investigation~\cite{CMR}.

In words, two compact metric spaces equipped with a Borel probability measure, say $(X, d_X, p_X)$ and $(Y, d_Y, p_Y)$, are close to each other in the \emph{Gromov--Hausdorff--Prokhorov topology} if one can find a subset of each which carries most of the mass and which are close to be isometric. Formally, a \emph{correspondence} between $X$ and $Y$ is a subset $R \subset X \times Y$ such that for every $x \in X$, there exists $y \in Y$ such that $(x,y) \in R$ and vice-versa. The \emph{distortion} of $R$ is defined as
\[\mathrm{dis}(R) = \sup\left\{\left|d_X(x,x') - d_Y(y,y')\right| ; (x,y), (x', y') \in R\right\}.\]
Then the Gromov--Hausdorff--Prokhorov distance between these spaces is the infimum of all the values $\varepsilon > 0$ such that there exists a coupling $\nu$ between $p_X$ and $p_Y$ and a compact correspondence $R$ between $X$ and $Y$ such that
\[\nu(R) \ge 1-\varepsilon \qquad\text{and}\qquad \mathrm{dis}(R) \le 2 \varepsilon.\]
This is only a pseudo-distance, but after taking the quotient by measure-preserving isometries, one gets a genuine distance which is separable and complete, see Miermont~\cite[Proposition~6]{Mie09}.

We keep the preceding notation: $\Msf^n$ is a plane map with $n$ edges and $\Vsf_n+1$ vertices and $v^n_\star$ denotes a distinguished vertex. Let $\Msf^n \setminus \{v_\star\}$ be the metric measured space given by the vertices of $\Msf^n$ different from $v_\star$, their graph distance \emph{in $\Msf^n$} and the uniform probability measure and let us observe that the Gromov--Hausdorff--Prokhorov distance between $\Msf^n$ and $\Msf^n \setminus \{v_\star\}$ is bounded by one so it suffices to consider the latter.
The distances to $v^n_\star$ are explicitly coded by the labels in~\eqref{eq:distance_carte_labels}, which is not the case of the distances between two arbitrary vertices. Recall that the vertices of $\Msf^n$ different from $v^n_\star$ are listed as $v^n_{\lambda^n(0)}, \dots, v^n_{\lambda^n(\Vsf_n-1)}$. 
For every $i,j \in \{0, \dots, \Vsf_n-1\}$, let us set
\begin{equation}\label{eq:distances_carte}
d_n(i,j) = d_{\Msf^n}(v^n_{\lambda^n(i)}, v^n_{\lambda^n(j)}),
\end{equation}
and then extend $d_n$ to a continuous function on $[0, \Vsf_n]^2$ by ``bilinear interpolation'' on each square of the form $[i,i+1] \times [j,j+1]$, as in~\cite[Section~2.5]{LG13} or~\cite[Section~7]{LGM11}.

As already mentioned, the next result can be traced back to Le~Gall's work~\cite{LG07} and we refer the reader to~\cite[Section~3.1]{Mar19} for a recent account of the arguments.
Let us briefly recall the main steps as the proof gives more information than what is stated.
Recall the assumption~\eqref{eq:hyp_proportion_feuilles} and the discussion below about the convergence of the vertex-counting process $\lambda^n$.

Here and below, if $D$ is a continuous pseudo-distance on $[0,1]$, we let $[0,1]/\{D=0\}$ denote the quotient space obtained by identifying pairs which lie at $D$-distance $0$ and we equip it with the induced distance and the image of the Lebesgue measure.

\begin{lem}\label{lem:tension_cartes_general}
Under the assumption~\eqref{eq:hyp_cv_labels_et_proportion_feuilles} in Lemma~\ref{lem:profil_general}, from every increasing sequence of integers, one can extract a subsequence along which
\[\left(a_n^{-1} d_n(\Vsf_n s, \Vsf_n t)\right)_{s,t \in [0,1]} \cv \left(d_\infty(s,t)\right)_{s,t \in [0,1]},\]
uniformly, where $d_\infty$ is a continuous pseudo-distance.
Consequently, along this subsequence, 
\[a_n^{-1}\Msf^n \cv{} [0,1]/\{d_\infty=0\}\]
in the Gromov--Hausdorff--Prokhorov topology.
\end{lem}

\begin{proof}
Let us modify $d_n$ by setting for every $i,j \in \{0, \dots, n\}$,
\[\widetilde{d}_n(i,j) = d_{\Msf^n}(v^n_i, v^n_j),\]
and then extend it similarly to the whole square $[0,n]^2$. 
From the convergence of $\lambda^n$ in~\eqref{eq:hyp_cv_labels_et_proportion_feuilles},
it is equivalent to replace $d_n(\Vsf_n \cdot, \Vsf_n \cdot)$ in the statement by $\widetilde{d}_n(n\cdot, n\cdot)$.

For a continuous function $g : [0, 1] \to \R$, let us set for every $0 \le s \le t \le 1$,
\[D_g(s,t) = D_g(t,s) = g(s) + g(t) - 2 \max\left\{\min_{r \in [s, t]} g(r); \min_{r \in [0, s] \cup [t, 1]} g(r)\right\}.\]
For $0 \le i < j \le n$, let $[i, j]$ denote the set of integers from $i$ to $j$, and let $[j,i]$ denote $[j, n] \cup [0, i]$. Recall that we construct our map from a labelled looptree, using a Schaeffer-type bijection; following the chain of edges drawn starting from two points of the forest to the next one with smaller label until they merge, one obtains the following upper bound on distances: for every $0 \le i,j \le n$,
\begin{equation}\label{eq:borne_sup_cactus}
d_{\Msf^n}(v^n_i, v^n_j) \le D_{L^{n}}(i, j) + 2.
\end{equation}
See Le~Gall~\cite[Lemma~3.1]{LG07} for a detailed proof in a different context.

By continuity, if $a_n^{-1} \Zsf^n$ converges uniformly towards $\Zsf$, then $a_n^{-1} D_{\Zsf^n}$ converges towards $D_\Zsf$. From the previous bound and~\eqref{eq:hyp_cv_labels_et_proportion_feuilles}, one can then extract a subsequence along which we have
\begin{equation}\label{eq:convergence_distances_sous_suite}
a_n^{-1} \left(\Zsf^n_{nt}, D_{\Zsf^n}(s, t), \widetilde{d}_{n}(ns, nt), d_{n}(\Vsf_n s, \Vsf_n t)\right)_{s,t \in [0,1]}
\cv
(\Zsf_t, D_\Zsf(s,t), d_\infty(s,t), d_\infty(s,t))_{s,t \in [0,1]},
\end{equation}
where $d_\infty$ is a continuous pseudo-distance which depends a priori on the subsequence and, by~\eqref{eq:borne_sup_cactus}, satisfies $d_\infty \le D_\Zsf$, see~\cite[Proposition~3.2]{LG07} for a detailed proof in a similar context.

Finally, let $\Msf_\infty$ be the quotient $[0,1] / \{d_\infty=0\}$ equipped with the metric induced by $d_\infty$, which we still denote by $d_\infty$. We let $\Pi_\infty$ be the canonical projection from $[0,1]$ to $\Msf_\infty$ which is continuous (since $d_\infty$ is) so $(\Msf_\infty, d_\infty)$ is a compact metric space, which finally we endow with the Borel probability measure $p_\infty$ given by the push-forward by $\Pi_\infty$ of the Lebesgue measure on $[0,1]$.
The set
\[R_n = \left\{\left(v^n_{\lambda^n(\Vsf_n t)}, \Pi_\infty(t)\right) ; t \in [0,1]\right\}.\]
is a correspondence between $\Msf^n\setminus\{v^n_\star\}$ and $\Msf_\infty$. Let further $\nu$ be the coupling between $p_{\Msf^n}$ and $p_\infty$ given by
\[\int_{(M_n\setminus\{v^n_\star\}) \times \Msf_\infty} f(x, z) \d\nu(x, z) = \int_0^1 f\left(v^n_{\lambda^n(\Vsf_n t)}, \Pi_\infty(t)\right) \d t,\]
for every test function $f$. Then $\nu$ is supported by $R_n$ by construction. Finally, the distortion of $R_n$ is given by
\[\sup_{s,t \in [0,1]} \left| a_n^{-1} d_{n}\big(\Vsf_n s, \Vsf_n t\big) - d_\infty(s,t)\right|,\]
which tends to $0$ whenever the convergence~\eqref{eq:convergence_distances_sous_suite} holds, which concludes the proof of the convergence of the maps.
\end{proof}

\section{Continuum labelled looptrees}
\label{sec:deterministe}

In this section, we replace our {\L}ukasiewicz paths by any c\`adl\`ag path with no negative jump and extend the construction of a looptree from this path in Section~\ref{ssec:generalites_looptrees}. We then construct random Gaussian labels on such a deterministic looptree in Section~\ref{ssec:labels_continus}. 
Let us first introduce some notation that we will use in the rest of this paper.

\subsection{Path decomposition}
\label{ssec:decomposition_sauts_partie_continue}

Let $\Xsf$ denote a c\`adl\`ag path in the space $\D([0, 1], \R)$, with $\Xsf_0 \ge 0$ and with no negative jump: $\Delta \Xsf_s \ge 0$ for all $s \in [0,1]$. Although it will not be needed for this conversation, we note that all the paths we shall consider also satisfy $X_1 = 0$. By convention, we shall extend it to $[-1, 1]$ by setting $\Xsf_t = 0$ for every $t \in [-1, 0)$. We stress that $\Xsf$ may have a positive jump at $0$.
Define a partial order $\preceq$ on $[0,1]$ as follows. For every $s, t \in [0,1]$, set
\[s \prec t \qquad\text{when both}\qquad s < t \quad\text{and}\quad \Xsf_{s-} \le \inf_{[s,t]} \Xsf.\]
We then set $s \preceq t$ if $s \prec t$ or $s=t$. 
Then for every pair $s, t \in [0,1]$, set
\begin{equation}\label{eq:def_R}
\Rsf^t_s \coloneqq \inf_{r \in [s,t]} \Xsf_r - \Xsf_{s-} \qquad\text{if}\enskip s \prec t,
\end{equation}
and let $\Rsf^t_s = 0$ otherwise. Note that $\Xsf$ must jump at time $s$ for $\Rsf^t_s$ to be nonzero.
In analogy to the discrete setting, when $\Xsf$ is the {\L}ukasiewicz path associated with a plane tree, we interpret $s = \sup\{r \in [0,1] : r \prec t\}$ as the parent of $t$, then $\Delta \Xsf_s$ as the offspring number of $s$, and finally $\Rsf^t_s$ is the number of siblings of $t$ that lie to its right; in the looptree, this corresponds to the right-length of the loop. Finally, for any $s, t \in [0,1]$, we let $s \wedge t \coloneq \sup\{r \in [0,1] : r \prec s \text{ and } r \prec t\}$ denote their last common ancestor.

Fix $t \in [0,1]$ and consider the dual path $\Xsf^t_s = \Xsf_{t-} - \Xsf_{(t-s)-}$ and its running supremum $\overline{\Xsf}\vphantom{X}^t_s = \sup_{r \in [0,s]} \Xsf^t_r$ for every $s \in [0,t]$. Then $\Rsf^t_{t-s}$ is equal to the size of the (possibly null) jump $\Delta \overline{\Xsf}\vphantom{X}^t_s$, and we may decompose this supremum path as the sum of its jump part and its continuous part, namely for $0 \le s \le t \le 1$,
\begin{equation}\label{eq:def_saut_pure_partie_continue}
\Jsf^t_s = \sum_{t-s \preceq r \prec t} \Rsf^t_r
\qquad\text{and}\qquad
\Csf^t_s 
= \overline{\Xsf}\vphantom{X}^t_s - \Jsf^t_s
= \Leb\big(\big\{\overline{\Xsf}\vphantom{X}^t_r; r \in [0,s]\big\}\big)
,\end{equation}
where here and below, a sum over an empty set is null, and $\Leb$ denotes the Lebesgue measure. Finally set
\begin{equation}\label{eq:def_J_et_C}
\Jsf_t = \Jsf^t_t
\qquad\text{and}\qquad
\Csf_t = \Csf^t_t,
\qquad\text{so}\qquad
\Jsf_t + \Csf_t = \Xsf_{t-}.
\end{equation}
We stress that even if $\Xsf$ is continuous as a process on $[0,1]$, it may have a jump $\Delta \Xsf_0 = \Xsf_0$ at the origin, which then contributes to $\Jsf$.
The path $\Csf$, thus given for every $t \in [0,1]$ by 
$\Csf_t = \Leb(\inf_{r \in [s,t]} \Xsf_r; s \in [0,t]\})$, is continuous, and it may be null. In the latter case, constructions and invariance principles are simpler, so we shall emphasise the following ``pure jump'' assumption:
\begin{equation*}\label{eq:PJ}\tag{PJ}
\text{For every $t \in [0,1]$, the function $\overline{\Xsf}\vphantom{X}^t$ is pure jump, i.e.~$\Csf_t=0$.}
\end{equation*}
For example, in Section~\ref{sec:Levy} we consider L\'evy processes, for which~\eqref{eq:PJ} is equivalent to having no Gaussian component in their L\'evy--Khintchine decomposition.

\subsection{Continuum trees}
\label{ssec:arbres_continus}

Let $g$ denote a continuous function on $[0,1]$. It is classical, see e.g.~\cite{LG05} for a detailed discussion, that it encodes a continuum real tree as follows. For every $0 \le s \le t \le 1$, set
\begin{equation}\label{eq:def_distance_arbre}
d_g(s,t) = g_s + g_t - 2 \min_{r \in [s,t]} g_r,
\end{equation}
and let $d_g(t,s) = d_g(s,t)$, then $d_g$ defines a continuous pseudo-distance. 
The quotient space $T_g = [0,1] / \{d_g=0\}$ is called the tree coded by $g$. 
It is a real tree in the topological sense, it is compact metric space when equipped with the canonical projection of $d_g$ (which we shall still denote by $d_g$), and finally the push-forward $p_g$ of the Lebesgue measure defines a uniform probability measure on $T_g$. 
This coding of a real tree from a continuous function is analogous to that of a discrete plane tree (actually a forest) from its associated contour function.

Given the path $\Xsf$, one can then construct the tree $T_\Csf$ from the continuous function $\Csf$ in~\eqref{eq:def_J_et_C}. We stress that it may not be the continuum tree for which $\Xsf$ would play the role of the {\L}ukasiewicz path. Indeed, under the pure jump assumption~\eqref{eq:PJ} it is reduced to a single point, whereas a nontrivial tree might still be constructed from $\Xsf$, see the example of L\'evy trees~\cite{DLG02} recalled in Section~\ref{sec:Levy}. However for L\'evy processes which do not satisfy~\eqref{eq:PJ}, the process $\Csf$ corresponds up to a scaling factor to the so-called height process, so $T_\Csf$ is indeed a multiple factor of the associated L\'evy tree (or rather forest).

\subsection{Continuum looptrees}
\label{ssec:generalites_looptrees}

We aim at defining a looptree $\Loop(\Xsf)$ from the path $\Xsf$.
Recall the partial order $\preceq$ on $[0,1]$ previously defined from $\Xsf$ and the notation $s \wedge t \coloneq \sup\{r \in [0,1] : r \preceq s \text{ and } r \preceq t\}$ for the last common ancestor of $s, t \in [0,1]$. Let $t \in [0,1]$ be a jump time of $\Xsf$, we associate with it a cycle $[0, \Delta \Xsf_t]$ with $\Delta \Xsf_t$ identified with $0$, equipped with the metric $\delta_t(a,b) = \min\{|a-b|, \Delta \Xsf_t - |a-b|\}$ for all $a,b \in [0, \Delta \Xsf_t]$. For definiteness, if $\Delta \Xsf_t = 0$, then we set $\delta_t(0,0) = 0$. For every $s, t \in [0,1]$, we let
\begin{equation}\label{eq:def_distance_looptree_reduit}
d_{\Loop(\Xsf)}^0(s,t) \coloneqq \delta_{s \wedge t}(\Rsf^s_{s \wedge t}, \Rsf^t_{s \wedge t}) + \sum_{s \wedge t \prec r \prec s} \delta_r(0, \Rsf^s_r) + \sum_{s \wedge t \prec r \prec t} \delta_r(0, \Rsf^t_r).
\end{equation}
When $s \prec t$, the first two terms on the right are zero.
This function was defined in~\cite{CK14} with the property~\eqref{eq:PJ} in mind but it is not satisfactory otherwise; in the extreme case when $\Xsf$ is continuous, it is constant equal to $0$ since only jump times 
can contribute to~\eqref{eq:def_distance_looptree_reduit}. Let us therefore modify it using the continuous function $\Csf$ from~\eqref{eq:def_J_et_C}: 
define the general looptree distance with parameter $a\ge0$ between any pair $s, t \in [0,1]$ by
\begin{equation}\label{eq:def_distance_looptree}
d^a_{\Loop(\Xsf)}(s,t) \coloneqq d_{\Loop(\Xsf)}^0(s,t) + a\cdot d_\Csf(s,t),
\end{equation}
and let simply $d_{\Loop(\Xsf)} = d^1_{\Loop(\Xsf)}$.
Intuitively $\Xsf$ plays the role of the {\L}ukasiewicz path, which counts for every vertex of the looptree the length of the path to the root by following each loop to their right; typically for random models, a geodesic from a vertex to the root follows a loop to its right about half of the time and follows the left part otherwise, so the parameter $a=1/2$ shall play a prominent role. However in a ``finite variance'' regime an extra factor comes into play; Theorem~\ref{thm:convergence_looptree_degres_prescrits_general} makes this discussion more precise.

\begin{rem}
Recall from Section~\ref{ssec:codage_chemins} the coding of a discrete looptree $\Lsf^n$ with $n$ edges by its {\L}ukasiewicz path $\Xsf^n$ and note that the latter satisfies~\eqref{eq:PJ}. Then Formula~\eqref{eq:def_distance_looptree_reduit} recovers the graph distance when $ns$ and $nt$ are both integers. In fact this formula describes the ``cable looptree'' obtained by replacing each edge of $\Lsf^n$ by a line segment with length $1$.
\end{rem}

The following result is given in~\cite[Lemma~2.1 and Proposition~2.2]{CK14} under~\eqref{eq:PJ}.

\begin{prop}\label{prop:distance_Looptree_general}
For any $a\ge 0$ the function $d_{\Loop(\Xsf)}^a$ is a continuous pseudo-distance on $[0,1]$. Moreover the following upper bound holds: for every $0 \le s \le t \le 1$,
\[d_{\Loop(\Xsf)}(s,t) \le \Xsf_s + \Xsf_{t-} - 2 \inf_{r \in [s,t]} \Xsf_r.\]
\end{prop}

In words, in the looptree distances, one follows the loops between $s$ and $t$ by taking their shortest side, whereas in the upper bound we always take the right side of each loop.

\begin{proof}
Let us first prove the upper bound for $d_{\Loop(\Xsf)}^0$.
First suppose that $s \prec t$ and observe that $\delta_r(0, \Rsf^t_r) \le \Rsf^t_r$ for any $s \prec r \prec t$, as well as $\delta_s(0, \Rsf^t_s) \le \Delta \Xsf_s - \Rsf^t_s$, so the identity~\eqref{eq:def_saut_pure_partie_continue} yields
\[d_{\Loop(\Xsf)}^0(s,t) 
= \delta_s(0, \Rsf^t_s) + \sum_{s \prec r \prec t} \delta_r(0, \Rsf^t_r)
\le \Delta \Xsf_s  + \Jsf^t_{t-s} - 2 \Rsf^t_s
.\]
Observe that
$\Jsf^t_{t-s} 
\le \overline{\Xsf}\vphantom{X}^t_{t-s} 
= \Xsf_{t-} - \Xsf_{s-}
$
so the bound follows in this case.
If $s$ is not an ancestor of $t$, then $s \wedge t < s$ and the minimum of $\Xsf$ between $s \wedge t$ and $t$ is actually achieved in the interval $[s,t]$. Moreover
$\delta_{s \wedge t}(\Rsf^s_{s \wedge t}, \Rsf^t_{s \wedge t})
\le \Rsf^s_{s \wedge t} - \Rsf^t_{s \wedge t}
$.
We infer that
\[d_{\Loop(\Xsf)}^0(s,t) 
\le \Jsf^s_{s - s \wedge t} + \Jsf^t_{t - s \wedge t} - 2 \Rsf^t_{s \wedge t}
.\]
As previously, we have
$\Jsf^s_{s - s \wedge t} + \Jsf^t_{t - s \wedge t} \le \Xsf_{s-} + \Xsf_{t-} - 2 \Xsf_{(s \wedge t)-}$ and the upper bound on $d_{\Loop(\Xsf)}^0(s,t)$ follows.

Let us turn to $d_{\Loop(\Xsf)}$. In upper bounding $\Jsf^t$ by $\overline{\Xsf}\vphantom{X}^t$, we actually added the continuous part $\Csf^t$. When $s \prec t$, we have $d_{\Csf}(s,t) = \Csf^t_{t-s}$ so
\[d_{\Loop(\Xsf)}(s,t) 
\le d_{\Loop(\Xsf)}^0(s,t) + \Csf^t_{t-s}
\le \Xsf_s + \Xsf_{t-} - 2 \inf_{[s,t]} \Xsf
.\]
Similarly, when $s \wedge t < s$,
\[d_{\Csf}(s,t) 
= \Csf^s_s + \Csf^t_{t} - 2 \inf_{r \in [s,t]} \Csf^{(r)}_r
= \Csf^s_{s-s\wedge t} + \Csf^t_{s-s\wedge t} 
,\]
and so
\[d_{\Loop(\Xsf)}(s,t) 
\le \Jsf^s_{s - s \wedge t} + \Jsf^t_{t - s \wedge t} - 2 \Rsf^t_{s \wedge t} + \Csf^s_{s-s\wedge t} + \Csf^t_{s-s\wedge t}
\le \Xsf_{s-} + \Xsf_{t-} - 2 \inf_{[s,t]} \Xsf.\]

A similar reasoning shows that $d_{\Loop(\Xsf)}^a$ satisfies the triangular inequality for any $a\ge0$ and is left to the reader. Moreover the upper bound shows that it is finite so it indeed defines a pseudo-distance. It remains to prove that it is continuous. 
By the triangle inequality, for every $s,s',t,t' \in [0,1]$,
\[|d^a_{\Loop(\Xsf)}(s, s') - d^a_{\Loop(\Xsf)}(t, t')| \le d^a_{\Loop(\Xsf)}(s, t) + d^a_{\Loop(\Xsf)}(s', t'),\]
and the right hand side is arbitrarily small when $s$ is close to $t$ and $s'$ close to $t'$ by the preceding upper bound.
\end{proof}

\subsection{Randomly labelled looptrees}
\label{ssec:labels_continus}

Let $g$ be a continuous function defined on $[0,1]$ and recall from the beginning of Section~\ref{ssec:generalites_looptrees} that it encodes a real tree $T_g$. One may define a random Gaussian field on this tree, known as the (head of the) \emph{Brownian snake} driven by $g$; formally, given $g$, there exists a centred Gaussian process $Z^g = (Z^g_t)_{t \in [0,1]}$ with covariance function
\[\E[Z^g_s Z^g_t] = \min_{r \in [s,t]} g_r,
\qquad s,t \in [0,1].\]
Equivalently, $\E[(Z^g_s-Z^g_t)^2] = d_g(s,t)$, where the latter is defined in~\eqref{eq:def_distance_arbre}.
Let us mention that the Brownian snake is actually a path-valued process which describes for every $t \in [0,1]$ the path $Z^g_s$ for all the ancestors $s$ of $t$ in $T_g$; see~\cite[Chapter~4]{DLG02} for details on such processes in a broader setting.
On an intuitive level, after associating a time $t$ with its projection in $T_g$, the values of $Z^g$ evolve along the branches of the tree like a Brownian motion, and these Brownian motions separate at the branchpoints to continue to evolve independently.

In the case of a looptree, although without this formalism, Le~Gall \& Miermont~\cite{LGM11} constructed a process that describes a Gaussian field on $\Loop(\Xsf)$ under the assumption~\eqref{eq:PJ}. In this case, one places on each cycle of the looptree an independent Brownian bridge, with duration given by the length of the cycle, which describes the increments of the field along the cycle. The value of the field at a given point is then the sum of the increments along a geodesic path on $\Loop(\Xsf)$ to the root, analogously to~\eqref{eq:def_processus_labels_discret} in the discrete setting. In the case where $\Csf$ defined in~\eqref{eq:def_J_et_C} is nontrivial, one then adds an independent Brownian snake driven by $\Csf$.

Formally, recall that the standard Brownian bridge $b = (b_t)_{t \in [0,1]}$ is centred Gaussian process with covariance
\[\E[b_s b_t] = \min(s,t) - st,
\qquad s,t \in [0,1].\]
In particular, the variance is $\E[b_t^2] = t(1-t)$. One can consider bridges of arbitrary given duration using a diffusive scaling.
Now recall the notation $\Rsf^t_{t_i} = \inf_{s \in [t_i,t]} \Xsf_s - \Xsf_{t_i-}$ from~\eqref{eq:def_R}, where the $t_i$'s are the jump times of $\Xsf$ up to time $t$, and also the partial order $\prec$ above. Let us finally recall the notation $\delta_t(a,b)$ used in~\eqref{eq:def_distance_looptree_reduit}.
Let $Z^\Csf$ be defined as above, and independently let $(b_i)_{i \ge 1}$ denote a sequence of i.i.d. standard Brownian bridges, 
and define for every $a\ge0$ and $t \in [0,1]$,
\begin{equation}\label{eq:def_labels_browniens}
Z^a_t \coloneqq \sqrt{a}\, Z^\Csf_t + \sum_{t_i \prec t} \sqrt{\Delta \Xsf_{t_i}}\, b_i\left((\Delta \Xsf_{t_i})^{-1} \Rsf^t_{t_i}\right),
\end{equation}
and set $Z = Z^1$. Note that by scaling $\sqrt{a}\, Z^\Csf$ has the same law as $Z^{a\Csf}$.

Let us prove that the series converges in $L^2$ for any fixed $t$; note that only jump times $t_i \prec t$ of $\Xsf$ contribute.
Since the $b_i$'s are independent Brownian bridges, then the nonzero summands in the formula are independent zero-mean Gaussian random variables with variance
\begin{align*}
\E\bigg[\left((\Delta \Xsf_{t_i})^{1/2} b_i\left((\Delta \Xsf_{t_i})^{-1} \Rsf^t_{t_i}\right)\right)^2\bigg]
= \frac{\Rsf^t_{t_i} (\Delta \Xsf_{t_i} - \Rsf^t_{t_i})}{\Delta \Xsf_{t_i}}
\le \delta_{t_i}(0, \Rsf^t_{t_i}).
\end{align*}
Then by definition of the looptree distance~\eqref{eq:def_distance_looptree},
\[\E[|Z^a_t|^2]
\le a \Csf_t + \sum_{t_i \prec t} \delta_{t_i}(0, \Rsf^t_{t_i})
= d^a_{\Loop(\Xsf)}(0,t).\]
This shows that the random variable $Z^a_t$ is well-defined for any fixed $t$.
Note that by Proposition~\ref{prop:distance_Looptree_general} we have $\E[|Z_t|^2] \le X_{t-}$.

\begin{prop}\label{prop:existence_labels_continus_general}
Suppose that there exist $c > 0$ and $K>0$ such that for every $s, t \in [0,1]$,
\begin{equation}\label{eq:hypothese_Kolmogorov_continuite_distance_loop}
d_{\Loop(\Xsf)}(s,t) \le K \cdot |t-s|^{c}.
\end{equation}
Then the process $(Z_t)_{t \in [0,1]}$ admits a continuous modification that is almost surely H\"older continuous with any exponent smaller than $c/2$.
\end{prop}

\begin{proof}
The preceding argument generalises to show that for any $q>0$, there exists $K_q > 0$ such that for every $s, t \in [0,1]$, it holds that
\begin{equation}\label{eq:label_moment_Kolmogorov_looptrees}
\Es{|Z_t-Z_s|^q} \le K_q \cdot d_{\Loop(\Xsf)}(s,t)^{q/2}.
\end{equation}
Indeed, note that $Z_t-Z_s$ is the sum of independent Gaussian random variables, and by splitting at the last common ancestor of $s$ and $t$, we see that the sum of the variances involved is given exactly by the same quantity as $d_{\Loop}(s,t)$, except that the function $\delta_t(x,y) = \min\{|x-y|, \Delta \Xsf_t - |x-y|\}$ used in~\eqref{eq:def_distance_looptree_reduit} is replaced by the covariance function $\widehat{\delta}_t(x) = |x-y| (\Delta \Xsf_t - |x-y|) / \Delta \Xsf_t \le \delta_t(x,y)$. We refer to the proof of Proposition~6 in~\cite{LGM11} for details. We infer that~\eqref{eq:label_moment_Kolmogorov_looptrees} holds, where $K_q$ is the absolute $q$'th moment of a standard Gaussian random variable.
We conclude from the Kolmogorov criterion and the assumption~\eqref{eq:hypothese_Kolmogorov_continuite_distance_loop}.
\end{proof}

\begin{rem}\label{rem:Holder_continuite_labels_looptree}
If $s$ and $t$ are such that $d_{\Loop(\Xsf)}(s,t) = 0$, then $Z_t-Z_s = 0$ almost surely by~\eqref{eq:label_moment_Kolmogorov_looptrees} so one can view $Z$ as a random process indexed by the looptree. The bound~\eqref{eq:label_moment_Kolmogorov_looptrees} further shows
that the process $(Z_x ; x \in \Loop(\Xsf))$ has a modification that is almost surely H\"older continuous with any exponent smaller than $1/2$.
\end{rem}

Let us mention that, even when $\Loop(\Xsf)$ is a tree coded by a continuous function, it may be the case that the process $Z$ is not continuous, see~\cite[Chapter~4.5]{DLG02}.
On the other hand, in Sections~\ref{sec:degres_prescrits} to~\ref{sec:convergence} we consider random paths with exchangeable increments and they do satisfy~\eqref{eq:hypothese_Kolmogorov_continuite_distance_loop} for some $c>0$.
In the particular case of a Brownian excursion, in which case the looptree reduces to the Brownian tree, from the H\"older continuity of the Brownian excursion, $c$ can take any value in $(0,1/2)$, but cannot exceed $1/2$.
On the other hand, when $\Xsf$ 
is an excursion of a stable L\'evy process with index $\alpha \in (1,2)$, although not explicitly written there, we can derive from~\cite{CK14} that~\eqref{eq:hypothese_Kolmogorov_continuite_distance_loop} holds with any $c$ smaller than $1/\alpha$, and again this is the largest possible value. For more general L\'evy processes, we prove in the forthcoming paper~\cite{KM22} that the optimal value for $c$ is the inverse of the so-called upper Blumenthal--Getoor exponent, see Section~\ref{sec:Levy} for more information.

\section{Invariance principles in the pure jump case}
\label{sec:convergence_saut_pur}

In this short section, we consider deterministic c\`adl\`ag paths $\Xsf$ and $(\Xsf^n)_{n \ge 1}$ with $\Xsf_0, \Xsf^n_0 \ge 0$ and no negative jump, all extended to $[-1, 1]$ by setting $\Xsf_t = \Xsf^n_t = 0$ for every $t \in [-1, 0)$. We assume that $\Xsf^n$ converges to some $\Xsf$ for the Skorokhod topology and that $\Xsf$ satisfies the pure jump property~\eqref{eq:PJ} and we derive invariance principles for the associated labelled looptrees and maps.
Before that, let us start with a simple result that we shall use throughout Section~\ref{sec:convergence} and whose proof's ideas will be used in Proposition~\ref{prop:convergence_looptrees_saut_pur} just below.

\subsection{The continuous part as limit of small jumps}
\label{ssec:approximation_partie_continue}

Recall from Section~\ref{ssec:decomposition_sauts_partie_continue} that we denote by $\Csf$ the continuous part of the supremum of the dual path of $\Xsf$, and by $\Jsf$ the jump part, so that $\Csf_t+\Jsf_t = \Xsf_{t-}$ for every $t\in[0,1]$. For $\delta > 0$, let us define $\Csf_t^\delta$ as the complement of the contribution to $\Xsf_{t-}$ of the jumps larger than $\delta$, namely for every $t\in[0,1]$,
\begin{equation}\label{eq:approximation_delta_partie_continue}
\Csf^\delta_t \coloneqq \Csf_t + \sum_{0 \preceq r \prec t} \Rsf^t_r \ind{\Delta \Xsf_r \le \delta}.
\end{equation}
We prove that for every $\delta>0$, the quantity $\Csf^\delta$ depends continuously on $\Xsf$, which is not the case of $\Csf$ itself.

\begin{lem}\label{lem:convergence_partie_continue}
Suppose that $\Xsf^n \to \Xsf$ for the Skorokhod topology as $n\to\infty$, then for every $\delta>0$ the path $\Csf^{n,\delta}$ converges uniformly to $\Csf^\delta$, which further converges uniformly to $\Csf$ as $\delta\downarrow0$.
\end{lem}

\begin{proof}
The second convergence follows from Dini's theorem which shows that if $(r_i)_{i\ge 1}$ is an enumeration of the set $\{r \prec t\}$, then the series $\sum_{i=1}^N \Rsf^t_{r_i}$ converges uniformly to $J_t$ as $N\to\infty$. Let us focus on the first assertion.
Fix $\delta>0$ and suppose by contradiction that $\Csf^{n,\delta}$ does not converge uniformly to $\Csf^\delta$, then 
there exists $\varepsilon > 0$ and a sequence $(t_n)_n$ such that $|\Csf^{n,\delta}_{t_n} - \Csf^\delta_{t_n}| \ge \varepsilon$ for infinitely many indices $n$. Extracting a subsequence if necessary, let us assume that this holds for all $n$ and that $t_n$ converges to some $t \in [0,1]$.
Note that the set $\{0 \preceq r \prec t : \Delta \Xsf_r > \delta\}$ is finite, and properties of the Skorokhod topology imply that the values $\Rsf^t_r$ for such $r$'s are the limits of analogous values $\Rsf^{n, t_n}_{r_n}$ for $0 \preceq r_n \prec t_n$. Consequently the sum of the latter, $\Xsf^n_{t_n-}-\Csf^{n,\delta}_{t_n}$, converges to that of the former, $\Xsf_{t-}-\Csf^\delta_{t}$.
On the other hand, the convergence in the Skorokhod topology also implies that $\Xsf_{t_n-}$ converges to $\Xsf_{t-}$ and this yields a contradiction.
\end{proof}

\begin{rem}\label{rem:convergence_partie_continue_Luka}
We shall apply this lemma to (random) discrete {\L}ukasiewicz paths $\Xsf^n$, which satisfy $\Csf^n = 0$ for every $n$, and which converge once rescaled towards some limit process $\Xsf$. Then the continuous part $\Csf$ in the latter arises as the limit of the contribution of the jumps in $\Xsf^n$ which are smaller than $\delta$ (after scaling) when first $n\to\infty$ and then $\delta\downarrow0$.
\end{rem}

\subsection{Convergence of deterministic looptrees}

For the rest of this section, we assume that $\Xsf$ satisfies the pure jump property~\eqref{eq:PJ}. In this case, the geometry of the looptree $\Loop(\Xsf)$ and the Gaussian labels on it are very constraint and the function $\Loop$ is actually continuous at $\Xsf$. The next result is very close to~\cite[Theorem~4.1]{CK14}; the proof there is completely deterministic and only uses the key property~\eqref{eq:PJ}, see~\cite[Corollary~3.4]{CK14}.

\begin{prop}\label{prop:convergence_looptrees_saut_pur}
Suppose that $\Xsf^n \to \Xsf$ for the Skorokhod topology and that $\Xsf$ satisfies~\eqref{eq:PJ}. Then
\[\Loop(\Xsf^n) \cv \Loop(\Xsf)\]
in the Gromov--Hausdorff--Prokhorov topology.
\end{prop}

\begin{proof}
The argument follows the lines of the preceding proof. 
We argue by contradiction and therefore suppose that there exist $\varepsilon > 0$, a pair $0 \le s < t \le 1$, and sequences $s_n \to s$ and $t_n \to t$ such that for every $n \ge 1$ (limiting ourselves to a subsequence if necessary),
\[\big|d_{\Loop(\Xsf^n)}(s_n, t_n) - d_{\Loop(\Xsf)}(s, t)\big| > \varepsilon.\]
To ease notation, we assume that $s_n \preceq t_n$ for every $n$, and so $s \preceq t$; the general case can be adapted in a straightforward way by cutting at the last common ancestor.
Recall that $\Xsf_{t-} - \Xsf_{s-} = \sum_{s \preceq r \prec t} \Rsf^t_r$; let us fix $\eta > 0$ small enough, so that, by the property~\eqref{eq:PJ}, it holds $\sum_{s \preceq r \preceq t} \Rsf^t_r \ind{\Rsf^t_r \le \eta} \le \varepsilon/2$. Note that the set $\{s \preceq r \prec t : \Rsf^t_r > \eta\}$ is finite, and properties of the Skorokhod topology imply that the values $\Delta \Xsf_r$ and $\Rsf^t_r$ for such $r$'s are approximated by similar quantities in $\Xsf^n$. Therefore the total contribution of these large cycles to $d_{\Loop(\Xsf^n)}(s_n, t_n)$ converges to that in $d_{\Loop(\Xsf)}(s, t)$. 

Now we note that $\Xsf_{t_n}-\Xsf_{s_n-}$ converges either to $\Xsf_{t-}-\Xsf_{s-}$ or to $\Xsf_t-\Xsf_{s-}$. In any case we have that
\[\limsup_{n \to \infty} \Xsf_{t_n}-\Xsf_{s_n-} - \sum_{s_n \preceq r \prec t_n} \Rsf^{t_n}_r \ind{\Rsf^{t_n}_r > \eta}
\le \sum_{s \preceq r \preceq t} \Rsf^t_r - \sum_{s \preceq r \prec t} \Rsf^t_r \ind{\Rsf^t_r > \eta}
\le \varepsilon/2 
.\]
Precisely, the term on the left converges to the one in the middle where $\Rsf^t_t$ may or may not be included.
The left hand side bounds above the contribution of the small cycles to $d_{\Loop(\Xsf^n)}(s_n, t_n)$ so this yields a contradiction and $d_{\Loop(\Xsf^n)}$ converges to $d_{\Loop(\Xsf)}$ uniformly on $[0,1]^2$, which implies the Gromov--Hausdorff--Prokhorov convergence.
\end{proof}

\begin{cor}\label{cor:convergence_looptrees_discrets_saut_pur}
For every $n \ge 1$, let $\Lsf^n$ be a looptree with $n$ edges and let $\Xsf^n$ be its {\L}ukasiewicz path. Suppose that there exists $a_n\to\infty$ such that $(a_n^{-1} \Xsf^n_{n t})_{t \in [-1,1]} \to \Xsf$ for the Skorokhod topology and that $\Xsf$ satisfies~\eqref{eq:PJ}. Suppose in addition that the convergence~\eqref{eq:hyp_proportion_feuilles} holds. Then
\[a_n^{-1} \Lsf^n \cv \Loop(\Xsf)\]
in the Gromov--Hausdorff--Prokhorov topology.
\end{cor}

\begin{proof}
We already observed that $\Loop(\Xsf^n)$ is the continuum looptree obtained by replacing each edge of $\Lsf^n$ by a segment with unit length, then $\Loop(a_n^{-1} \Xsf^n_{n \cdot})$ describes the same looptree when the distances have been multiplied by $a_n^{-1}$, which therefore lies at Gromov--Hausdorff distance less than $a_n^{-1}$ from $a_n^{-1} \Lsf^n$. Finally  the convergence~\eqref{eq:hyp_proportion_feuilles} shows that the uniform probability measure on the vertices is close to the one on the corners, so the rescaled discrete and continuum looptrees lie at Gromov--Hausdorff--Prokhorov distance $O(a_n^{-1})$ from each other.
The claim then immediately follows from Proposition~\ref{prop:convergence_looptrees_saut_pur}.
\end{proof}

\begin{rem}\label{rem:convergence_Luka_pas_suffisante}
The sole functional convergence fails to dictate the behaviour of the looptrees in the following two cases:
\begin{enumerate}
\item In the presence of a nontrivial continuous part.
Even in the extreme case where $\Xsf$ is continuous, so for every $t \in [0,1]$ one has $\Jsf_t = 0$, one needs to look closer to deduce the convergence of looptrees towards the tree $T_\Xsf$. 
This is studied in Sections~\ref{sec:degres_prescrits} and~\ref{sec:convergence} in a probabilistic setting.

\item Also, for looptrees as defined in~\cite{CK14}, i.e.~in the middle of Figure~\ref{fig:looptree_intro}, without merging each right-most offspring in the tree with its parent. This is due to the fact that in this looptree, the graph distance between such a pair equals $1$, whereas the {\L}ukasiewicz path at the time of visit of the last offspring is exactly at the same level as the left limit at the time of visit of the parent. In~\cite[Theorem~4.1]{CK14}, the authors therefore also require that the maximal height of the tree is small compared to $a_n$. It can be reduced to requiring that the maximum over all vertices of the tree of the number of its ancestors which are the right-most offspring of their parent is small compared to $a_n$. One easily checks that this assumption is necessary, at least when $a_n = o(n)$, since in this case, one can add a chain of length $a_n$ of cycles with length two (individuals with only one offspring in the tree) to the root with no effect on the convergence of the {\L}ukasiewicz path. Again this is discussed in the probabilistic setting of Sections~\ref{sec:degres_prescrits} and~\ref{sec:convergence}.
\end{enumerate}
\end{rem}

\subsection{Random labels on deterministic looptrees}
\label{ssec:labels_aleatoires}

Let us next consider random labels on the looptrees and define $Z$ as in~\eqref{eq:def_labels_browniens}. We easily deduce the following result from Proposition~\ref{prop:convergence_looptrees_saut_pur}.

\begin{cor}\label{cor:convergence_labels_continus_saut_pur}
Suppose that $\Xsf^n \to \Xsf$ for the Skorokhod topology, that $\Xsf$ satisfies~\eqref{eq:PJ}, and that there exist $c > 0$ and $K>0$ such that for every $s, t \in [0,1]$, for every $n$ large enough,
\[d_{\Loop(\Xsf^n)}(s,t) \le K \cdot |t-s|^{c}.\]
Define $Z^n$ as in~\eqref{eq:def_labels_browniens} with $\Xsf^n$ instead of $\Xsf$, then the convergence in distribution
\[Z^n \cvloi Z\]
holds for the uniform topology.
\end{cor}

\begin{proof}
Combining our assumption with Equation~\eqref{eq:label_moment_Kolmogorov_looptrees}, we deduce tightness of the sequence $(Z^n)_n$ from Kolmogorov's criterion. The convergence in distribution then follows from Proposition~\ref{prop:convergence_looptrees_saut_pur} which ensures the convergence of the covariance function of our Gaussian processes.
\end{proof}

We shall be more interested in the convergence of labels on discrete looptrees such as defined in Section~\ref{ssec:codage_chemins}. Fix $n \ge 1$, let $\Lsf^n$ be a looptree with $n$ edges, and let $\Xsf^n$ be its {\L}ukasiewicz path. Let also $(\xi_k)_{k \ge 0}$ be i.i.d. copies of a random variable $\xi$ supported by $\Z$, which 
is centred and has variance $\Var(\xi) > 0$. 
Assume finally that for every $\ell \ge 1$, the event $\{\xi_1 + \dots + \xi_\ell = 0\}$ has a nonzero probability and for each $0 \le i \le n$ let $\ell^n_i = \Delta \Xsf^n_i$ and then let $(\xi^{i}_1,\dots,\xi^{i}_{\ell^n_i})$ have the conditional law of $(\xi_1, \dots, \xi_{\ell^n_i})$ under $\P(\,\cdot\mid \xi_1 + \dots + \xi_{\ell^n_i} = 0)$. Assume that these bridges are independent. We then define the random label process as in~\eqref{eq:def_processus_labels_discret}, namely
\begin{equation}\label{eq:def_processus_labels_discret_random}
Z^n_j \coloneqq \sum_{i=0}^{j-1} \sum_{k=0}^{\Delta \Xsf^n_i} \left(\xi^{i}_1+\dots+\xi^{i}_{k}\right) \ind{\inf_{t \in [i, j]} \Xsf^n_{\floor t} - \Xsf^n_{i-}=\Delta \Xsf^n_i-k},
\end{equation}
The process $Z^n$ is then linearly interpolated between integer times. Note that $\E[\xi^{i}_{k}] = 0$ for each $i$ and $k$ and so $\E[Z^n_j] = 0$ for each $j$. Also $Z^n_0 = Z^n_n = 0$ almost surely.

In the sequel, for every $0 \le j \le \ell$, we denote by $\Xi^\ell_j$ a random variable with the law of $\xi_1+\dots+\xi_j$ under under $\P(\,\cdot\mid \xi_1 + \dots + \xi_{\ell} = 0)$.
As proved in~\cite[Lemma~10 \&~11]{Bet10}, if $\xi$ admits a moment of order $q_0>2$, then:
\begin{enumerate}
\item\label{item:moments_ponts}
For every $q \in [2,q_0]$, there exists a constant $K_q > 0$ such that for every $0 \le j \le \ell$,
\[\E\left[\big|\Xi^\ell_j\big|^q\right] \le K_q \min\{j, \ell-j\}^{q/2}.\]

\item\label{item:convergence_pont_brownien}
If $b = (b_t)_{t \in [0,1]}$ denotes a standard Brownian bridge with duration $1$, then we have the convergence in distribution:
\[(\Var(\xi) \ell)^{-1/2} \Xi^\ell_{\floor{\ell \cdot}} \cvloi[\ell] b.\]
\end{enumerate}

Recall that for our applications to random plane maps, in order to equip the looptree with an independent uniformly random good labelling, $\xi$ shall have the centred geometric distribution $\sum_{i \ge -1} 2^{-i-2} \delta_i$. The latter admits moments of all order and has variance $\Var(\xi) = 2$.

In the next theorem, we let $\Lsf^n$ be a looptree with $n$ edges and let $\Xsf^n$ be its {\L}ukasiewicz path. For every pair of integers $0 \le i \le j \le n-1$, we let $d_{\Lsf^n}(i,j)$ denote the graph distance between the (vertices incident to the) corners $c_i$ and $c_j$ in $\Lsf^n$.
We also define $Z^n$ as in~\eqref{eq:def_processus_labels_discret_random} and $Z$ as in~\eqref{eq:def_labels_browniens}.

\begin{prop}\label{prop:convergence_labels_saut_pur}
Suppose that there exists $a_n \to \infty$ such that $(a_n^{-1} \Xsf^n_{n t})_{t \in [-1,1]} \to \Xsf$ for the Skorokhod topology and that $\Xsf$ satisfies~\eqref{eq:PJ}.
Suppose also that $\E[|\xi|^{q_0}] < \infty$ for some $q_0 > 2$.
Then the convergence in distribution
\[\left((\Var(\xi) a_n)^{-1/2} Z^{n}_{n t}\right)_{t \in [0,1]} \cvloi Z,\]
holds in the sense of finite dimensional marginals.
It holds for the topology of uniform convergence if furthermore 
there exist $c > 2/q_0$ and $K>0$ such that for every $n \ge 1$ and every $0 \le s \le t \le 1$, it holds
\begin{equation}\label{eq:hypothese_Kolmogorov_tension_distance_loop}
d_{\Lsf^n}(\floor{n s}, \floor{n t}) \le K \cdot a_n \cdot |t-s|^{c}.
\end{equation}
\end{prop}

\begin{proof}
This can be adapted from previous works which focus on the particular case where $\P(\xi=i) = 2^{-(i+2)}$ for $i \ge -1$. Indeed, the only properties of the random bridges $\Xi$'s which are used are the two listed previously.
With these two properties, the convergence in the sense of finite-dimensional marginals follows as in~\cite[Proposition~7]{LGM11} which is written for stable L\'evy processes but their argument is general. Let us note that there the authors work with the two-type mobile of~\cite{BDG04}, we refer alternatively the reader to the proof of Proposition~4.4 in~\cite{Mar18a} which adapted the work of~\cite{LGM11} with the tree version of our looptree (still for stable processes).
As for Proposition~\ref{prop:convergence_looptrees_saut_pur}, the only key result of the stable L\'evy processes used there is~\eqref{eq:PJ}, see~\cite[Equation~32]{LGM11}.
The proof is similar to that of Proposition~\ref{prop:convergence_looptrees_saut_pur}: by~\eqref{eq:PJ} only a finite number of large increments of $a_n^{-1} \Xsf^n$ will contribute, up to an arbitrarily small error, and the labels around these cycles converge to Brownian bridges by Property~\ref{item:convergence_pont_brownien} of $\Xi^\ell$.

The convergence for the uniform topology then requires a tightness argument. We claim that, similarly to the continuum setting, for every $q \in [2, q_0]$, there exists $C_q > 0$ such that for every $n \ge 1$ and every $0 \le s \le t \le 1$,
\begin{equation}\label{eq:moment_labels_looptrees_discrets}
\Es{\big|Z^{n}_{\floor{n s}} - Z^{n}_{\floor{n t}}\big|^q} \le C_q \cdot d_{\Lsf^n}(\floor{n s}, \floor{n t})^{q/2}.
\end{equation}
Combined with our assumption~\eqref{eq:hypothese_Kolmogorov_tension_distance_loop} and since $q_0c/2 > 1$, this bound implies tightness by applying Kolmogorov's criterion. 
The bound~\eqref{eq:moment_labels_looptrees_discrets} follows from Property~\ref{item:moments_ponts} of $\Xi^\ell$ as in~\cite[Equation~22]{Mar18b} and the next few lines there, where a similar bound is shown with a larger quantity $\widehat{d}_{\Lsf^n}$ instead of $d_{\Lsf^n}$. Recall that for $0\le i < j \le n$, we follow the chain of loops between the $i$'th and the $j$'th corner of $\Lsf^n$ and then the sum over all these loops of the shortest between their left and right length equals $d_{\Lsf^n}(i,j)$; then $\widehat{d}_{\Lsf^n}$ is constructed similarly, but following always the loops of the left branch to their right and vice versa. See the beginning of Section~\ref{ssec:tension} for a more detailed description. The argument from~\cite{Mar18b} still applies after this change.
\end{proof}

In Proposition~\ref{prop:tension_Holder_looptrees} below, we shall prove that~\eqref{eq:hypothese_Kolmogorov_tension_distance_loop} holds with high probability for random {\L}ukasiewicz paths under only an ``exchangeability'' property, for any $c \in (0,1/2)$ so this requires $q_0>4$ as in Theorem~\ref{thm:convergence_labels_intro} in the introduction.
We shall also prove in the same setting the convergence in probability of $\lambda^n$ to the identity as in~\eqref{eq:hyp_proportion_feuilles}.
By taking $\xi$ to have the centred geometric distribution, so $\Var(\xi) = 2$, we may therefore apply Lemma~\ref{lem:profil_general} and Lemma~\ref{lem:tension_cartes_general} (on a set of arbitrarily high probability) to deduce invariance principles for the random maps associated with these random looptrees equipped with uniformly random good labellings.
We shall further argue in Section~\ref{sec:convergence} that for this model the convergence of looptrees and labels still hold when the random limit path $X$ does not satisfy~\eqref{eq:PJ}.

\section{Random models with prescribed degrees}
\label{sec:degres_prescrits}

In the previous section we obtained deterministic results on labelled looptrees and plane maps, when their coding paths converge towards a limit which satisfies the assumption~\eqref{eq:PJ}. In the next sections we extend them in a random setting, when the random limit paths may not satisfy this assumption. 
Let us first formally introduce the model, then discuss some tightness results, and finally present in Section~\ref{ssec:epine} a key spinal decomposition which will be used in Section~\ref{sec:convergence} to identify the scaling limits of the discrete objects.

\subsection{Model and notations}
\label{ssec:model_degres_prescrits}

Let us extend the models presented in the introduction by considering more generally looptrees and maps \emph{with a boundary} by seeing the face to the right of the root edge as an external face; maps without boundary can be seen as maps with a boundary with length $2$ by doubling the root edge.
Fix henceforth $n\ge 1$, an integer $\varrho_n \ge 1$, and a finite sequence of integers
\[\Degf_n = (\degf_{n,i})_{1 \le i \le n+\varrho_n}
\quad\text{such that}\quad
\degf_{n,1} \ge \dots \ge \degf_{n,n+\varrho_n} \ge 0
\quad\text{and}\quad
\sum_{i =1}^{n+\varrho_n} \degf_{n,i} = n.\]
For $k \ge 0$, define then $\degf_n(k) = \#\{i \ge 1 : \degf_{n,i} = k\}$. Let also
\[E_n = n+\varrho_n,
\qquad
F_n = \sum_{k \ge 1} \degf_n(k),
\qquad\text{so that}\qquad
\degf_n(0) 
= E_n-F_n
= \varrho_n + \sum_{k \ge 1} (k-1) \degf_n(k).\]
Note that $\degf_{n,i}$ is zero if and only if $i > F_n$. In order to avoid trivialities, we shall assume that both $\degf_n(0)$ and $F_n$ tend to infinity as $n \to \infty$.

Consider first the set of ordered plane forests with $\varrho_n$ trees and which have exactly $\degf_n(k)$ individuals with $k$ offspring for every $k \ge 0$. By the above relations this set is nonempty and every such forest has $n$ edges, $\degf_n(0)$ leaves, and $F_n$ internal vertices. Note that $\degf_{n,i}$ equals the $i$'th largest offspring number.
We shall canonically view such a forest as a single tree with $E_n$ edges by adding an extra root vertex connected to the root of each tree and preserving the order. 
Then consider the set of looptrees of such forests (merging each right-most offspring with its parent): the root cycle has length $\varrho_n$ and there are exactly $\degf_n(k)$ other cycles with length $k$ for every $k \ge 1$. For every $i \le F_n$ the quantity $\degf_{n,i}$ equals the length of the $i$'th largest nonroot cycle. Note that any such looptree has $E_n$ edges, $F_n+1$ cycles in total, and so it has $\degf_n(0)$ vertices.
By endowing such a looptree with a good labelling, the image by the bijection from Section~\ref{ssec:la_bijection} is a pointed negative map with boundary length $2\varrho_n$ and with exactly $\degf_n(k)$ inner faces with degree $2k$ for every $k \ge 1$. Such a map has $F_n$ inner faces, $E_n$ edges, as well as $\degf_n(0)+1$ vertices. This is summarised in Table~\ref{fig:tableau} below.

\begin{table}[!ht] \centering
\includegraphics[page=28, height=11\baselineskip]{bijections}
\caption{The key quantities in our graphs.}
\label{fig:tableau}
\end{table}

Let us make a useful observation. Let us sample a non pointed map with a boundary uniformly at random amongst those with such degree statistics and then distinguish one of the $\degf_n(0)+1$ vertices uniformly at random and independently of the map; exactly $\varrho_n$ edges on the boundary, oriented so that the external face lie to their right, are negative; sample one of these oriented edges uniformly at random and let it replace the root edge. This results in a pointed negative map and the latter precisely has the uniform distribution amongst those with the previous degree statistics. Therefore it is equivalent to study the geometry of a non pointed map and of such a negative pointed one, for which we may use the bijection from Lemma~\ref{lem:bijection}.

Throughout this section, we let the dependence in $\varrho_n$ and in the sequence $\Degf_n$ be implicit in order to lighten the notation and we denote by $T^n$ a random forest sampled uniformly at random amongst those with $\varrho_n$ trees and offspring numbers $\Degf_n = (\degf_{n,i})_{1 \le i \le E_n}$, we denote by $LT^n$ the corresponding random looptree, and by $X^n = (X^n_t ; 0 \le t \le E_n)$ its {\L}ukasiewicz path.
We also equip this looptree with a random labelling and denote by $Z^n = (Z^n_t ; 0 \le t \le E_n)$ the associated label process, thus given by~\eqref{eq:def_processus_labels_discret_random} with some label increment distribution $\xi$. When $\P(\xi=i)=2^{-(i+2)}$ for $i \ge -1$, this random labelling is a uniform good labelling and we let $(M^n,v^n_\star)$ denote the associated negative pointed map by Lemma~\ref{lem:bijection}.

Recall finally from Section~\ref{ssec:codage_chemins} the notation $\lambda^n(t) \in \{0, \dots, E_n-1\}$ with $0 \le t \le \degf_n(0)$, for the index such that the vertex $v^n_{\lambda(\ceil{t})}$ is the $\ceil{t}$'th vertex fully visited in the contour of the looptree, so $(v^n_{\lambda^n(0)}, \dots, v^n_{\lambda^n(\degf_n(0)-1)})$ lists the vertices without redundancies. Then~\cite[Lemma~2]{Mar19} with $A = \{-1\}$ provides our assumption~\eqref{eq:hyp_proportion_feuilles} by showing that
\begin{equation}\label{eq:proportion_feuilles_random}
\left(\frac{\Lambda^n(E_n t)}{\degf_n(0)}, \frac{\lambda^n(\degf_n(0) t)}{E_n}\right)_{t \in [0,1]} \cvproba \left(t, t\right)_{t \in [0,1]}
\end{equation}
for the uniform topology. We henceforth concentrate on the convergence of the processes $X^n$ and $Z^n$, and the spaces $LT^n$ and $M^n$.

\subsection{Tightness results}
\label{ssec:tension}

Fix $n \ge 1$ and recall that $LT^n$ denotes a looptree sampled uniformly at random amongst those with cycle lengths $\varrho_n$ and $\Degf_n$. Recall that its has $\degf_n(0)$ vertices and $E_n$ edges in total.
Theorem~6 in~\cite{Mar19} states that the label process associated with a uniform random good labelling on $LT^n$ is always tight; we explain below that behind this result hides tightness of the distances on $LT^n$.

Fix two corners, say $c_i$ and $c_j$, of $LT^n$ with $i<j$. Recall that their looptree distance $d_{LT^n}(i,j)$ is obtained by considering all the cycles between them and summing the shortest between the left and right side. Let us modify this distance by defining $\widehat{d}_{LT^n}(i,j) \ge d_{LT^n}(i,j)$ as the length of the path between $c_i$ and $c_j$ which always follows the right side of the ancestral cycles of $c_i$ and, if $c_j$ does not belong these ancestral cycles, always follows the left side of the ancestral cycles of $c_j$, and finally follows the top part of their common ancestral cycle.
We extend bo $d_{LT^n}$ and $\widehat{d}_{LT^n}$ to $[0,E_n]$ by linear interpolation.

\begin{prop}\label{prop:tension_Holder_looptrees}
For every $\varepsilon > 0$, there exists $K > 0$ such that for every $n$ large enough, 
\[\P\left(\sup_{s \ne t} \frac{\widehat{d}_{LT^n}(E_n s, E_n t)}{\sigma_n \cdot |t-s|^{1/2-\varepsilon}} \le K\right) \ge 1-\varepsilon.\]
Consequently, the sequence of processes $(\sigma_n^{-1} d_{LT^n}(E_n s, E_n t))_{s,t \in [0,1]}$ is tight for the uniform topology.
\end{prop}

\begin{proof}
This result in implicit in the proof of Theorem~6 in~\cite{Mar19} to which we refer for details and only sketch the argument here. Note that we may assume that the number $E_n-\degf_n(1)$ of edges which do not belong to single loops tends to infinity as otherwise, if it is uniformly bounded say then so is the number of vertices and the diameter of the looptree as well as $\sigma_n$.

Let us fix $s < t$ and assume to ease notation that $E_n s$ and $E_n t$ are both integers. We shall abuse notation and for an integer $i \in \{0, \dots, E_n\}$, we shall talk about the corner $i$ or the vertex $i$ to mean the corner $c_i$ or its incident vertex.
Then let $r \in [s, t]$ be such that $E_n r$ is also an integer and such that the vertex $E_n r$ is the first one that belongs to an ancestral cycle of $E_n t$ but not of $E_n s$. If $E_n t$ itself belongs to an ancestral cycle of $E_n s$, then we let $r=t$. Let $X^n$ denote the {\L}ukasiewicz path associated with $LT^n$, then we have
\[\widehat{d}_{LT^n}(E_n s, E_n r) = X^n_{E_n s - 1} - \inf_{q \in [s,r]} X^n_{E_n q - 1}.\]
According to~\cite[Proposition~3]{Mar19}, for $\varepsilon > 0$ fixed, there exists $C > 0$ such that the right hand side is smaller than $C \cdot \sigma_n \cdot |t-s|^{(1-\varepsilon)/2}$ uniformly for $0 \le s < t \le 1$ with a probability at least $1-\varepsilon$.
It only remains to prove a similar bound for $\widehat{d}_{LT^n}(E_n r, E_n t)$.

The argument is provided by the proof of Theorem~6 in~\cite{Mar19}. One would like to proceed symmetrically by inverting the orientation of the edges of the looptree and thus following the contour from right to left, however there is a break of symmetry in the construction of the looptree version of a tree since we contract each edge from an inner vertex of the tree to its right-most offspring. This is consistent with the {\L}ukasiewicz path in which the value at the instant of visit of the right-most offspring equals the value of the left limit at the instant of visit of its parent, but if we were to consider the {\L}ukasiewicz path of the mirror tree, we would miss the left-most edge on each ancestral cycle of $E_n t$ and only count the other ones.

Note that an ancestral cycle must have length at least two. There are two cases when following such a cycle on its left. Suppose first that we traverse $i \ge 2$ edges, since $i \le 2(i-1)$ and the quantity $i-1$ is counted by the mirror {\L}ukasiewicz path, then the previous argument applies and this only replaces the constant $C$ above by $2C$.
However if we traverse only the left-most edge of the cycle, then this case really is lost when considering the mirror {\L}ukasiewicz path. It is shown in~\cite[Lemma~5]{Mar19} that with high probability the proportion of such cycles is bounded away from $1$ on all ancestral lines consisting of at least some constant times $\log(E_n-\degf_n(1))$ cycles. 
Then on the one hand the argument for the left branch applies to these ancestral lines, and only the constant $C$ is modified, and on the other hand the total contribution of the left-most edge of the ancestral lines with less than $\log(E_n-\degf_n(1))$ cycles is anyway small.

Finally, since $\widehat{d}_{LT^n} \ge d_{LT^n}$, then the latter satisfies the claim as well. We then infer from the triangle inequality that for every $\varepsilon > 0$, there exists $K > 0$ such that for every $n$ large enough, 
\[\P\left(\sup_{(s,t) \ne (s',t')} \frac{|d_{LT^n}(\floor{E_n s}, \floor{E_n t}) - d_{LT^n}(\floor{E_n s'}, \floor{E_n t'})|}{\sigma_n \cdot \|(s,t) - (s',t')\|^{1/2-\varepsilon}} \le K\right) \ge 1-\varepsilon.\]
Tightness of the rescaled looptree distance then follows.
\end{proof}

This bound combined with that from~\eqref{eq:moment_labels_looptrees_discrets} shows a similar result for the label process which extends~\cite[Theorem~6]{Mar19} which focused on the particular case of the labels associated with maps. Recall that $Z^n$ is defined in~\eqref{eq:def_processus_labels_discret_random}.

\begin{cor}\label{cor:tension_Holder_labels}
Suppose that $\E[|\xi|^{q_0}] < \infty$ for some $q_0 > 4$.
For every $\varepsilon > 0$, there exists $K > 0$ such that for every $n$ large enough, 
\[\P\left(\sup_{s \ne t} \frac{|Z^n_{E_n s} - Z^n_{E_n t}|}{\sigma_n^{1/2} |t-s|^{1/4-\varepsilon}} \le K\right) \ge 1-\varepsilon.\]
Consequently, the sequence of processes $(\sigma_n^{-1/2} Z^n_{E_n t})_{t \in [0,1]}$ is tight for the uniform topology.
\end{cor}

Recall the convergence~\eqref{eq:proportion_feuilles_random}; then jointly with this corollary (and say Skorokhod's representation theorem), this allows us to apply Lemma~\ref{lem:profil_general} and Lemma~\ref{lem:tension_cartes_general}. We deduce that from every sequence of integers one can extract a subsequence along which $\sigma_n^{-1/2} Z^n$ converges in distribution to a continuous limit $Z$. When $\xi$ has the centred geometric distribution, one can extract a further subsequence along which the associated rescaled map $\sigma_n^{-1/2} M^n$ converges in distribution.
In Section~\ref{sec:convergence} we characterise the subsequential limits of $Z^n$ and $LT^n$; the key technical input for this is Lemma~\ref{lem:consequence_epine} in the next subsection.

Before moving to this topic, let us consider looptrees $\overp{LT}^n$ as defined in~\cite{CK14}, i.e.~without merging the right-most offspring of each individual with its parent (so the cycles are longer by $1$).
Typically, our random models satisfy $\limsup_{n\to\infty} E_n^{-1} \degf_n(1) < 1$ so $\overp{\sigma}_n$ below can be replaced by $\sigma_n$.

\begin{cor}\label{cor:tension_looptrees_CK}
For every $\varepsilon > 0$, there exists $K > 0$ such that for every $n$ large enough, 
\[\P\left(\sup_{s \ne t} \frac{d_{\overp{LT}^n}(E_n s, E_n t)}{\overp{\sigma}_n \cdot |t-s|^{1/2-\varepsilon}} \le K\right) \ge 1-\varepsilon
\qquad\text{where}\qquad
\overp{\sigma}_n = \sigma_n \frac{E_n}{E_n - \degf_n(1)}.\]
Consequently, the sequence of processes $(\overp{\sigma}_n^{-1} d_{LT^n}(E_n s, E_n t))_{s,t \in [0,1]}$ is tight for the uniform topology. Similarly, if one constructs random labels on $\overp{LT}^n$ using random $\xi$-bridges with $\E[|\xi|^{q_0}] < \infty$ for some $q_0 > 4$, then the corresponding label process is tight as well, once rescaled by a factor $\overp{\sigma}_n^{1/2}$.
\end{cor}

\begin{proof}
Let us briefly sketch the proof. 
First, the difference with Proposition~\ref{prop:tension_Holder_looptrees} is that the right-most edge along each cycle is not counted there; the argument in the previous proof allows to control this difference for cycles with length at least $3$ in this looptree, but cycles with length exactly $2$ are really lost when identifying the single child with its parent. 
Therefore, if one first samples $\overp{LT}^n$ and then removes from it all the cycles with length $2$, then the resulting looptree satisfies the claim from Proposition~\ref{prop:tension_Holder_looptrees}, with the scaling factor $\sigma_n$.
One then needs to plug back these cycles. This was done for the trees in~\cite[Lemma~8]{BM14} to which we refer for a detailed argument. Without these $\degf_n(1)$ cycles, we have lost as many vertices, so it remains $E_n-\degf_n(1)$ of them and we distribute these cycles above these vertices, according to a multinomial law with parameters $\degf_n(1)$ and $(E_n-\degf_n(1))^{-1}, \dots, (E_n-\degf_n(1))^{-1}$. Such a distribution is well concentrated around its mean (namely the Chernoff bound applies), which allows to show that the distances are very close to get multiplied by $1+\degf_n(1) / (E_n - \degf_n(1))$, see~\cite[page~303]{BM14} for details.
\end{proof}

Note that for the looptrees $\overp{LT}^n$, the looptree distance is lower bounded by the tree distance, which is not the case for the looptrees $LT^n$ because of the identification of the right-most offspring of each individual with its parent. Then we infer from the preceding result tightness for trees. Precisely, consider the \emph{height process}, defined for every integer $j \in \{0, \dots, E_n\}$ by
\[H^n_j = \#\{i \prec j\} = \#\{0\le i < j : X^n_{i-1} \le \inf_{[i,j]} X^n\},\]
and then linearly interpolated.

\begin{cor}\label{cor:tension_Holder_arbres}
Suppose that $\limsup_{n\to\infty} E_n^{-1} \degf_n(1) < 1$. 
Then for every $\varepsilon > 0$, there exists $K > 0$ such that for every $n$ large enough, 
\[\P\left(\sup_{s \ne t} \frac{|H^n_{E_n s} - H^n_{E_n t}|}{\sigma_n \cdot |t-s|^{1/2-\varepsilon}} \le K\right) \ge 1-\varepsilon.\]
Consequently, the sequence of processes $(\sigma_n^{-1} H^n_{E_n t})_{t \in [0,1]}$ is tight for the uniform topology.
\end{cor}

However, the reader should have in mind that the scaling factor $\sigma_n$ may in general be too large, and the rescaled height process may converge to the null process. Indeed, as shown in~\cite[Proposition~5]{Mar19}, for $U$ independently sampled uniformly at random on $[0,1]$, the height $H^n_{E_n U}$ grows at most like $n/\sigma_n$, and we expect the maximal height of the tree to grow at this rate in most cases. If this is the case, then Corollary~\ref{cor:tension_Holder_arbres} is only useful when $\sigma_n^2$ behaves asymptotically like some constant, say $\sigma^2>0$ times $n$, and in this regime it extends~\cite[Lemma~8]{BM14}.
Let us point out that the maximal height of a leaf can be much larger than $n / \sigma_n$ and  trees are not always tight, as opposed to their looptree versions; see~\cite{Kor15} for examples with subcritical size-conditioned Bienaym\'e--Galton--Watson trees.
The very recent work~\cite[Theorem~8]{BR21} provides a precise tail bound for the maximal height of such trees.

\subsection{A spinal decomposition and its consequences}
\label{ssec:epine}

Lemma~3 in~\cite{Mar19} describes the ancestral lines of finitely many random vertices in the forest $T^n$, from which we deduce here a key lemma which will allow us in Section~\ref{sec:convergence} to determine the geometry of our random (loop)trees and their labels from the behaviour of the {\L}ukasiewicz path.

Let us here abuse notation and identify the vertices of $T^n$ with their instant of visit in depth-first search order, which belongs to $\{1, \dots, E_n\}$. 
Fix $q \ge 1$ and sample $q$ i.i.d. random vertices independent of $T^n$; the set of these vertices and their ancestors inherits a forest structure from $T^n$, let us remove from it its leaves and branchpoints and let us denote by $A^n_q$ the resulting set. The latter is made of a random number $N \in \{1, \dots, 2q-1\}$ of chains of vertices, which we denote respectively by $A^n_{q,1}, \dots, A^n_{q,N}$.
Any individual $i \in A^n_q$ has say, $k^n_i \ge 1$ offspring \emph{in the original forest}, and only one of them, the $j^n_i$'th say, is an ancestor of (or equal to) at least one of the random vertices. In the case $q=1$, with a single random vertex $U_n$ say, the set $A^n_1$ is simply given by $\{i \prec U_n\} = \{i < U_n : X^n_{i-} < X^n_t \text{ for every } t \in (i,U_n)\}$ and each pair $(k^n_i, j^n_i)$ is given by $(\Delta X^n_i, \Delta X^n_i - R^{n, U_n}_i)$, where
\[R^{n,j}_i = \inf_{t \in [i,j]} X^n_t - X^n_{i-}.\]

We let $g$ be a function from $\{(k,j) \in \N^2 : k \ge j \ge 1\}$ to $\R_+$ and define
\[G_n^2 = \sum_{k \ge j \ge 1} \degf_n(k) g(k,j) = \sum_{k \ge 1} \degf_n(k) \sum_{j=1}^k g(k,j).\]
Below, we shall use the notation for any $\delta>0$,
\[G_n^{2,\delta} = \sum_{k \le \delta \sigma_n} \degf_n(k) \sum_{j=1}^k g(k,j)
\qquad\text{and}\qquad
\sigma_n^{2,\delta} = \sum_{k \le \delta \sigma_n} k (k-1) \degf_n(k).\]

\begin{lem}\label{lem:consequence_epine}
Suppose that $g$ satisfies the following property: there exist $c, C > 0$ such that
\[c \min\{j, k-j\} \le g(k,j) \le C (k-1) \qquad\text{for every}\enskip k \ge j \ge 1.\]
For every $q\ge 1$ and every $\varepsilon>0$, we have
\[\lim_{\delta \downarrow 0} \limsup_{n \to \infty}
\P\bigg(\sup_{1 \le m \le N} \bigg| \frac{\sigma^{2,\delta}_n}{G^{2,\delta}_n} \sum_{i \in A^n_{q,m}} \ind{\Delta X^n_i \le \delta \sigma_n} g(k^n_i, j^n_i) - \frac{\sigma^{2,\delta}_n}{n} \#A^n_{q,m} \bigg| > \varepsilon \sqrt{\sigma^{2,\delta}_n}\bigg) 
= 0,\]
as soon as $\sigma_n\to\infty$ and $\sigma_n^{-1} \varrho_n$ is uniformly bounded above.
\end{lem}

Note that the function $g(k,j) = k-j$ satisfies the assumptions and moreover $G^{2,\delta}_n = \sigma^{2,\delta}_n/2$ and further in the case $q=N=1$ of a single random vertex $U_n$, the sum in the lemma equals
\[C^{n,\delta}_{U_n} \coloneqq \sum_{i \prec U_n} R^{n,U_n}_i \ind{\Delta X^n_i \le \delta \sigma_n}\]
where the notation is reminiscent of Lemma~\ref{lem:convergence_partie_continue}. Finally $\#A^n_1 = H^n_{U_n}$ is the height of the random vertex. Therefore, applying the lemma twice, we infer that for every $\varepsilon>0$, we have both
\begin{equation}\label{eq:consequence_epine_hauteur}
\lim_{\delta \downarrow 0} \limsup_{n \to \infty}
\P\bigg(\bigg| C^{n,\delta}_{U_n} - \frac{\sigma^{2,\delta}_n}{2 n} H^n_{U_n} \bigg| > \varepsilon \sqrt{\sigma^{2,\delta}_n}\bigg) 
= 0,
\end{equation}
and then, for any function $g$ as in the lemma,
\begin{equation}\label{eq:consequence_epine_Luka}
\lim_{\delta \downarrow 0} \limsup_{n \to \infty}
\P\bigg(\bigg| \frac{\sigma^{2,\delta}_n}{2 G^{2,\delta}_n} \sum_{i \prec U_n} \ind{\Delta X^n_i \le \delta \sigma_n} g(\Delta X^n_i, \Delta X^n_i-R^{n,U_n}_i) - C^{n,\delta}_{U_n} \bigg| > \varepsilon \sqrt{\sigma^{2,\delta}_n}\bigg) 
= 0.
\end{equation}
According to Lemma~\ref{lem:convergence_partie_continue} and Remark~\ref{rem:convergence_partie_continue_Luka}, if the {\L}ukasiewicz path, rescaled by a factor $\sigma_n^{-1}$, converges in distribution to some limit path, then $\sigma_n^{-1} C^{n,\delta}$ converges in distribution towards the continuous part when first $n\to\infty$ and then $\delta\downarrow0$. Along these limits the ratio $\sigma^{2,\delta}_n/\sigma^2_n$ will converge in our main results; therefore Lemma~\ref{lem:consequence_epine} will allow us to derive in Section~\ref{sec:convergence} the convergence of a function averaged along the ancestral line of random vertices. For example, in Section~\ref{ssec:convergence_looptrees_aleatoires}, we apply it to the function $g(k,j) = \min\{j, k-j\}$, in which case $G^{2,\delta}_n$ equals the contribution to the looptree spanned by random vertices of the loops shorter than $\delta \sigma_n$.

For every $k \ge 1$, one easily calculates
\[\sum_{j=1}^k \min\{j, k-j\} = \frac{k^2-\ind{k \in 2\Z+1}}{2} \ge \frac{k(k-1)}{2}.\]
Then the bounds on $g$ imply that
\begin{equation}\label{eq:bornes_G_sigma}
\frac{c}{2} \sigma_n^{2,\delta} \le G_n^{2,\delta} \le C \sigma_n^{2,\delta}
.\end{equation}

We shall need the following straightforward calculations. 
Let $(K^n_{i})_{1 \le i \le n}$ denote a uniform random permutation of the finite sequence consisting of $k \degf_n(k)$ occurrences of the integer $k$ for every $k \ge 1$, and conditionally on this vector, let $(J^n_{i})_{1 \le i \le n}$ be independent random variables, such that each $J^n_{i}$ has the uniform distribution in $\{1, \dots, K^n_{i}\}$.

\begin{lem}\label{lem:calculs_moments_K_J}
Let $g$ be as in Lemma~\ref{lem:consequence_epine}, then for every $\delta>0$, $n \ge 1$, and $h \ge 1$, it holds,
\[\Es{\sum_{i \le h} g(K^n_{i}, J^n_{i}) \ind{K^n_{i} \le \delta \sigma_n}} 
= h \frac{G_n^{2,\delta}}{n}
\qquad\text{and}\qquad
\Var\left(\sum_{i \le h} g(K^n_{i}, J^n_{i}) \ind{K^n_{i} \le \delta \sigma_n}\right)
\le 2 C h \delta \sigma_n \frac{G_n^{2,\delta}}{n}.\]
\end{lem}

\begin{proof}
Let us first denote for every $k$ by $J_k$ a random variable with the uniform distribution on $\{1, \dots, k\}$ and let us denote by $K$ a random variable with the size-biased law $(k \degf_n(k)/n)_{k \ge 1}$. Then by bounding the variance by the second moment, we have
\[\E[g(k,J_k)] = \sum_{j=1}^k \frac{g(k,j)}{k}
\qquad\text{and}\qquad
\Var(g(k,J_k)) 
\le \sum_{j=1}^k \frac{g(k,j)^2}{k}
.\]
Further
\[\Es{\sum_{j=1}^K \frac{g(K,j)}{K} \ind{K \le \delta \sigma_n}}
= \sum_{k \le \delta \sigma_n} \sum_{j=1}^k \frac{g(k,j)}{k} \frac{k \degf_n(k)}{n}
= \frac{G_n^{2,\delta}}{n}
.\]
Also, since $g(k,j) \le C k$ for every pair $j \le k$, then
\[\Es{\sum_{j=1}^K \frac{g(K,j)^2}{K} \ind{K \le \delta \sigma_n}}
\le C \delta \sigma_n \Es{\sum_{j=1}^K \frac{g(K,j)}{K} \ind{K \le \delta \sigma_n}}
\le C \delta \sigma_n \frac{G_n^{2,\delta}}{n}
,\]
and similarly, bounding again the variance by the second moment and then the ratios in one of the two sums by $C$,
\[\Var\left(\sum_{j=1}^K \frac{g(K,j)}{K} \ind{K \le \delta \sigma_n}\right)
\le \Es{C K \sum_{j=1}^K \frac{g(K,j)}{K} \ind{K \le \delta \sigma_n}}
\le C \delta \sigma_n \frac{G_n^{2,\delta}}{n}
.\]
Let $K^{n,\delta}_{i} = K^n_{i} \ind{K^n_{i} \le \delta \sigma_n}$ and similarly $J^{n,\delta}_{i} = J^n_{i} \ind{K^n_{i} \le \delta \sigma_n}$ in order to lighten the notation; let us also set $g(0,0) = 0$. Then we infer first that
\[\Es{\sum_{i \le h} g(K^{n,\delta}_{i}, J^{n,\delta}_{i})} 
= \Es{\Esc{\sum_{i \le h} g(K^{n,\delta}_{i}, J^{n,\delta}_{i})}{(K^n_{i})_{1 \le i \le h}}} 
= h \frac{G_n^{2,\delta}}{n}.\]
Furthermore, the $J^{n,\delta}_{i}$'s are conditionally independent and it is well known that the $K^n_{i}$'s are negatively correlated, and thus so are random variables of the form $f(K^n_{i})$ with a given function $f$; in particular the variance of their sum is bounded above by the sum of their variances. Consequently,
\begin{align*}
\Var\left(\sum_{i \le h} g(K^{n,\delta}_{i}, J^{n,\delta}_{i})\right)
&= \Var\left(\Esc{\sum_{i \le h} g(K^{n,\delta}_{i}, J^{n,\delta}_{i})}{(K^n_{i})_{i \le h}}\right) + \Es{\Var\left(\sum_{i \le h} g(K^{n,\delta}_{i}, J^{n,\delta}_{i}) \;\middle|\; (K^n_{i})_{i \le h}\right)}
\\
&\le \Var\left(\sum_{i \le h} \sum_{j=1}^{K^n_{i}} \frac{g(K^n_{i},j)}{K^n_{i}} \ind{K^n_{i} \le \delta \sigma_n}\right) + \Es{\sum_{i \le h} \sum_{j=1}^{K^n_{i}} \frac{g(K^n_{i},j)^2}{K^n_{i}} \ind{K^n_{i} \le \delta \sigma_n}}
\\
&\le h \left(\Var\left(\sum_{j=1}^{K^n_{1}} \frac{g(K^n_{1},j)}{K^n_{1}} \ind{K^n_{1} \le \delta \sigma_n}\right) + \Es{\sum_{j=1}^{K^n_{1}} \frac{g(K^n_{1},j)^2}{K^n_{1}} \ind{K^n_{1} \le \delta \sigma_n}}\right)
\\
&\le 2 C h \delta \sigma_n \frac{G_n^{2,\delta}}{n}
,\end{align*}
and the proof is complete.
\end{proof}

\begin{proof}[Proof of Lemma~\ref{lem:consequence_epine}]
Let us translate the spinal decomposition for $q$ uniformly random vertices as stated in~\cite[Lemma~3]{Mar19}.
First according to Proposition~4, Proposition~5, and Remark~1 in~\cite{Mar19}, as well as a union bound, if $\sigma_n^{-1} \varrho_n$ is bounded above, then there exist constants $c_1, c_2 > 0$, that depend on $q$ and on the sequences $(\varrho_n)_n$ and $(\Degf_n)_n$, such that for any $n \ge 1$ and any $x \in [1, \infty)$,
\[\P\bigg(\#A^n_q \le x \frac{n}{\sigma_n} \enskip\text{and}\enskip \sum_{i \in A^n_q} (k^n_i - 1) \le x \sigma_n\bigg) \ge 1 - c_1 \e^{-c_2 x}.\]
Since we assume that $\sigma_n\to\infty$, then the bound on $\#A^n_q$ implies in particular that none of the $q$ random vertices is an ancestor of another.
Recall the notation $(K^n_{i})_{1 \le i \le n}$ for a uniform random permutation of the finite sequence consisting of $k \degf_n(k)$ occurrences of the integer $k$ for every $k \ge 1$, and conditionally on this vector, $(J^n_{i})_{1 \le i \le n}$ are independent random variables, such that each $J^n_{i}$ has the uniform distribution in $\{1, \dots, K^n_{i}\}$.
It is proved in~\cite[Lemma~3]{Mar19} that on the above event, for any $b \in \{0, \dots, q-1\}$, 
any $h_0, \dots, h_{b+q}$ with $h_0=0$ and $h_1 + \dots + h_{b+q} \le x n/\sigma_n$,
and any integers $k_{m,i} \ge j_{m,i} \ge 1$ for every $m \in \{1, \dots, q+b\}$ and $i \in \{1, \dots, h_m\}$, 
the probability that the forest reduced to the $q$ random vertices and their ancestors possesses $b$ branch-points and that for every $m$, we have $\#A^n_{q,m} = h_m$ and $(k^n_{m,i}, j^n_{m,i})_{1 \le i \le \#A^n_{q,m}} = (k_{m,i}, j_{m,i})_{1 \le i \le h_m}$ is upper bounded by
\begin{equation}\label{eq:borne_epine}
C' \left(\frac{\sigma_n}{n}\right)^{q+b} \Pr{\bigcap_{m = 1}^{q+b} \bigcap_{i = h_{m-1}+1}^{h_m} \left\{(K^n_{i}, J^n_{i}) = (k_{m,i-h_{m-1}}, j_{m,i-h_{m-1}})\right\}},
\end{equation}
where $C'$ depends on the sequences $(\varrho_n)_n$ and $(\Degf_n)_n$, as well as on $x$ and $q$.

Fix $\varepsilon > 0$, $x\ge 1$, and $h \ge 1$. We deduce from Lemma~\ref{lem:calculs_moments_K_J} and Chebychev's inequality, and then the bound~\eqref{eq:bornes_G_sigma} that
\[\P\bigg(\bigg| \frac{\sigma_n^{2,\delta}}{G_n^{2,\delta}} \sum_{\substack{1 \le i \le h \\ K^n_{i} \le \delta \sigma_n}} g(K^n_{i}, J^n_{i}) - \frac{\sigma_n^{2,\delta}}{n} h\bigg| > \varepsilon \sqrt{\sigma_n^{2,\delta}}\bigg)
\le \left(\frac{\sigma_n^{2,\delta}}{G_n^{2,\delta}}\right)^2 \frac{2 C h \delta \sigma_n G_n^{2,\delta}}{n \varepsilon^2 \sigma_n^{2,\delta}}
\le \frac{4 C h \sigma_n \delta}{c \varepsilon^2 n}
.\]
We infer by first union bounds and then crude bounds that for any $\varepsilon > 0$, $x \ge 1$, and $\delta > 0$,
\begin{align*}
&\P\bigg(\sup_{m \le N} \bigg| \frac{\sigma^{2,\delta}_n}{G^{2,\delta}_n} \sum_{i \in A^n_{q,m}} \ind{\Delta X^n_i \le \delta \sigma_n} g(k^n_i, j^n_i) - \frac{\sigma^{2,\delta}_n}{n} \#A^n_{q,m} \bigg| > \varepsilon \sqrt{\sigma^{2,\delta}_n}\bigg) 
\\
&\le c_1 \e^{-c_2 x} + C' \sum_{b=0}^{q-1} \left(\frac{\sigma_n}{n}\right)^{q+b} \sum_{h_1+\dots+h_{q+b} \le x n/\sigma_n}
\sum_{m \le q+b} \frac{4 C h_m \delta \sigma_n}{c \varepsilon^2 n}
\\
&\le c_1 \e^{-c_2 x} + C' q x^{2q+1} \frac{4 C \delta}{c \varepsilon^2}
.\end{align*}
Given $\varepsilon>0$ arbitrary, one can choose $x \ge 1$ large enough and then $\delta>0$ small enough so that the last line is arbitrarily small when $n$ is large. 
\end{proof}

\section{Random paths with exchangeable increments}
\label{sec:echangeable}

Lemma~\ref{lem:consequence_epine} will allow us in the next section to relate geometric information about random trees, looptrees, and their labels to their {\L}ukasiewicz paths. In this section, we provide more information on the latter and prove some invariance principles.
Let henceforth $X^n$ be the {\L}ukasiewicz path of a looptree sampled uniformly at random with cycle lengths $\varrho_n$ and $\Degf_n$.
Then $X^n$ is sampled uniformly at random in the set of c\`adl\`ag paths on $[0, E_n]$ such that:
\begin{enumerate}
\item $X^n_0 = \varrho_n$ and $\#\{1 \le i \le E_n : \Delta X^n_i = k\} = \degf_n(k)$ for every $k \ge 0$,
\item $X^n_t = X^n_{\floor{t}} - (t-\floor{t})$ for every $t \in [0, E_n]$,
\item $X^n_t > 0$ for every $t \in [0, E_n)$.
\end{enumerate}
As previously, we shall extend $X^n$ to $[-E_n, E_n]$ by setting $X^n_t = 0$ for every $t \in [-E_n, 0)$.
Note that 
\[X^n_{E_n} 
= X^n_0 + \sum_{i=1}^{E_n} \Delta X^n_i - E_n
= \varrho_n + \sum_{i=1}^n \degf_{n,i} - E_n
= 0.\]
A classical, simple way to sample such a random path is to first remove the positivity constraint and then make a conjugation operation also known as cyclic shift or Vervaat transform, as we next recall.

\subsection{The Vervaat transform}

Let us first introduce a variation of the classical conjugation pf paths, also known as Vervaat transform. Our modification is chosen so that the shifted path may have a different value at the starting point compared to the original one, which will be needed in a finite variation regime.

Fix $T > 0$, let $Y$ be a c\`adl\`ag path on $[-T,T]$ with no negative jump, with $Y_t = 0$ for every $t<0$, and let $r \in [0, T]$. Consider the shift operation $\vartheta_r$ defined for every $t \in [0,T]$ by
\begin{equation}\label{eq:shift_fonctions}
(\vartheta_r Y)_t = 
\begin{cases}
Y_{t+r} - Y_{r-} + Y_0 &\text{if } t+r \le T,
\\
Y_{(t+r-)-T} - Y_{r-} + Y_T &\text{if } t+r > T.
\end{cases}
\end{equation}
Set also $(\vartheta_r Y)_t = 0$ for $t \in [-T, 0)$.
One checks that this path c\`adl\`ag; it is obtained by swapping the parts $(Y_s ; 0 \le s \le r)$ and $(Y_s ; r \le s \le T)$ while keeping the value at the right edge and adding the value of the jump at time $r$ to the starting point: $(\vartheta_r Y)_T = Y_T$ and $(\vartheta_r Y)_0 = Y_0 + \Delta Y_r$.
The properties of the Skorokhod's topology immediately imply that for any fixed path $Y$, the function $r \mapsto \vartheta_rY$ is left-continuous; it is right-continuous at $r$ as soon $Y$ is continuous at $r$. 
More generally, in this case, since $Y$ does not vary much in a neighborhood of $r$, then neither does a function close to $Y$ in the Skorokhod's topology.
We leave the proof of the following lemma to the reader.

\begin{lem}\label{lem:shift_continu}
Let $Y$ be a path as above and $r \in [0,T]$ be such that $Y$ is continuous at $r$, then the shift function $\vartheta$ is continuous at the pair $(Y,r)$.
\end{lem}

Now let $\varrho \ge 0$ and suppose that $Y_0 = \varrho$ and $Y_T = 0$.
Let $u \in [0, \varrho]$ and define
\begin{equation}\label{eq:temps_passage}
I_u = \inf\left\{t \in [0,T] : Y_{t-} = u + \inf_{s \in [0,T]} Y_s\right\},
\end{equation}
to which we shall refer by abuse of language as the first passage time of $Y$ at level $u + \inf Y$, although if $Y$ makes a positive jump at time $I_u$ this is incorrect.
Then the \emph{Vervaat transform} of $Y$ at height $u \in [0, \varrho]$ is the c\`adl\`ag path defined on $[-T,T]$ by
\begin{equation}\label{eq:Vervaat}
V_uY = \vartheta_{I_u} Y.
\end{equation}
For every $u \in [0, \varrho]$ the path $V_uY$ satisfies $(V_uY)_0 = \varrho + \Delta Y_{I_u}$ and $(V_uY)_T = 0 = Y_T$, and furthermore $(V_uY)_t > 0$ for every $t \in (0,T)$. 
We say that $Y$ realises a weak local minimum at a time $r$ if there exists $\varepsilon > 0$ such that $Y_s \ge Y_{r-}$ for every $s \in [r-\varepsilon, r+\varepsilon] \cap [0,T]$; if moreover the inequality is strict except maybe for $s=r$, we say that $Y$ realises a strict local minimum at a time $r$. Note that this may include the case when $Y$ has a positive jump at time $r$. Again the denomination is slightly incorrect (one should in fact replace $Y$ by its left-continuous version).

\begin{lem}\label{lem:Vervaat_continue}
Let $Y$ be a c\`adl\`ag path on $[-T,T]$, null on $[-T, 0)$, with no negative jump, and with $Y_0 = \varrho \ge 0$ and $Y_T = 0$.
\begin{enumerate}
\item If there exists a unique $I = I_0 \in [0,T]$ such that $Y_{I-} = \inf Y$, then the function $V_0$ is continuous at $Y$.
\item If $u \in (0, \varrho)$ and $Y$ does not realise a weak local minimum at time $I_u$, then the function $V : (Y,u) \mapsto V_uY$ is continuous at the pair $(Y,u)$. Moreover the shifted path satisfies $(V_uY)_0 = Y_0 = \varrho$.
\end{enumerate}
\end{lem}

\begin{proof}
The first part of the lemma is~\cite[Lemma~3]{Ber01} when furthermore $Y$ is continuous at $I_0$. It extends when $\Delta Y_{I_0} > 0$ thanks to our definition of the shift. Indeed, let $(Y^n)_{n\ge1}$ be a sequence of paths satisfying the same properties as $Y$ and which converges to it in the Skorokhod's topology. Then by uniqueness of $I_0$, the time $I_0^n$ defined as in~\eqref{eq:temps_passage} for $Y^n$ converges to $I_0$ and $Y^n_{I^n_0-} \to Y_{I_0-}$. Moreover there exists an instant $t^n \ge I^n_0$ which also converges to $I_0$ such that both $Y^n_{t^n-} \to Y_{I_0-}$ and $Y^n_{t^n} \to Y_{I_0}$. 
Consequently, the path $Y^n$ shifted at time $I^n_0$ is close when $n\to\infty$ to the same path shifted at time $t^n$, which converges to the path $Y$ shifted at time $I_0$.
For the second claim, since $Y$ does not realise a weak local minimum at time $I_u$, then if $u_n$ converges to $u$, then $I_{u_n}$ converges to $I_u$. Note also that since $Y$ makes no negative jump, then it is continuous at $I_u$. We conclude from Lemma~\ref{lem:shift_continu}.
\end{proof}

\begin{rem}\label{rem:min_locaux}
Let us note that, as any real-valued function, the \emph{times} at which $Y$ realises a strict local minimum are at most countable, and so are the \emph{values} of the weak local minima (since there is at most one such value in each open time interval with rational edges). Therefore, for any $Y$, the second claim holds true for every $u \in (0,\varrho)$ outside a countable set.
\end{rem}

\subsection{Bridges with exchangeable increments}

Let $Y^n$ be a uniformly random path satisfying:
\begin{enumerate}
\item $Y^n_0 = \varrho_n$ and $\#\{1 \le i \le E_n : \Delta Y^n_i = k\} = \degf_n(k)$ for every $k \ge 0$,
\item $Y^n_t = Y^n_{\floor{t}} - (t-\floor{t})$ for every $t \in [0, E_n]$.
\end{enumerate}
Recall that our {\L}ukasiewicz paths $X^n$ further satisfy $X^n_t > 0$ for every $t \in [0, E_n)$. The path $Y^n$ can simply be obtained by taking $(\Delta Y^n_i)_{1 \le i \le E_n}$ to be a uniform random permutation of $(\degf_{n,i})_{1 \le i \le E_n}$. By definition, these increments are exchangeable: their joint law is invariant under permutation. In continuous time, a random process $Y$ on the interval $[0,1]$, is said to have \emph{exchangeable increments} when for every $k \ge 1$, the law of the increments $(Y_{i/k}-Y_{(i-1)/k})_{1 \le i \le k}$ is exchangeable.
According to~\cite[Theorem~16.21]{Kal02}, such random processes which start at $\varrho$ and end at $0$ admit the following representation: for every $t \in [0,1]$,
\begin{equation}\label{eq:representation_echangeable_general}
Y_t = \varrho (1-t) + \theta_0 b_t + \sum_{i \ge 1} \theta_i (\ind{U_i \le t} - t),
\end{equation}
where $b$ is a standard Brownian bridge, the $U_i$'s are i.i.d. uniformly distributed on $[0,1]$, and $\varrho$ and the $\theta_i$'s are (possibly dependent) random real numbers with $\sum_i \theta_i^2 < \infty$ a.s. and the three groups $b$, $(U_i)_i$ and $(\varrho, (\theta_i)_i)$ are independent. 
In our case, $\varrho$ and the $\theta_i$'s are deterministic and nonnegative.

Recall the (abusively called) first passage times $I_u$ for $u \in [0, \varrho]$ from~\eqref{eq:temps_passage}.
In the case $\varrho = 0$, by the work of Knight~\cite{Kni96}, the process $Y$ admits a.s. a unique instant of infimum $I = I_0$.
Then by Lemma~\ref{lem:Vervaat_continue} the Vervaat transform $V_0$ is a.s. continuous at $Y$ and the \emph{excursion} path
$X = V_0Y$
satisfies $X_0 = \Delta Y_I \ge 0$, $X_1=0$, and $X_t>0$ for every $t \in (0,1)$. According to Bertoin~\cite[Section~3.2]{Ber01}, the path $Y$ is continuous at $I$, and so $X_0=0$, as soon as
\begin{equation}\label{eq:variation_infinie_echangeable}
\text{either}\quad \theta_0 > 0,\qquad
\text{or}\quad \sum_{i \ge 1} \theta_i = \infty,
\end{equation}
which is known to be equivalent to requiring that $Y$ has infinite variation paths. On the other hand in the case of finite variation paths, $\Delta Y_I > 0$ a.s. as explained by Miermont~\cite{Mie01}, see Corollary~1 and the discussion before there; precisely $\Delta Y_I$ in this case has the law of a size biased pick from the sequence of jump values $(\theta_i)_{i\ge 1}$.

Let us next consider the case $\varrho > 0$.
Let us mention Proposition~1 in~\cite{CUB15} for a less restrictive sufficient condition for $\Delta Y_I=0$ and Theorem 2 in~\cite{AHUB20} which relates continuity at the infimum and regularity at $0$. We shall not use these results since in this case, we instead sample $U$ uniformly at random on $[0,\varrho]$ and independently of $Y$ and define the \emph{first-passage bridge} version of $Y$ as the random path
\begin{equation}\label{eq:pont_premier_passage_continu}
X = V_U Y = \vartheta_{I_U} Y.
\end{equation}
It is classical that a.s. every local minimum of the Brownian motion is a strict local minimum, so there only are countably many times of local minimum. The proof can be adapted to the path $Y$ whenever $\theta_0>0$. We shall not consider this question when $\theta_0=0$ and only recall from Remark~\ref{rem:min_locaux} that the \emph{values} of the local minima are always countable.
Indeed, we infer by combining Lemma~\ref{lem:Vervaat_continue} and Remark~\ref{rem:min_locaux} that since $U$ has a diffuse law, then a.s. the Vervaat transform $V : (Y,u) \mapsto V_uY$ is continuous at the pair $(Y,U)$ and moreover 
$X_0 = \varrho$, $X_1 = 0$, and $X_t > 0$ for every $t \in (0,1)$.

A key feature is that in both cases, the path $X$ and the random time $I_0$ or $I_U$ are independent, and the latter has the uniform distribution on $[0,1]$, see e.g.~\cite[Equation~3.14]{Ber01}.

Let us mention that when $\varrho = 0$ Theorem~3 in~\cite{AHUB20} provides a sufficient condition, weaker than finite variation paths, for $X = V_0 Y$ to really be the excursion version of $Y$, in the sense that it is the weak limit as $\varepsilon \to 0$ of the process $Y$ conditioned on the event $\{\inf Y > -\varepsilon\}$. When $\varrho > 0$, it is argued under more technical assumptions in~\cite{BCP03} that $X$ defined in~\eqref{eq:pont_premier_passage_continu} is similarly the first-passage bridge version of $Y$.
Again we do not consider this interesting question in general and only content ourselves with this construction of $X$.

Recall from the previous sections that the continuous part $C$ of $X$ as in~\eqref{eq:def_J_et_C} plays a role in the construction of the labelled looptrees and the pure jump case, when this path is constant null has been treated in Section~\ref{sec:convergence_saut_pur}. An important question is then to characterise when $X$ satisfies~\eqref{eq:PJ}.

\begin{lem}\label{lem:PJ_echangeable}
Suppose that $\sum_i \theta_i < \infty$, then the following holds:
\begin{enumerate}
\item Either $\theta_0>0$ and then almost surely, for Lebesgue almost every $t\in[0,1]$, we have $C_t > 0$,
\item Or $\theta_0=0$ and then almost surely $C_t = 0$ for all $t\in[0,1]$.
\end{enumerate}
\end{lem}

\begin{rem}\label{rem:PJ_echangeable_mieux}
We believe that this result can be strengthened into the following directions, although we were enable to prove it: regardless the convergence of the sum $\sum_i \theta_i$, almost surely, for all $t\in[0,1]$,
\begin{enumerate}
\item $C_t = 0$ if and only if $X_t = \min_{s \le t} X_s$ when $\theta_0>0$;
\item $C_t = 0$ when $\theta_0=0$.
\end{enumerate}
Both are known to hold in the case of L\'evy processes, which are mixtures of such processes in which the $\theta_i$'s are random, see Proposition~\ref{prop:Levy_partie_continue_hauteur}.
\end{rem}

\begin{proof}[Proof of Lemma~\ref{lem:PJ_echangeable}]
Since $C$ has continuous paths, the claims are equivalent to showing that if $T$ is an independent random time with the uniform distribution, then $C_T > 0$ a.s. or $C_T = 0$ a.s. according as wether $\theta_0 > 0$.
Recall that $1-T$ can be taken as the time at which we cyclically shift $Y$ to produce $X$, then by time-reversal of $Y$, $C_T$ has the same law as the Lebesgue measure of the range of the supremum process of $Y$, namely $\Leb(\sup_{s \in [0,t]} Y_s; t \in [0,1]\})$.

Now when $\sum_i \theta_i < \infty$, the process $Y$ can be written in the following form
\[Y_t = \varrho + b_t' + Y_t'
\qquad\text{where}\qquad
b_t' = \theta_0 b_t - \bigg(\varrho + \sum_{i \ge 1} \theta_i\bigg) t
\qquad\text{and}\qquad
Y_t' = \sum_{i \ge 1} \theta_i \ind{U_i \le t},\]
so $b'$ is a Brownian bridge with a drift and $Y'$ is a nondecreasing process (and both are independent).
Suppose first that $\theta_0 = 0$, so $b'$ reduces to a negative linear drift. Arguing as in Section~\ref{ssec:approximation_partie_continue}, removing the (finitely many) jumps larger than a small $\delta>0$, the sum of the other jumps can be made arbitrarily small and this sum is certainly larger than or equal to the Lebesgue measure of the range of the supremum process. This proves the second claim.

As for the first claim, using e.g. the law of the iterated logarithm for the Brownian motion combined with standard absolute continuity arguments to transfer it to the bridge version, we have that $\limsup_{t \downarrow 0} (2t \ln\ln 1/t)^{-1/2} b_t = 1$ a.s. In particular the negative drift is negligible and $b'$ reaches positive values at arbitrarily small times and this shows that the Lebesgue measure of the range of the supremum process after removing the jumps larger than $\delta$, does not converge to $0$.
\end{proof}

Let us note that even when $\sum_i \theta_i = \infty$, it holds that $\limsup_{t \downarrow 0} (2t \ln\ln 1/t)^{-1/2} \sum_i \theta_i (\ind{U_i \le t}-t) = 0$ a.s. so the Brownian bridge part dominates at small times, which may eventually lead to the first claim in this regime, but we loose here the monotonicity of $Y'$ which was used in the previous proof.

\subsection{Scaling limits of random paths}

Let $Y^n$ be as above: a c\`adl\`ag path on $[0, E_n]$ started from $Y^n_0 = \varrho_n$, which only jumps at integer times, as a uniform random permutation of the sequence $\Degf_n = (\degf_{n,i})_{1 \le i \le E_n}$, and decreases at unit speed otherwise. 
Independently, sample $U_n$ uniformly at random in $[0, \varrho_n]$ and define $X^n$ as for the Vervaat transform~\eqref{eq:pont_premier_passage_continu} of $Y^n$, except that the first passage time $I_{U_n}$ defined in~\eqref{eq:temps_passage} is replaced by its integer part $\floor{I_{U_n}}$.
Necessarily, the path $Y^n$ is continuous at time $\floor{I_{U_n}}$ and precisely $Y^n$ is linear with slope $-1$ in the open interval $(\floor{I_{U_n}}-1, \floor{I_{U_n}}+1)$ so $X^n_0 = \varrho_n$.
Then $X^n$ has the same law as the {\L}ukasiewicz path of a uniformly random looptree with root cycle with length $\varrho_n$ and other cycle length prescribed by $\Degf_n$.

The convergence criterion from~\cite[Theorem~16.23]{Kal02} can be written as follows. Fix any sequence $a_n \to \infty$; then the rescaled processes $(a_n^{-1} Y^n_{E_n t})_{t \in [-1,1]}$ converge in distribution in the Skorokhod topology to some limit, say $Y$, if and only if there exist $\theta_0 \ge 0$ and $\theta_1 \ge \theta_2 \ge \dots \ge 0$ with $\sum_{i \ge 1} \theta_i^2 < \infty$ such that, as $n\to\infty$,
\[a_n^{-1} \varrho_n \to \varrho,
\qquad
a_n^{-1} \degf_{n,i} \to \theta_i \quad\text{for every}\enskip i \ge 1,
\qquad\text{and}\qquad
a_n^{-2} \sum_{k \ge 0} k (k-1) \degf_n(k) 
\to \sum_{i \ge 0} \theta_i^2.
\]
Furthermore, in this case, the limit $Y$ takes the form~\eqref{eq:representation_echangeable_general} with also $Y_t = 0$ for every $t < 0$.
To be precise and fit the framework of~\cite[Theorem~16.23]{Kal02} we should replace $Y^n$ by a path that at any integer time $i \in\{1, \dots, E_n+1\}$ makes a jump with value $\Delta Y^n_{i-1} - 1$ and remains constant otherwise; then the quoted theorem stipulates that the above statement holds, with only the last condition being replaced by $\lim_{n\to\infty} a_n^{-2} \sum_{k \ge 0} (k-1)^2 \degf_n(k) = \sum_{i \ge 0} \theta_i^2$. On the one hand, the difference between the two series equals $\sum_{k \ge 0} k (k-1) \degf_n(k) - \sum_{k \ge 0} (k-1)^2 \degf_n(k) = \sum_{k \ge 0} (k-1) \degf_n(k) = -\varrho_n$, which we assume is of order $a_n$ and therefore small compared to $a_n^2$. On the other hand, the two paths lie at Skorokhod distance $1$ from each other, hence one converges after scaling if and only if the other one does, and with the same limit.

Considering this convergence criterion, let us define
\begin{equation}\label{eq:def_sigma}
\sigma_n^2 = \sum_{k \ge 0} k (k-1) \degf_n(k) = \sum_{i=1}^{E_n} \degf_{n,i} (\degf_{n,i}-1).
\end{equation}
We shall assume that $\sigma_n^2\to\infty$, as if it is bounded our objects are too degenerate.
We conclude from the preceding results that the following holds. The case $\varrho = 0$ can also be found in~\cite[Proposition~1]{AHUB20b}.
Let $\varrho \ge 0$, $\theta_0 \ge 0$, and $\theta_1 \ge \theta_2 \ge \dots \ge 0$ be such that $\sum_{i \ge 0} \theta_i^2 = 1$ and define $X = V_U Y$ as in~\eqref{eq:pont_premier_passage_continu}, where $Y$ is as in~\eqref{eq:representation_echangeable_general} and $U$ is independently sampled uniformly at random on $[0,\varrho]$.

\begin{prop}\label{prop:convergence_excursions}
Suppose that, as $n\to\infty$, we have $\sigma_n^2 \to \infty$ and
\[\sigma_n^{-1} \varrho_n \to \varrho
\qquad\text{and}\qquad
\sigma_n^{-1} \degf_{n,i} \to \theta_i \quad\text{for every}\enskip i \ge 1.\]
Then $(\sigma_n^{-1} X^n_{E_n t})_{t \in [-1,1]}$ converges in distribution to $X$ in the Skorokhod topology.
\end{prop}

\begin{proof}
Under the assumptions, the uniformly permuted path $\sigma_n^{-1} Y^n$ converges in distribution to $Y$, and (jointly) the independent random height $\sigma_n^{-1} U_n$ converges to $U$. The convergence in distribution of the first-passage bridges $\sigma_n^{-1} X^n$ to $X$ then follows from the almost sure continuity of the Vervaat transform at $(Y,U)$ provided by combining Remark~\ref{rem:min_locaux} and Lemma~\ref{lem:Vervaat_continue}.
\end{proof}

\begin{rem}\label{rem:convergence_excursions_brownien}
When the largest cycle length is $\degf_{n,1} = o(\sigma_n)$, then $\theta_i = 0$ for every $i \ge 1$, so $\theta_0=1$ and then $X$ reduces to the Brownian first passage bridge; in this case, the result was proved in~\cite[Proposition~1]{Mar19} and previously in~\cite[Theorem~1.6]{Lei19} under restrictive assumptions.
\end{rem}

\begin{rem}\label{rem:one_big_jump}
The case $\sigma_n^{-1} \varrho_n \to \infty$ has also been considered in~\cite{Mar19}, where now $(\varrho_n^{-1} X^n_{E_n t})_{t \in [-1,1]}$ converges in probability to $t\mapsto (1-t) \ind{t\ge0}$. Since the latter function satisfies~\eqref{eq:PJ}, then the results from Section~\ref{sec:convergence_saut_pur} apply.
\end{rem}

\section{Invariance principles for random (loop)trees and maps}
\label{sec:convergence}

Throughout this section we consider the uniform random models with degrees prescribed by $\varrho_n$ and $\Degf_n$ as defined in Section~\ref{ssec:model_degres_prescrits}. Let us refer to this section for the notation.
Let us start with weak results on trees, one of our points being that one does not need to control them in order to control labelled looptrees and eventually maps.
We then state and prove more precise versions of the theorems in the introduction.
Recall from Section~\ref{ssec:tension} our tightness results. Here we apply Lemma~\ref{lem:consequence_epine} to different functions $g$ in order to characterise the subsequential limits from the convergence of the rescaled {\L}ukasiewicz paths provided by Proposition~\ref{prop:convergence_excursions}. We henceforth denote by $\varrho$, $\bdtheta = (\theta_i)_{i \ge 0}$, and $X$ the quantities and process which appear in this proposition.
Recall that the case where $X$ satisfies the pure jump property~\eqref{eq:PJ} has been treated in Section~\ref{sec:convergence_saut_pur} and recall from Lemma~\ref{lem:PJ_echangeable} that when $\sum_i \theta_i < \infty$ this is the case if and only if $\theta_0=0$.

\subsection{Weak scaling limits of random trees}
\label{ssec:convergence_arbres_reduits}

Let $T^n$ denote an ordered plane forest with $\varrho_n$ trees sampled uniformly at random amongst those with exactly $\degf_n(k)$ vertices with outdegree $k$ for every $k\ge0$. 
Note that it has $\degf_n(0)$ leaves, $F_n+\varrho_n$ inner vertices, and $n$ edges.
In the case $\varrho_n = 1$, under technical assumptions on the degree sequence $(\degf_n)_{n\ge 1}$, in particular that $\sigma_n^2 \sim \sigma^2 n$ for some $\sigma \in (0,\infty)$ and $\lim_{n \to \infty} \sigma_n^{-1} \degf_{n,1} = 0$, Broutin \& Marckert~\cite{BM14} proved the convergence in distribution
\begin{equation}\label{eq:convergence_arbres_BM}
\frac{\sigma_n}{2 n} T^n \cvloi T_{\mathbf{e}},
\end{equation}
in the strong sense of Gromov--Hausdorff--Prokhorov, where $\mathbf{e}$ is the standard Brownian excursion, and the associated continuum random tree $T_{\mathbf{e}}$ as in~\eqref{eq:def_distance_arbre} is the celebrated Brownian CRT.
More precisely, they show that the contour or height process of the rescaled tree converges in distribution towards $\mathbf{e}$.
When instead $\varrho_n \sim \varrho \sqrt{n}$ with $\varrho \in (0,\infty)$, Lei~\cite{Lei19} extended this result to the sequence of trees, ranked in decreasing order of their number of vertices.

As in Section~\ref{ssec:model_degres_prescrits}, we shall canonically view the forest as a single tree by adding an extra root vertex, connected to the root of each tree and preserving the order. 
Under the sole assumption $\lim_{n \to \infty} \sigma_n^{-1} \degf_{n,1} = 0$ of no macroscopic degree, it was proved in~\cite[Theorems~4 and~5]{Mar19} that these trees converge in a weaker sense of subtrees spanned by finitely many random vertices, as depicted in Figure~\ref{fig:arbre_reduit}. Let us here extend this result in the case of possibly large degrees.
Fix $q \ge 1$ and let $u_1, \dots, u_q$ be $q$ i.i.d. uniform random vertices of $T^n$ and keep only these vertices and their ancestors and remove all the other ones; further let us merge each chain of vertices with only one child in this new tree into a single edge with a length given by the number of edges of the former chain. The resulting tree $\mathscr{R}^n(q)$ is called a discrete tree with edge-lengths; its combinatorial structure is that of a plane tree with at most $q$ leaves and no vertex with outdegree $1$, so there are only finitely many possibilities, and thus there are a bounded number of edge-lengths to record. We equip the space of trees with edge-lengths with the natural product topology.

\begin{figure}[!ht] \centering
\includegraphics[height=10\baselineskip]{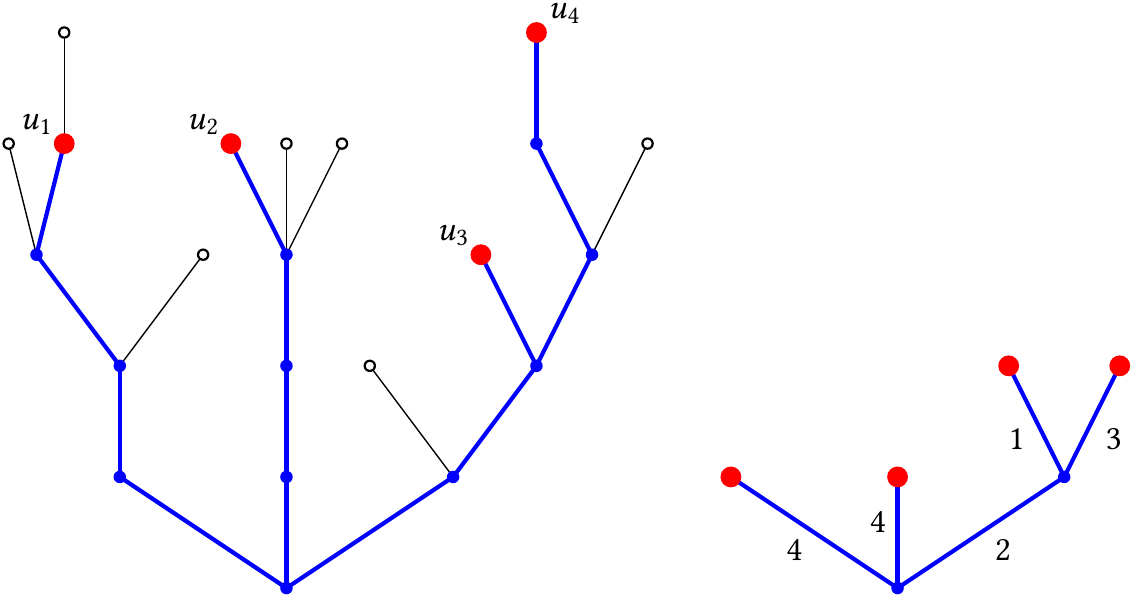}
\caption{Left: a plane tree and four distinguished vertices in red. Right: the associated reduced tree with edge-lengths.}
\label{fig:arbre_reduit}
\end{figure}

Analogously, fix $\varrho \ge 0$, $\theta_0 \ge 0$, and $\theta_1 \ge \theta_2 \ge \dots \ge 0$ such that $\sum_{i \ge 0} \theta_i^2 = 1$, let $Y$ be the bridge in~\eqref{eq:representation_echangeable_general}, sample $U$ independently and uniformly at random on $[0,\varrho]$, and define the first-passage bridge version $X = V_U Y$ as in~\eqref{eq:pont_premier_passage_continu}. Finally, let us denote by $C$ the continuous part of the running supremum of the dual path as in~\eqref{eq:def_J_et_C}. The latter encodes a continuum tree $T_C$; sample next i.i.d. random points $x_1, \dots, x_q$ in $T_C$ from its natural probability measure and construct similarly a discrete tree with edge-lengths $\mathscr{R}_C(q)$. Let us note that if $\theta_0=1$, then $X$ is a Brownian first passage bridge from $\varrho$ to $0$ and the process $C$ takes the form $C_t = X_t-\min_{s \in [0,t]} X_s$ for every $t \in [0,1]$. If moreover $\varrho=0$, then $X = C = \mathbf{e}$ and the law of $\mathscr{R}_{\mathbf{e}}(q)$ has been described by Aldous~\cite[Section~4.3]{Ald93}.

\begin{thm}\label{thm:convergence_arbres_reduits}
Suppose that $\theta_0 > 0$ and that, as $n\to\infty$, we have $\sigma_n^2 \to \infty$ with $n^{-1} \sigma_n \to 0$ and
\[\sigma_n^{-1} \varrho_n \to \varrho
\qquad\text{and}\qquad
\sigma_n^{-1} \degf_{n,i} \to \theta_i \quad\text{for every}\enskip i \ge 1.\]
Then for every $q \ge 1$ we have the convergence in distribution
\[\frac{\theta_0^2 \sigma_n}{2 n} \mathscr{R}^n(q) \cvloi \mathscr{R}_C(q),\]
jointly with the convergence in distribution of the rescaled {\L}ukasiewicz path $(\sigma_n^{-1} X^n_{E_n t})_{t \in [-1,1]}$ to $X$ from Proposition~\ref{prop:convergence_excursions}.
\end{thm}

Alternatively, in the spirit of~\cite{BM17}, one could build a tree from the forest by instead linking each root to the next one in a chain, and letting the former right-most root of the forest be that of the forest. This only requires mild adaptation and the theorem result holds if we replace the process $C$ by $\tildep{C}_t = C_t+\min_{s\le t} X_s$. If $X$ is a Brownian first passage bridge from $\varrho$ to $0$, then the process $\tildep{C}$ coincides with $X$.

\begin{rem}\label{rem:BHT}
Independently of this work, this result was proved in~\cite[Theorem~1]{BHT21} in the case of a single tree and with the restriction that $\sum_{i \ge 1} \theta_i < \infty$, using a method very close to ours. Note from Lemma~1 there that this extra assumption $\sum_{i \ge 1} \theta_i < \infty$ implies that $n^{-1}\sigma_n \to 0$, which seems to be sufficient in their work.
Recall from Lemma~\ref{lem:PJ_echangeable} that in the regime $\sum_i \theta_i < \infty$ and $\theta_0 > 0$, the process $C$ is nontrivial, therefore if we replace the assumption $n^{-1}\sigma_n \to 0$ by $\sum_i \theta_i < \infty$, then we can conclude than the discrete trees in the theorem have a nontrivial limit. We preferred however writing the theorem in this way because we believe that the limit is nontrivial also in the regime $\sum_i \theta_i = \infty$ and $\theta_0 > 0$ (recall Remark~\ref{rem:PJ_echangeable_mieux}).
\end{rem}

In the regime $\sum_i \theta_i < \infty$ and $\theta_0 > 0$ and in the case $\varrho=0$, the process $C$ is (up to a scaling factor) the one denoted by $Y^{\boldsymbol{\theta}}$ in~\cite{AMP04} where the authors prove that it is the height process of the so-called Inhomogeneous Continuum Random Trees obtained by a stick-breaking procedure.
If Remark~\ref{rem:PJ_echangeable_mieux} is true and $C$ is non trivial also when $\theta_0 > 0$ and $\sum_i \theta_i = \infty$, then Theorem~\ref{thm:convergence_arbres_reduits} would provide another clue indicating that $C$ is the height process of these trees in this regime as well.
On the other hand, 
the question of constructing the height process (if it exists) from $X$ in the case $\theta_0=0$ is much more involved; recall from Lemma~\ref{lem:PJ_echangeable} that in this case $C=0$ at least when in addition $\sum_i \theta_i < \infty$.

\begin{proof}[Proof of Theorem~\ref{thm:convergence_arbres_reduits}]
Let us rewrite our claim in terms of the coding paths.
Fix $q \ge 1$ and let $U_1, \dots, U_q$ be i.i.d. uniform random variables in $[0,1]$ independent of the rest and denote by $0 = U_{(0)} < U_{(1)} < \dots < U_{(q)}$ their ordered statistics. Then the theorem is equivalent to the fact that the convergence in distribution
\begin{equation}\label{eq:convergence_hauteur_arbre_reduit}
\frac{\theta_0^2 \sigma_n}{2 n} \left(H^n_{E_n U_{(i)}}, \inf_{U_{(i-1)} \le t \le U_{(i)}} H^n_{E_n t}\right)_{1 \le i \le q}
\cvloi
\left(C_{U_{(i)}}, \inf_{U_{(i-1)} \le t \le U_{(i)}} C_t\right)_{1 \le i \le q}
\end{equation}
holds jointly with that of $(\sigma_n^{-1} X^n_{E_n t})_{t \in [-1,1]}$ to $X$ from Proposition~\ref{prop:convergence_excursions}.

Fix a single random time $U$ and let us set $U_n = E_n U$. 
Recall from~\eqref{eq:consequence_epine_hauteur} that
\[\lim_{\delta \downarrow 0} \limsup_{n \to \infty}
\P\bigg(\bigg| C^{n,\delta}_{U_n} - \frac{\sigma^{2,\delta}_n}{2 n} H^n_{U_n} \bigg| > \varepsilon \sqrt{\sigma^{2,\delta}_n}\bigg) 
= 0.\]
Let us make the following remark: since $\sigma_n^{-1} \degf_{n,i} \to \theta_i$ for every $i \ge 1$, then for every $\delta>0$, 
\begin{equation}\label{eq:sigma_tronque}
\frac{\sigma_n^{2,\delta}}{\sigma_n^2}
= 1 - \frac{1}{\sigma_n^2}\sum_{k > \delta \sigma_n} \degf_n(k) k(k-1) 
\cv 1 - \sum_{i \ge 1} \theta_i^2 \ind{\theta_i > \delta}
\enskip\mathop{\longrightarrow}^{}_{\delta\downarrow0}\enskip \theta_0^2.
\end{equation}
Therefore $\frac{1}{\sigma_n} C^{n,\delta}_{U_n}$ is arbitrarily close to $\frac{\theta_0^2 \sigma_n}{2n} H^n_{U_n}$ with high probability. By a union bound, this holds jointly for each $E_n U_{(i)}$ for $1 \le i \le q$.
On the other hand Lemma~\ref{lem:convergence_partie_continue}, combined with Proposition~\ref{prop:convergence_excursions} shows that, letting first $n\to\infty$ and then $\delta\downarrow0$, jointly with the convergence of $\sigma_n^{-1} X^n_{E_n \cdot}$ to $X$ holds that of $\sigma_n^{-1} C^{n,\delta}_{E_n \cdot}$ to $C$. We deduce therefore the convergence of the first marginals in~\eqref{eq:convergence_hauteur_arbre_reduit}.

More generally, assume that the tree $T^n$ reduced to the ancestors of $q$ i.d.d. vertices has $q$ leaves (this occurs with high probability since the height of such a vertex is of order $n / \sigma_n = o(n)$ as we just proved) and a random number, say, $b \le q-1$ of branchpoints; then remove these $b$ branchpoints from the reduced tree to obtain a collection of $N=q+b$ single branches. Then Lemma~\ref{lem:consequence_epine} shows that, uniformly for $1 \le i \le q+b$, the number of individuals with less than $\delta \sigma_n$ offspring in the original forest on the $i$'th branch, multiplied by $\theta_0^2 \sigma_n/(2 n)$, is close to $1/\sigma_n$ times the total number of these offspring that lie strictly to the right of this path. Since this number is encoded by the {\L}ukasiewicz path, the theorem then follows as above from Proposition~\ref{prop:convergence_excursions} and Lemma~\ref{lem:convergence_partie_continue}. We leave the details to the reader.
\end{proof}

Below, we shall prove strong invariance principles for labelled looptrees. What prevents us to conclude to a Gromov--Hausdorff--Prokhorov convergence of the trees is a tightness argument which, as mentioned in Section~\ref{ssec:tension}, does not hold in general, as opposed to looptrees. Tightness can be obtained under some technical assumption thanks to the very recent work by Blanc-Renaudie~\cite{BR21}, see also the short discussion after Theorem~1 in~\cite{BHT21}.
By Corollary~\ref{cor:tension_Holder_arbres}, a very simple where tightness holds is when $n/\sigma_n$ is of order $\sigma_n$, which allows us to improve on~\cite{BM14, Lei19} which assume $\theta_1 = 0$ so $\theta_0=1$ and further technical assumptions. As for Remark~\ref{rem:BHT}, this was also observed in~\cite[Proposition~1]{BHT21}.

\begin{cor}\label{cor:convergence_arbres_var_finie}
Suppose that $\limsup_{n\to\infty} n^{-1} \degf_n(1) < 1$, that $\theta_0 > 0$, and that there exists $\sigma>0$ such that, as $n\to\infty$, we have 
\[n^{-1} \sigma_n^2 \to \sigma^2,
\qquad
n^{-1/2} \varrho_n \to \sigma\varrho,
\qquad\text{and}\qquad
n^{-1/2} \degf_{n,i} \to \sigma\theta_i \quad\text{for every}\enskip i \ge 1.\]
Then the following convergences in distribution hold jointly:
\[\left(\frac{1}{\sigma n^{1/2}} X^n_{E_n t}\right)_{t \in [0,1]} \cvloi X,
\qquad
\left(\frac{\theta_0^2 \sigma}{2 \sqrt{n}} H^n_{E_nt}\right)_{t \in [0,1]} \cvloi C,
\qquad\text{and}\qquad
\frac{\theta_0^2 \sigma}{2 \sqrt{n}} T^n \cvloi T_C,\]
respectively in the Skorokhod, uniform, and Gromov--Hausdorff--Prokhorov topology.
\end{cor}

\subsection{Scaling limits of random looptrees}
\label{ssec:convergence_looptrees_aleatoires}

Let us keep the same notation as above and let now $LT^n$ denote a uniform random looptree with cycle lengths prescribed by $\varrho_n$ and $(\degf_{n,i})_{1 \le i \le E_n}$. Using a similar argument as in the proof of Theorem~\ref{thm:convergence_arbres_reduits} we can control distances between random points in $LT^n$. As opposed to trees, we possess a tightness argument for the looptrees so we can conclude to a strong convergence.
Recall from~\eqref{eq:def_distance_looptree} the construction of the looptree $\Loop^a(X)$ with parameter $a \ge 0$ from $X$.

\begin{thm}\label{thm:convergence_looptree_degres_prescrits_general}
Suppose that, as $n\to\infty$, we have $\sigma_n^2 \to \infty$ 
and, 
\[\sigma_n^{-1} \varrho_n \to \varrho,
\qquad\text{and}\qquad
\sigma_n^{-1} \degf_{n,i} \to \theta_i \quad\text{for every}\enskip i \ge 1.\]
Assume also that there exists $a \ge 0$ such that 
\[\sigma_n^{-2} \degf_n(2\Z_+) \cv a \theta_0^2,\]
where $a=0$ if $\theta_0=0$.
Then the convergence in distribution
\[\sigma_n^{-1} LT^n \cvloi \Loop^{(a+1)/2}(X)\]
holds in the Gromov--Hausdorff--Prokhorov topology.
\end{thm}

We stress that $a$ can also be null when $\theta_0>0$, as in Corollary~\ref{cor:convergence_looptree_CRT} below.

\begin{proof}
First, exactly as in the proof of Corollary~\ref{cor:convergence_looptrees_discrets_saut_pur}, using the convergence~\eqref{eq:proportion_feuilles_random}, it suffices to prove the convergence in distribution
\begin{equation}\label{eq:convergence_distance_looptrees_echangeable}
\left(\sigma_n^{-1} d_{LT^n}(\floor{E_n s}, \floor{E_n t})\right)_{s,t \in [0,1]}
\cvloi (d^{(a+1)/2}_{\Loop(X)}(s,t))_{s,t \in [0,1]}
\end{equation}
in the uniform topology.
Proposition~\ref{prop:tension_Holder_looptrees} shows tightness of the distances, so it suffices to fix $q\ge 1$ and $q$ i.i.d. random times $U_1, \dots, U_q$ independent of the rest and to prove such a convergence for $s$ and $t$ restricted to these values.

As in the preceding proof, let us sample $U$ uniformly at random on $[0,1]$ and independent of the rest and let us set $U_n = \floor{E_n U}$. 
Let us apply Lemma~\ref{lem:consequence_epine} with $q=N=1$ and the function given by $g(k,j) = \min\{j, k-j\}$ for $k\ge j \ge 1$. In this case, we have for every $\delta>0$,
\[G_n^{2,\delta} 
= \sum_{k \le \delta \sigma_n} \degf_n(k) \sum_{j=1}^k \min\{j, k-j\}
= \sum_{k \le \delta \sigma_n} \degf_n(k) \frac{k^2 - \ind{k \in 2\Z+1}}{4}
= \frac{\sigma_n^{2,\delta}}{4} \left(1+\frac{N_n^\delta - \degf_n^\delta(2\Z+1)}{\sigma_n^{2,\delta}}\right),\]
where we have set $N_n^\delta = \sum_{k \le \delta \sigma_n} k \degf_n(k)$ and $\degf_n^\delta(2\Z+1) = \sum_{k \le \delta \sigma_n} \degf_n(k) \ind{k \in 2\Z+1}$.
Also the sum in Lemma~\ref{lem:consequence_epine} denotes the contribution of the cycles shorter than $\delta\sigma_n$ to the distance $d_{LT^n}(0, U_n)$, which we shall denote by $d^{n,\delta}_{LT^n}(0, U_n)$.
Then Equation~\eqref{eq:consequence_epine_Luka} 
and the fact that $\sigma^{2,\delta} \le \sigma_n^2$ imply that for any $\varepsilon>0$,
\[\lim_{\delta \downarrow 0} \limsup_{n \to \infty}
\P\bigg(\bigg| \frac{2}{1+(N_n^\delta - \degf_n^\delta(2\Z+1))/\sigma_n^{2,\delta}} \cdot d^{n,\delta}_{LT^n}(0, U_n) - C^{n,\delta}_{U_n} \bigg| > \varepsilon \sigma_n\bigg) 
= 0.\]
Next Lemma~\ref{lem:convergence_partie_continue} and Proposition~\ref{prop:convergence_excursions} show that, letting first $n\to\infty$ and then $\delta\downarrow0$, jointly with the convergence of $\sigma_n^{-1} X^n_{E_n \cdot}$ to $X$ holds that of $\sigma_n^{-1} C^{n,\delta}_{E_n \cdot}$ to $C$, hence the convergence in distribution
when first $n\to\infty$ and then $\delta \downarrow 0$,
\begin{equation}\label{eq:convergence_partie_continue_distance_looptrees}
\frac{2}{\sigma_n (1 + (N_n^\delta - \degf_n^\delta(2\Z+1))/\sigma_n^{2,\delta})} \cdot d^{n,\delta}_{LT^n}(0, U_n) \to  C_U.
\end{equation}
Now recall from~\eqref{eq:sigma_tronque} that, when $\delta>0$ is fixed and $n\to\infty$, the ratio $ \sigma_n^{2,\delta} / \sigma_n^2$ converges to $1 - \sum_{i \ge 1} \theta_i^2 \ind{\theta_i > \delta}$. 
Similarly, it holds for $\delta>0$ fixed, when $n\to\infty$,
\[(n-\degf_n(2\Z+1)) - (N_n^\delta - \degf_n^\delta(2\Z+1))
= \sum_{k > \delta \sigma_n} \degf_n(k) (k-\ind{k \in 2\Z+1}) 
\sim \sigma_n \sum_{i \ge 1} \theta_i \ind{\theta_i > \delta}
.\]
Recall that we assume that $\sigma_n^{-1} (n-\degf_n(2\Z+1)) \to a \theta_0^2$; we infer that for any $\delta>0$ fixed, it holds
\[\frac{N_n^\delta - \degf_n^\delta(2\Z+1)}{\sigma_n^{2,\delta}}
\cv \frac{a\theta_0^2}{1 - \sum_{i \ge 1} \theta_i^2 \ind{\theta_i > \delta}}.\]
For definiteness, the denominator in the right hand side must be nonzero; on the other hand, if it equals $0$, then $\theta_0=0$ and only finitely many $\theta_i$'s are nonzero, in which case $X$ clearly satisfies~\eqref{eq:PJ} a.s. so we can apply Proposition~\ref{prop:convergence_looptrees_saut_pur} to conclude about the theorem. 
Now either $\theta_0=0$ and the limit above is null for $\delta>0$ fixed, or $\theta_0 > 0$ and then it further converges to $a$ as $\delta\downarrow0$.

This shows the convergence of the contribution of the short cycles to $\sigma_n^{-1}$ times the looptree distance to $(a+1)/2 \cdot d_C(0,U)$.
On the other hand, the same argument used in the proof of Proposition~\ref{prop:convergence_looptrees_saut_pur}, under~\eqref{eq:PJ}, i.e.~when $C = 0$, shows that for any $\delta > 0$, the contribution of the cycles longer than $\delta \sigma_n$ to $\sigma_n^{-1}$ times the looptree distance between $0$ and $U_n$ converges to the analogous quantity in $\Loop(X)$. Letting then $\delta \downarrow 0$, the latter converges to $d^0_{\Loop(X)}(0,U)$.
We have therefore proved that
\[\sigma_n^{-1} d_{LT^n}(0, E_n U) \cvloi d_{\Loop(X)}^{(a+1)/2}(0,U).\]

The convergence of the pairwise distances between the $E_nU_i$'s is shown similarly by considering separately the branches on the reduced tree to which we remove the branchpoints, as in the proof of Theorem~\ref{thm:convergence_arbres_reduits}.
We leave the details to the reader.
\end{proof}

The proof can be adapted to control the looptrees $\overp{LT}^n$ as defined in~\cite{CK14}, without merging the right-most offspring of each individual with its parent.

\begin{thm}\label{thm:convergence_looptrees_CK}
Suppose that $\limsup_n n^{-1} \degf_n(1) < 1$ and that, as $n\to\infty$, we have $\sigma_n^2 \to \infty$ and, 
\[\sigma_n^{-1} \varrho_n \to \varrho,
\qquad\text{and}\qquad
\sigma_n^{-1} \degf_{n,i} \to \theta_i \quad\text{for every}\enskip i \ge 1.\]
Assume also that there exists $a \ge 0$ such that 
\[\frac{4n - \degf_n(2\Z_+)}{\sigma_n^2} \cv a \theta_0^2,\]
where $a=0$ if $\theta_0=0$.
Then the convergence in distribution
\[\sigma_n^{-1} \overp{LT}^n \cvloi \Loop^{(a+1)/2}(X)\]
holds in the Gromov--Hausdorff--Prokhorov topology.
\end{thm}

\begin{proof}
Recall from Corollary~\ref{cor:tension_looptrees_CK} that the rescaled graph distance is tight under our first assumption, so as previously it only remains to consider distances between random vertices.
Here one should consider the function $g(k,j)=\min\{j, k-j+1\}$. In this case, we have
\[\sum_{j=1}^k \min\{j, k-j+1\} 
= \frac{k^2 + 2k + \ind{k \in 2\Z+1}}{4}
.\]
Then the argument of the previous proof shows that, instead of~\eqref{eq:convergence_partie_continue_distance_looptrees}, we have here
\[\frac{2}{\sigma_n (1 + (3N_n^\delta + \degf_n^\delta(2\Z+1))/\sigma_n^{2,\delta})} \cdot d^{n,\delta}_{LT^n}(0, U_n) \to  C_U\]
when first $n\to\infty$ and then $\delta \downarrow 0$. Further, as in the previous proof, for any $\delta>0$ fixed,
\[\frac{3N_n^\delta + \degf_n^\delta(2\Z+1)}{\sigma_n^{2,\delta}}
\cv \frac{a\theta_0^2}{1 - \sum_{i \ge 1} \theta_i^2 \ind{\theta_i > \delta}},\]
and the rest is adapted verbatim.
\end{proof}

In the particular case when there is no macroscopic jump, i.e.~when $\sigma_n^{-1} \degf_{n,1} \to \theta_1 = 0$, first for every $\delta>0$ all values when removing the cycles longer than $\delta \sigma_n$ are equal to their non truncated counter part for $n$ large enough. Furthermore the limit path $X=B^\varrho$ is a Brownian first-passage bridge from $\varrho$ to $0$. When $\varrho = 0$, it reduces to the standard Brownian excursion, and the associated continuum random tree $\Loop^a(B^0) = a T_{B^0}$ as in~\eqref{eq:def_distance_arbre}, is up to a constant the celebrated Brownian CRT. When $\varrho > 0$, each excursion above its past minimum in $B^\varrho$ codes a tree, and the tree $T_{B^\varrho}$ is obtained by glueing these small trees along a line segment with length $\varrho$, see e.g.~\cite{BM17} for details. Instead, $\Loop^a(B^\varrho)$ is obtained by identifying together the two extremities of this segment to create a cycle, and multiplying all the trees glued on it by $a$.

\begin{cor}\label{cor:convergence_looptree_CRT}
Let $\varrho \ge 0$, $\sigma^2 \in (0, \infty]$, and $a \in [0,1)$ and suppose that
\[n^{-1} \sigma_n^2 \to \sigma^2,
\qquad
n^{-1} \degf_n(2\Z_+) \to a,
\qquad
\sigma_n^{-1} \varrho_n \to \varrho,
\qquad\text{and}\qquad
\sigma_n^{-1} \degf_{n,1} \to 0.\]
Then the convergence in distribution
\[\sigma_n^{-1} LT^n \cvloi \Loop^{(c+1)/2}(B^\varrho)
\qquad\text{where}\qquad
c = \frac{a}{\sigma^2} \ind{\sigma^2 < \infty},\]
holds in the Gromov--Hausdorff--Prokhorov topology.
\end{cor}

\begin{rem}\label{rem:convergence_looptrees_KR}
With the same adaptation as in Theorem~\ref{thm:convergence_looptrees_CK}, this result may also be written for the looptrees $\overp{LT}^n$, in which case it is close to~\cite[Theorem~1.2]{KR20} which considers so-called Bienaym\'e--Galton--Watson (loop)trees conditioned to have $n$ vertices. Actually it allows to recover this theorem, and extend it by conditioning also on the length of the root cycle as well as to other size conditionings (having $n$ cycles with length in any given fixed set $A \in \N$) by checking that the random empirical cycle lengths of such a conditioned looptree satisfies the assumptions on $\degf_n$, see Section~\ref{sec:Levy}.
We note that when $\sigma^2<\infty$, the constant $c$ here does not match $c_\mu$ in~\cite[Theorem~1.2]{KR20}; the latter is incorrect and should be divided by $\sigma_\mu^2/2$.
The error comes from Equation~5.9 there: with the notation there, $R(v_n^\ast)$ is a sum of $n$ i.i.d. random variables with mean $\sum_{k\ge 1} \mu(k) k (k-1)/2 = \sigma_\mu^2/2$ and not $1$ in general.
\end{rem}

Another extreme case is when the root cycle is very long. Recall from Remark~\ref{rem:one_big_jump} that when $\sigma_n^{-1} \varrho_n \to \infty$, it is shown in~\cite[Proposition~1]{Mar19} that $(\varrho_n^{-1} X^n_{E_n t})_{t \in [-1,1]}$ converges in probability to $t\mapsto (1-t) \ind{t\ge0}$. Note that the latter satisfies~\eqref{eq:PJ} so we may apply Proposition~\ref{prop:convergence_looptrees_saut_pur} to deduce the convergence of their looptrees.

\begin{cor}\label{cor:convergence_looptree_cercle}
Suppose that $\sigma_n^{-1} \varrho_n \to \infty$, then $\varrho_n^{-1}LT^n$ converges in probability towards the circle with unit perimeter in the Gromov--Hausdorff--Prokhorov topology.
\end{cor}

\subsection{Convergence of random labels on random looptrees}
\label{ssec:convergence_labels_aleatoires}

Recall that $X^n$ denotes the {\L}ukasiewicz path of a uniformly random looptree with cycle lengths prescribed by $\varrho_n$ and $(\degf_{n,i})_{1 \le i \le E_n}$. Let $(\xi_i)_{i\ge1}$ be i.i.d. copies of  an integer valued, centred random variable $\xi$, with nonzero and finite variance and recall the notation 
$\Xi^\ell_j$ for the law of $\xi_1 + \dots + \xi_j$ under the conditional law $\P(\,\cdot\mid \xi_1 + \dots + \xi_\ell = 0)$ for every $0 \le j \le \ell$. 
Finally define the label process $Z^n$ on such a random looptree by~\eqref{eq:def_processus_labels_discret_random}. Similarly, let $\varrho \ge 0$, $\theta_0 \ge 0$, and $\theta_1 \ge \theta_2 \ge \dots \ge 0$ be such that $\sum_{i \ge 0} \theta_i^2 = 1$ and let $Y$ be a random process as in~\eqref{eq:representation_echangeable_general} and then $X = V_U Y$ as in~\eqref{eq:pont_premier_passage_continu}. Conditionally on $X$, define $Z^a$ as in~\eqref{eq:def_labels_browniens}.
Finally define
\[\Sigma_n^2 = \sum_{k \ge 1} \degf_n(k) k (k+1) \Var(\Xi^k_1),\]
which is of order $\sigma_n^2$ since $\Var(\Xi^k_1)$ converges to $\Var(\xi)$ as $k\to\infty$.

\begin{thm}\label{thm:convergence_labels_degres_prescrits_general}
Suppose that $\E[|\xi|^{4+\varepsilon}] < \infty$ for some $\varepsilon > 0$ and that
\[\sigma_n^{-1} \varrho_n \to \varrho,
\qquad\text{and}\qquad
\sigma_n^{-1} \degf_{n,i} \to \theta_i \quad\text{for every}\enskip i \ge 1.\]
Finally assume that there exists $\Sigma \ge 0$ such that 
\[\sigma_n^{-2} \Sigma^2_n \cv \Var(\xi) (1 + (\Sigma^2-1) \theta_0^2),\]
where $\Sigma^2 = 1$ if $\theta_0=0$.
Then the convergence in distribution
\[\left(\frac{1}{\sigma_n} X^n_{E_n t}, \frac{1}{\sqrt{\Var(\xi) \sigma_n}} Z^n_{E_n t}\right)_{t \in [-1,1]} \cvloi (X, Z^{\Sigma^2/3}),\]
holds for the Skorokhod topology.
\end{thm}

We shall see that the main point of the proof Theorem~\ref{thm:convergence_labels_degres_prescrits_general} is to show a conditional central limit theorem, given $X^n$, for a sum of independent random variables of the form $\Xi^\ell_k$, where each $\ell$, each $k$, and the number of terms itself are measurable with respect to $X^n$ and an independent sample of i.i.d. uniform random times in $[0,1]$. A useful idea to prove such a convergence is to replace these random variables by Gaussian random variables with the same variance, for which the CLT holds as soon as the sum of the variances converges once suitably rescaled. Since these variances are measurable with respect to $X^n$ (and the independent uniform random times), this removes one layer of randomness. The next lemma bounds the cost of such a replacement.

\begin{lem}\label{lem:difference_ponts_gaussiennes}
Suppose that $\E[|\xi|^3] < \infty$. Then there exists a constant $K>0$
such that the following holds:
Let $h\ge1$ and $k_m \ge j_m \ge 1$ for every $1 \le m \le h$ and sample independently $\Xi^{1, k_1}_{j_1}, \dots, \Xi^{h, k_h}_{j_h}$, where each $\Xi^{m, k_m}_{j_m}$ has the law of $\xi_1 + \dots + \xi_{j_m}$ under the conditional law $\P(\,\cdot\mid \xi_1 + \dots + \xi_{k_m} = 0)$,
then for every $z \in \R$,
\[\bigg|\E\bigg[\exp\bigg(\i z \sum_{m=1}^h \Xi^{m, k_m}_{j_m}\bigg)\bigg] - \exp\bigg(- \frac{z^2}{2} \sum_{m=1}^h \Var(\Xi^{m, k_m}_{j_m})\bigg)\bigg|
\le K |z|^3 \sqrt{\max_{1 \le m \le h} k_m} \sum_{m=1}^h (k_m - j_m)
.\]
\end{lem}

\begin{proof}
We use the fact that for any random variable, $X$ say, which is centred and has finite third moment, it holds that
\[\left|\Es{\exp\left(\i z X\right)} - \left(1 - \frac{z^2}{2} \Es{X^2}\right)\right| \le \frac{|z|^3}{6} \Es{|X|^3}.\]
When $X$ has the standard Gaussian law, so $\E[|X|^3] = \sqrt{8/\pi} \le 2$, we obtain using also Jensen's inequality that for any $k \ge j\ge 1$,
\[\left|\exp\left(- \frac{z^2}{2} \Var(\Xi^k_j)\right) - \left(1 - \frac{z^2}{2} \Var(\Xi^k_j)\right)\right|
\le \frac{|z|^3}{6} \E\big[|\Xi^k_j|^2\big]^{3/2} \sqrt{\frac{8}{\pi}}
\le \frac{|z|^3}{3} \E\big[|\Xi^k_j|^3\big].\]
The same bound applied to $\Xi^k_j$ and the triangle inequality then yield
\[\left|\Es{\exp(\i z\, \Xi^k_j)} - \exp\left(- \frac{z^2}{2} \Var(\Xi^k_j)\right)\right|
\le \frac{|z|^3}{2} \E\big[|\Xi^k_j|^3\big].\]
On the other hand, one can show by induction that for any complex numbers in the closed unit disk $(\alpha_m)_{1 \le m \le h}$ and $(\beta_m)_{1 \le m \le h}$, it holds that $|\prod_{1 \le m \le h} \alpha_m - \prod_{1 \le m \le h} \beta_m| \le \sum_{1 \le m \le h} |\alpha_m - \beta_m|$. We infer that
\[\bigg|\E\bigg[\exp\bigg(\i z \sum_{m=1}^h \Xi^{m, k_m}_{j_m}\bigg)\bigg] - \exp\bigg(- \frac{z^2}{2} \sum_{m=1}^h \Var(\Xi^{k_m}_{j_m})\bigg)\bigg|
\le \frac{|z|^3}{2} \sum_{m=1}^h \E\big[|\Xi^{m, k_m}_{j_m}|^3\big]
.\]
Property~\ref{item:moments_ponts} of the bridges, below Equation~\eqref{eq:def_processus_labels_discret_random}, shows that there exists a constant $K_3$ such that for every $k \ge j \ge 1$, we have
\[\E[|\Xi^k_j|^3] \le K_3 |k-j|^{3/2} \le K_3 \sqrt{k} |k-j|.\]
Our claim then follows.
\end{proof}

Let us mention that a better bound in the first line of the proof allows to extend the claim assuming only that $\E[|\xi|^{2+\varepsilon}] < \infty$ for some $\varepsilon > 0$, with $|z|^{2+\varepsilon} \max_{1 \le m \le h} k_m^{\varepsilon/2}$ instead on the right hand side and this is sufficient for our purpose. We nevertheless assume that $\E[|\xi|^{4+\varepsilon}] < \infty$ for some $\varepsilon > 0$ in Theorem~\ref{thm:convergence_labels_degres_prescrits_general} in order to deduce tightness of the labels.

\begin{proof}[Proof of Theorem~\ref{thm:convergence_labels_degres_prescrits_general}]
Recall that the convergence of the {\L}ukasiewicz paths was proved in Proposition~\ref{prop:convergence_excursions} and that Corollary~\ref{cor:tension_Holder_labels} ensures that the sequence of rescaled label processes is tight. Therefore it suffices to sample independently of the rest a finite number, say $U_1, \dots, U_q$ of i.i.d. uniform random times in $[0,1]$, and prove the joint convergence
\begin{equation}\label{eq:convergence_marginales_labels}
\left(\Var(\xi) \sigma_n\right)^{-1/2}
\left(Z^n_{E_n U_1}, \dots, Z^n_{E_n U_q}\right) \cvloi (Z_{U_1}, \dots, Z_{U_q}),
\end{equation}
jointly with the convergence of the {\L}ukasiewicz paths. Let us assume that the latter holds almost surely by Skorokhod's representation theorem and let us work conditionally given $X^n$ and $X$. 
This was done in~\cite[Proposition~6]{Mar19} when $\theta_i = 0$ for every $i \ge 1$. 
As in the proof of Theorem~\ref{thm:convergence_looptree_degres_prescrits_general} let us introduce a cut-off and argue as in the proof of Proposition~\ref{prop:convergence_labels_saut_pur} for the long loops and differently for the short loops. Let us first focus on the case $q=1$ of a single random time $U$; let $U_n = \floor{E_n U}$ and let us also denote by $U_n$ the vertex visited at this time.

Let us henceforth fix $\delta > 0$. As in the proof of Proposition~\ref{prop:convergence_labels_saut_pur} the number of loops longer than $\delta \sigma_n$ converges, as well as their left and right length, towards the same quantities in the limit looptree coded by $X$. Moreover, by Property~\ref{item:convergence_pont_brownien} of the $\Xi$-bridges, below Equation~\eqref{eq:def_processus_labels_discret_random}, the contribution of these loops to $(\Var(\xi) \sigma_n)^{-1/2} Z^n_{U_n}$ converges in distribution towards that of $Z_U$. 
This limit finally converges in distribution towards the series in~\eqref{eq:def_labels_browniens} when $\delta \downarrow 0$. The difference with Proposition~\ref{prop:convergence_labels_saut_pur} is that the rest does not vanish when $\delta \downarrow 0$, but instead converges to the ``snake'' part $\sqrt{\Sigma^2/3} Z^C_U$. We therefore aim at showing that the latter is the limit of the contribution of the short loops.
In order to deal with these short loops, define
\[Z^{n,\delta}_{U_n} = \sum_{i=0}^{U_n-1} \sum_{\ell=1}^{\delta \sigma_n} \sum_{k=0}^\ell \Xi^{\ell,i}_k \ind{\Delta X^n_{i-1} = \ell} \ind{\inf_{t \in [i, U_n]} X^n_{\floor t} - X^n_{i-}=\ell-k},\]
which is the contribution of the loops shorter than $\delta \sigma_n$ to the label of the random point $U_n$. It remains to prove that when first $n\to\infty$ and then $\delta \downarrow 0$, the law of $(\Var(\xi) \sigma_n)^{-1/2} Z^{n,\delta}_{U_n}$ converges to a centred Gaussian law with variance $\Sigma^2 C_U/3$.
Recall the notation $C^{n,\delta}_{U_n}$ for the sum of the right part of the ancestral loops of $U_n$ with length smaller than or equal to $\delta \sigma_n$.
According to Lemma~\ref{lem:convergence_partie_continue}, almost surely in our probability space, $\sigma_n^{-1} C^{n,\delta}_{U_n}$ converges to $C^\delta_U$ as $n\to\infty$, which itself converges to $C_U$ as $\delta\downarrow0$. Therefore, in order to deduce~\eqref{eq:convergence_marginales_labels} with $q=1$, it only remains to show that, when first $n\to\infty$ and then $\delta \downarrow 0$, the law of 
$(3\sigma_n/(\Var(\xi) \Sigma^2 C^{n,\delta}_{U_n}))^{1/2} Z^{n,\delta}_{U_n}$ 
converges to the standard Gaussian law. Let us first prove that its variance converges to $1$.

Observe that the conditional variance of $Z^{n,\delta}_{U_n}$ given $X^n$ is of the form of the sum in Lemma~\ref{lem:consequence_epine}, with the function $g(k,j) = \Var(\Xi^k_j)$. As discussed in~\cite[page~1664]{MM07}, it holds by exchangeability for every $k \ge j \ge 1$,
\[\Var(\Xi^k_j) = j \Var(\Xi^k_1) + j (j-1) \Cov(\Xi^k_1, \Xi^k_2-\Xi^k_1).\]
Since this vanishes for $j=k$, then the covariance term equals $-(k-1)^{-1} \Var(\Xi^k_1)$, hence
\[g(k,j) = \Var(\Xi^k_j) 
= \left(j - \frac{j (j-1)}{k-1}\right) \Var(\Xi^k_1)
= \frac{j (k - j)}{k-1} \Var(\Xi^k_1).\]
Note that $g$ satisfies the assumptions of Lemma~\ref{lem:consequence_epine}, so we read from Equation~\eqref{eq:consequence_epine_Luka} that for any $\varepsilon>0$,
\[\lim_{\delta \downarrow 0} \limsup_{n \to \infty}
\P\bigg(\bigg| \frac{\sigma^{2,\delta}_n}{2 G^{2,\delta}_n} \Var(Z^{n, \delta}_{U_n} \mid X^n) - C^{n,\delta}_{U_n} \bigg| > \varepsilon \sqrt{\sigma^{2,\delta}_n}\bigg) 
= 0,\]
where
\[G_n^{2,\delta} = \sum_{k \le \delta \sigma_n} \degf_n(k) \sum_{j=1}^k g(k,j) 
= \sum_{k \le \delta \sigma_n} \degf_n(k) \frac{k (k+1)}{6} \Var(\Xi^k_1)
.\]
We already mentioned that $\Var(\Xi^k_1) \to \Var(\xi)$ as $k\to\infty$, therefore
since $\sigma_n^{-1} \degf_{n,i} \to \theta_i$ for every $i \ge 1$, then for every $\delta>0$,
\[6 \left(G_n^2 - G_n^{2,\delta}\right)
= \sum_{k > \delta \sigma_n} \degf_n(k) k (k+1) \Var(\Xi^k_1)
\enskip\mathop{\sim}_{n\to\infty}\enskip \sigma_n^2 \Var(\xi) \sum_{i \ge 1} \theta_i^2 \ind{\theta_i > \delta}
\enskip\mathop{\sim}_{\delta\downarrow0}\enskip \sigma_n^2 \Var(\xi) (1 - \theta_0^2)
.\]
Recall that we assume that $\sigma_n^{-2} \Sigma^2_n \to \Var(\xi) (1 - \theta_0^2 + \Sigma^2 \theta_0^2)$, with $\Sigma^2 = 1$ when $\theta_0=0$. We next consider limits when first $n \to \infty$ and then $\delta \downarrow 0$. First, recall from~\eqref{eq:sigma_tronque} that the ratio $ \sigma_n^{2,\delta} / \sigma_n^2$ converges to $\theta_0^2$ and this implies that $G^{2,\delta}_n / \sigma_n^{2,\delta}$ converges to $\Sigma^2 \Var(\xi)/6$.
Since $\sigma_n^{2,\delta}\le \sigma_n^2$ and $\sigma_n^{-1} C^{n, \delta}_{U_n}$ converges to $C_U > 0$ a.s. by Lemma~\ref{lem:convergence_partie_continue}, then we infer that for any $\varepsilon>0$, 
\[\lim_{\delta \downarrow 0} \limsup_{n \to \infty}
\P\bigg(\bigg|\frac{3}{\Sigma^2 \Var(\xi) \sigma_n C_U} \Var(Z^{n, \delta}_{U_n} \mid X^n) - 1 \bigg| > \varepsilon\bigg) 
= 0,\]
which is what we aimed at.

Let us next apply Lemma~\ref{lem:difference_ponts_gaussiennes} to the bridges that contribute to $Z^{n, \delta}_{U_n}$: the $k_m$' are all bounded by $\delta \sigma_n$, and the sum of the $(k_m-j_m)$'s equals $C^{n,\delta}_{U_n}$.
We infer that for every $z\in \R$ and $\delta>0$,
\begin{multline*}
\bigg|\E\bigg[\exp\bigg(\i z \sqrt{\frac{3}{\Sigma^2 \Var(\xi) \sigma_n C_U}} Z^{n, \delta}_{U_n}\bigg) \;\bigg|\; X^n\bigg] - \exp\bigg(- \frac{z^2}{2} \frac{3}{\Sigma^2 \Var(\xi) \sigma_n C_U} \Var(Z^{n, \delta}_{U_n} \mid X^n)\bigg)\bigg|
\\
\qquad\le K |z|^3 \bigg(\frac{3}{\Sigma^2 \Var(\xi) \sigma_n C_U}\bigg)^{3/2} \sqrt{\delta \sigma_n} C^{n,\delta}_{U_n}
\cv K |z|^3 \bigg(\frac{3}{\Sigma^2 \Var(\xi)}\bigg)^{3/2} \sqrt{\delta} C_U^{-3/2} C^{\delta}_{U}
,\end{multline*}
which further tends to $0$ when $\delta\downarrow0$. This finishes the proof of~\eqref{eq:convergence_marginales_labels} with $q=1$.

When $q \ge 2$, we proceed exactly as in the proof of Theorem~\ref{thm:convergence_arbres_reduits} by splitting at the branchpoints of the reduced tree. Then the previous argument extends to show the joint convergence of the label increments along each branch of this reduced tree as well as the increments at the branchpoints with degree greater than $\delta\sigma_n$. As for the label increments at the branchpoints with degree smaller than $\delta\sigma_n$, we fist note that there are at most $q-1$ of them, and each such increment has the conditional law of $\Xi^k_i$ for some pair $1 \le i \le k \le \delta \sigma_n$. We conclude from the property~\ref{item:moments_ponts} of the $\Xi$-bridges, below Equation~\eqref{eq:def_processus_labels_discret_random}, combined with a union bound and the Markov inequality that with high probability as first $n\to\infty$ and then $\delta\downarrow0$ the maximal increment contributing to the label of our $q$ random points over all branchpoints with degree smaller than $\delta\sigma_n$ is small compared to $\sigma_n^{1/2}$, which completes the proof.
\end{proof}

\begin{rem}\label{rem:convergence_labels_CK}
Similarly to Theorem~\ref{thm:convergence_looptrees_CK}, this results holds for random labels constructed in the same way but on the looptrees $\overp{LT}^n$ when $\limsup_n n^{-1} \degf_n(1) < 1$. 
One should just instead use the function $g(k,j) = \Var(\Xi^{k+1}_j)$ for $k\ge j \ge 1$, 
so the factor $\Sigma_n^2$ gets replaced by
$\sum_{k \ge 1} \degf_n(k) (k+1) (k+2) \Var(\Xi^{k+1}_1)$.
\end{rem}

\subsection{Application to random maps}
\label{ssec:cartes_degres_prescrits}

Recall that in the case $\P(\xi=i) = 2^{-(i+2)}$ for every $i \ge -1$, we obtain a uniformly random good labelling of the looptree; further, by Lemma~\ref{lem:bijection}, this codes a random pointed map $(M^n, v^n_\star)$ sampled uniformly at random with degrees prescribed by $\varrho_n$ and $\Degf_n$. 
Recall from the introduction that in this case Marckert \& Miermont~\cite[page~1664]{MM07} have calculated $\Var(\Xi^k_1) = \frac{2(k-1)}{k+1}$, so $\Sigma^2_n = 2\sigma_n^2$ and thus $\Sigma^2=1$ in Theorem~\ref{thm:convergence_labels_degres_prescrits_general}. Let us therefore set
\[\mathbf{Z} = Z^{1/3}.\]
Then combining our previous results, we deduce invariance principles which extend the work~\cite{Mar19} which focused on the Brownian case $\theta_1=0$.
Let us denote by $(v^n_1, \dots, v^n_{\degf_n(0)})$ the vertices of $M^n$ different from $v^n_\star$, which are the vertices of its associated looptree, listed in the order of visit of their last external corner when following the contour of the looptree.
Recall that we denote by $d_{M^n}(v^n_\star, \vec{e}_n)$ for the smallest between the distiguinshed vertex $v^n_\star$ and the endpoints of the root edge.

\begin{thm}\label{thm:convergence_cartes_degres_prescrits}
Suppose that 
$\sigma_n^2 \to \infty$, and that
\[\sigma_n^{-1} \varrho_n \to \varrho,
\qquad\text{and}\qquad
\sigma_n^{-1} \degf_{n,i} \to \theta_i \quad\text{for every}\enskip i \ge 1.\]
\begin{enumerate}
\item The convergence in distribution
\[\left(\frac{1}{\sqrt{2 \sigma_n}}\dgr(v^n_\star, v^n_{\floor{\degf_n(0) t}})\right)_{t \in [0,1]} \cvloi \left(\mathbf{Z}_t - \min \mathbf{Z}\right)_{t \in [0,1]},\]
holds for the uniform topology. Consequently,
\[\frac{1}{\sqrt{2 \sigma_n}} d_{M^n}(v^n_\star, \vec{e}_n) \cvloi - \min \mathbf{Z}
\qquad\text{and}\qquad
\frac{1}{\sqrt{2 \sigma_n}} \max_{v \in V(M^n)} d_{M^n}(v^n_\star, v) \cvloi \max \mathbf{Z} - \min \mathbf{Z},\]
and for every continuous and bounded function $F$, we have
\[\frac{1}{\degf_n(0)} \sum_{v \in V(M^n)} F\left(\frac{1}{\sqrt{2 \sigma_n}} d_{M^n}(v^n_\star, v)\right) \cvloi \int_0^1 F(\mathbf{Z}_t - \min \mathbf{Z}) \d t.\]

\item From every increasing sequence of integers, one can extract a subsequence along which the convergence in distribution
\[\left(\frac{1}{\sqrt{2 \sigma_n}} d_n(\degf_n(0) s, \degf_n(0) t)\right)_{s,t \in [0,1]} \cvloi \left(D_\infty(s,t)\right)_{s,t \in [0,1]},\]
holds for the uniform topology, where $D_\infty$ is a random continuous pseudo-distance which satisfies
\[D_\infty(U,\cdot) \eqloi \mathbf{Z} - \min \mathbf{Z},\]
where $U$ is sampled uniformly at random on $[0,1]$ and independently of the rest.
Finally along this subsequence, we have
\[\frac{1}{\sqrt{2 \sigma_n}} M^n \cvloi{} [0,1]/\{D_\infty=0\}\]
in the Gromov--Hausdorff--Prokhorov topology. 
\end{enumerate}
\end{thm}

\begin{proof}
Combining Theorem~\ref{thm:convergence_labels_degres_prescrits_general}, the convergence~\eqref{eq:proportion_feuilles_random}, and Skorokhod's representation theorem, we may assume the almost sure convergence 
\[\left(\frac{1}{\sqrt{2 \sigma_n}} Z^n_{n t}, \frac{1}{E_n} \lambda^n(\degf_n(0) t)\right)_{t \in [0,1]} \cvps \left(\mathbf{Z}_t, t\right)_{t \in [0,1]}.\]
The first claim then follows from Lemma~\ref{lem:profil_general} and the second one from Lemma~\ref{lem:tension_cartes_general}. 
For the identity in law $D_\infty(U,\cdot) = \mathbf{Z} - \min \mathbf{Z}$, the left hand side is the limit in law of the rescaled distances between the vertices $v^n_1, \dots, v^n_{\degf_n(0)}$ of $M^n\setminus\{v^n_\star\}$ and a uniformly random vertex $u_n$ in this set. As remarked in Section~\ref{ssec:model_degres_prescrits}, the distinguished vertex $v^n_\star$ has the uniform distribution on the set of all vertices of $M^n$ and is independent of the latter. Therefore we can couple $u_n$ and $v^n_\star$ in such a way that they differ with a probability at most $1/(\degf_n(0)+1) \to 0$. Then replacing $u_n$ by $v^n_\star$, we conclude from the first claim.
\end{proof}

\section{Boltzmann distributions and L\'evy processes}
\label{sec:Levy}

In this last section, let us briefly remark the applications of our results to \emph{mixtures} obtained by first sampling the degree sequence at random and then sampling a (loop)tree or a map uniformly at random given these degrees. This recovers the models of simply generated and size-conditioned Bienaym\'e--Galton--Watson (loop)trees and Boltzmann plane maps.
In the case of no macroscopic degrees, this is already discussed in~\cite[Section~6]{Mar19}. This will be pursued in~\cite{KM22}.
Let us first present the continuum objects before presenting the discrete ones and finally discuss invariance principles.

\subsection{L\'evy processes}
\label{ssec:Levy}

Let $X$ denote a L\'evy process, i.e. a random c\`adl\`ag path with stationary and independent increments. Such a process has in particular exchangeable increments, and so have bridge versions defined below. They can therefore fit in the setting of Section~\ref{sec:echangeable}, where the parameters $\theta_i$ in~\eqref{eq:representation_echangeable_general} are now random.
We shall work in the same setting as Duquesne \& Le~Gall~\cite{DLG02} and consider such a process with no negative jump and whose law is determined by its Laplace transform given for all $t, \lambda \ge 0$, by
$\E[\exp(-\lambda X_t)] = \exp(t \psi(\lambda))$, where
\[\psi(\lambda) = d \lambda + \beta \lambda^2 + \int_{(0,\infty)} (\e^{-\lambda r} - 1 + \lambda r) \pi(\d r),\]
where $d \ge 0$ is the drift coefficient, $\beta \ge 0$ is the Gaussian coefficient, and finally $\pi$ is the L\'evy measure, which satisfies $\int_{(0,\infty)} \min\{r, r^2\} \pi(\d r) < \infty$. This integrability condition, together with the nonnegativity of $d$, ensures that the path is either recurrent, when $d = 0$, or drifts to $-\infty$ otherwise (this is a kind of subcritical regime); this ensures that the path has no finite lower bound so one can easily consider excursions and first passage-bridges. Finally the path has finite variation if and only if $\beta = 0$ and $\int_{(0,1)} r \pi(\d r) < \infty$.
In the case $\psi(\lambda) = \lambda^2$, the process $X$ is $\sqrt{2}$ times a standard Brownian motion; more generally the case $\psi(\lambda) = \lambda^\alpha$ where $\alpha \in (1,2]$ is known as that of stable L\'evy processes with index $\alpha$.

From any such L\'evy process, 
Duquesne \& Le~Gall~\cite{DLG02} constructed the so-called \emph{height process} $H$ and studied thoroughly its properties. If (an excursion of) $X$ is the analogue of the {\L}ukasiewicz path of a plane tree, then (an excursion of) $H$ is the analogue of its height or contour process, and the continuum L\'evy tree is $T_H = [0,1]/\{d_H=0\}$, where $d_H$ is defined in~\eqref{eq:def_distance_arbre}.
Recall the notation at the very beginning of Section~\ref{sec:deterministe}, in particular that for every $t \in [0,1]$, we write the running supremum of the dual path $\overline{\Xsf}\vphantom{X}^t$ as the sum of a pure jump path $J^t$ and a continuous path $C^t$; then $X_{t-} = J^t_t + C^t_t = J_t + C_t$.
As shown in the next proposition, L\'evy processes satisfy the pure jump property~\eqref{eq:PJ} if and only if they have no Gaussian component, which in this case answers Remark~\ref{rem:PJ_echangeable_mieux}.

\begin{prop}\label{prop:Levy_partie_continue_hauteur}
For any such L\'evy process, it holds:
\begin{enumerate}
\item Either $\beta = 0$ and then $C = 0$ almost surely,
\item Or $\beta > 0$ and then $C_t = \beta H_t = \Leb(\{\min_{r \in [s,t]} X_s ; s \in [0,t]\})$, which vanishes only at times $t$ such that $X_t = \min_{s \le t} X_s$.
\end{enumerate}
\end{prop}

\begin{proof}
Fix $t \in [0,1]$ and first note that the dual process $X^t$ has the same law as $X$.
According to~\cite[Lemma~1.1.2]{DLG02}, its ladder height process is a subordinator with 
drift coefficient $\beta$, and the range of the ladder process and the running supremum agree except on a set which is at most countable. Therefore if $\beta=0$, then the ladder process is a pure jump process and so $\overline{\Xsf}\vphantom{X}^t = J^t$. On the other hand, if $\beta > 0$, then $\overline{\Xsf}\vphantom{X}^t$ has a continuous part, given by $\beta$ times the local time at the supremum of the dual process, which is exactly the height process~\cite[Definition~1.2.1]{DLG02}. 
The second claim is then~\cite[Equation~14]{DLG02}.
\end{proof}

We next want to consider conditioned L\'evy processes.
Under the two assumptions that 
(i) the point $0$ is regular for both $(0, \infty)$ and $(-\infty, 0)$, which holds since $X$ has no negative jump and infinite variation~\cite[Chapter~VII]{Ber96}, and 
(ii) that $\int |\E[\e^{i u X_t}]| \d u < \infty$ for every $t > 0$,
it is shown in~\cite{UB14} that one can construct bridge and excursion versions of $X$, i.e.~paths $X^{\rm br}$ and $X^{\rm ex}$ which are regular conditional distributions of $(X_t)_{t \in [0,1]}$, conditioned on $X_0 = X_1 = 0$ for the former and moreover on $X_t > 0$ for all $t \in (0,1)$ for the latter. Moreover, the bridge is locally absolutely continuous with respect to the unconditioned path~\cite[Equation~2.1]{UB14} and then the excursion is obtained from the bridge via the Vervaat transform~\cite[Theorem~4]{UB14}, therefore Proposition~\ref{prop:Levy_partie_continue_hauteur} remains valid for both processes. More generally one can construct first-passage bridges versions from $\varrho > 0$ to $0$, which satisfy $X_0 = \varrho$, $X_1 = 0$, and $X_t > 0$ for every $t \in (0,1)$.

We may therefore construct as in Section~\ref{sec:deterministe} the looptree associated with such a process $X$, or with a bridge or excursion version, and then add random Gaussian labels on it.
We study more these L\'evy labelled looptrees and associated maps in the forthcoming paper~\cite{KM22}; in particular we prove the regularity assumption~\eqref{eq:hypothese_Kolmogorov_continuite_distance_loop} on the looptree distance with the optimal exponent and we calculate the fractal (Hausdorff, Minkowski, packing) dimensions of the looptree and the corresponding map in terms of the so-called Blumenthal--Getoor exponents of $X$, which generalise the Brownian and stable cases~\cite{CK14, LGM11}.

\subsection{Boltzmann distributions and random degree sequences}
\label{sec:Boltzmann}

Marckert \& Miermont~\cite{MM07} defined laws on pointed bipartite plane maps which are designed so the maps have random face degrees, called \emph{Boltzmann distributions}. Let us also refer to~\cite{BM17} for more details, in the context of maps with a boundary as well as to~\cite{Mar18b} for a presentation closer to the present work. The model is parameterised by a sequence $\q = (q_k)_{k\ge0}$ of nonnegative real numbers which in a sense plays the role of the offspring distribution for Bienaym\'e--Galton--Watson random trees. More precisely, the $\q$-Boltzmann distribution on finite pointed maps with boundary length $2\varrho$ is defined as
\[\P^\varrho(M, v_\star) = \frac{1}{W^\varrho} \prod_{f \text{ inner face}} q_{\mathrm{deg}(f)/2},\]
where $W^\varrho$ is a normalising constant. For $\P^\varrho$ to be well-defined, this constant has to be finite, and this can be checked analytically on the sequence $\q$, see~\cite[Proposition~1]{MM07}, recast in~\cite[Proposition~9]{Mar18b}. 

By the bijection from Lemma~\ref{lem:bijection}, such a random pointed map is related to a labelled (loop)tree whose offspring distribution, say $\mu_\q$, is explicit~\cite{MM07,Mar18b}. One can then condition such a map to have ``size'' $S=N$, in the sense of having $N$ vertices, or $N$ edges, or $N$ faces, or even having $N$ faces whose degree belongs to a given set $A\in2\N$, and possibly other faces, but with a degree in $2\N \setminus A$. By Lemma~\ref{lem:bijection}, this corresponds to conditioning the looptree on its number of vertices, edges, or loops, i.e. conditioning the corresponding plane tree on its number of leaves, edges, or internal vertices respectively.
Such conditionings have been especially studied by Kortchemski~\cite{Kor12}, relying on the \L ukasiewicz path, as well as Rizzolo~\cite{Riz15} using a different approach.

Note that one can consider the same law $\P^\varrho$ on non-pointed maps, with a different normalising constant, and further distinguish a vertex uniformly at random. Since the number of vertices of the map is random ---~unless one condition on it~--- then sampling a pointed map induces a bias on the total number of vertices; however it is now classical that this bias vanishes in the large $N$ limit in many models and a pointed Boltzmann is close to a non pointed one in which one further distinguish a vertex uniformly at random, see~\cite[Proposition~6.4]{Mar19}, adapted from~\cite{BJM14, Abr16, BM17,Mar18b}.

Thanks the preceding remark, one can focus on pointed maps and thus rely on the bijection from Lemma~\ref{lem:bijection}.
Let $\degf_{n,1} \ge \degf_{n,2} \ge \dots \ge 0$ denote the ranked outdegrees of the associated conditioned random labelled forest, which, for the nonzero terms, are the cycle lengths of the looptree and the half-face degrees of the pointed map.
The key observation is that these $\degf_{n,i}$'s and $n$ itself (unless we condition on the number of edges) are random, but conditional on these numbers, the forests, looptrees, and pointed maps have the uniform distribution with these degrees. Therefore, using e.g.~Skorokhod's representation theorem, one can apply our previous general results conditionally given these degrees as soon as they satisfy the assumptions with convergences in distribution, where the $\theta_i$'s are a priori random. Note that the scaling factor $\sigma_n^2 = \sum_i \degf_{n,i} (\degf_{n,i}-1)$ is also random, however in several cases it behaves asymptotically deterministically, by a kind of law of large numbers.

\subsubsection{The Brownian regime}
According to~\cite[Theorem~8.1]{Kor12} when $\varrho=1$ and the law $\mu_\q$ is critical and has finite variance, the \L ukasiewicz path rescaled by some constant times $N^{-1/2}$ converges in distribution towards a Brownian excursion. 
Moreover in this case, as shown in~\cite[Section~6]{Mar19} the random scaling $\sigma_n^2$ divided by $N$ converges in probability to a constant. We deduce that our results apply, when $\sigma_n^2$ is replaced by this constant times $N$, and where $\theta_1 = 0$ so $\theta_0=1$. In particular in this case, the trees converge to the Brownian tree at a scaling of order $N^{1/2}$ by Corollary~\ref{cor:convergence_arbres_var_finie}, so do the looptrees by Corollary~\ref{cor:convergence_looptree_CRT}, the labels converge to the head of the Brownian snake at a scaling of order $N^{1/4}$ by Theorem~\ref{thm:convergence_labels_degres_prescrits_general}, and finally the maps then converge to the Brownian sphere at the same scaling by Theorem~\ref{thm:convergence_cartes_degres_prescrits}. This can be generalised to forests with $\varrho_n \sim \varrho N^{1/2}$ trees with $\varrho \in [0,\infty)$. 
More generally~\cite[Theorem~8.1]{Kor12} treats the case when $\mu_\q$ is critical and belongs to \emph{the domain of attraction of a Gaussian law}, which means that $(j^2 \sum_{k\ge j} \mu_\q(k))_{j \ge 1}$ is a \emph{slowly varying sequence}. The previous results extend, except tightness of trees, with all scaling factors multiplied by some slowly varying sequence.
We refer to~\cite[Section~6]{Mar19} for details.

\subsubsection{The stable regimes}
Theorem~8.1 in~\cite{Kor12} also treats the case when $\mu_\q$ is critical and belongs to the domain of attraction of a stable law with index $\alpha \in (1,2)$, i.e. when $(j^\alpha \sum_{k\ge j} \mu_\q(k))_{j \ge 1}$ is a slowly varying sequence. In this case the \L ukasiewicz path, rescaled by some other slowly varying sequence times $N^{-1/\alpha}$ converges in distribution towards the excursion of a stable L\'evy process with index $\alpha$ as defined in Section~\ref{ssec:Levy}. Les us mention that when the map is conditioned on having $N$ faces with degree in a set $A$, then it is assumed in~\cite{Kor12} that either $A$ or its complement is finite, however Th\'evenin~\cite{The20} lifted this restriction.
As opposed to the Brownian regime, the stable L\'evy excursion admits positive jumps, so now the $\theta_i$'s are positive for $i \ge 1$ and random, and satisfy $\sum_i \theta_i = \infty$ a.s., and finally $\theta_0=0$. By Proposition~\ref{prop:Levy_partie_continue_hauteur} and the remark after, these processes and their conditioned versions satisfy~\eqref{eq:PJ}, so we can apply our results from Section~\ref{sec:convergence_saut_pur}. In this case the stable looptrees are defined in~\cite{CK14}, whereas the limit of the label process is the ``continuous distance process'' in~\cite{LGM11}. Finally, as in the Brownian regime, one can replace the random scaling factor of labels and maps $\sigma_n^{1/2}$ by the deterministic one $N^{-1/(2\alpha)}$ times a slowly varying sequence, see again~\cite[Section~6]{Mar19} for details. We stress that this reference could only provide an abstract tightness argument whereas our present results now allow to recover (and extend to more general conditionings and boundaries) the results of~\cite{LGM11, Mar18a}.
Let us mention that as in these reference, the convergence of the maps holds only after extraction of a subsequence and proving the uniqueness of the limits in this case is under active investigation~\cite{CMR}.

\subsubsection{Condensation regimes}
When $\mu_\q$ has mean smaller than $1$, the situation is different and the \L ukasiewicz paths of the trees conditioned to have $N$ edges, now rescaled by a factor $N$, converge in distribution towards a linear slope by~\cite{Kor15}. This is also the case when $\mu_\q$ is critical and belongs to the domain of attraction of a Cauchy distribution by~\cite{KR19}.
In both cases we have $\theta_0 = 1$ and Corollary~\ref{cor:convergence_looptree_cercle} applies: the rescaled looptrees converge towards circles. Then random labels on them converge towards Brownian bridges and the associated maps converge to the Brownian CRT. We refer again to~\cite[Section~6]{Mar19} as well as~\cite{JS15} for the original result on subcritical random maps.

\subsubsection{Non stable L\'evy processes}

In the recent paper~\cite{KM21b} we consider maps conditioned to have $N$ edges \emph{and} an arbitrary number $K_N+1$ of vertices, and so $N-K_N+1$ faces by Euler's formula. 
Under appropriate assumptions on both $\q$ and $K_N$, the largest degree satisfies $\sigma_n^{-1} \degf_{n,1} \to 0$ in probability and we identify the asymptotic behaviour of $\sigma_n$, so the rescaled maps again converge in distribution to the Brownian sphere, see precisely~\cite[Theorem~5.2 \&~5.4]{KM21b}. Again our results here on the one hand show the convergence in this case of the underlying looptree to the Brownian CRT, and on the other hand they also provide the convergence of the labels and thus more precise results on the geometry of these maps in the presence of macroscopic degrees, when~\cite[Theorem~5.5]{KM21b}, based on~\cite{Mar19} does not. Let us point out that in this particular framework the limits of the \L ukasiewicz paths are simply excursions of stable L\'evy processes with a drift (positive or negative, depending on $K_N$); a condensation phenomenon can also occur if $K_N$ is large enough.
These natural models are the main motivation to study more these non stable L\'evy objects in~\cite{KM22}, which should also more generally arise as limits of size-conditioned Boltzmann maps sampled with a weight sequence which varies with the size.

\addcontentsline{toc}{section}{References}


\end{document}